\documentclass[11pt]{article}
\usepackage[margin=1in]{geometry}
\usepackage{amsmath,amssymb,amsthm,mathtools,array,booktabs,enumitem}
\usepackage[T1]{fontenc}
\usepackage{lmodern}
\usepackage{microtype}
\usepackage{needspace}
\usepackage{xcolor}
\definecolor{RevisionColor}{RGB}{150,0,150}
\definecolor{ReaderGuideColor}{RGB}{0,70,160}

\usepackage{authblk} % 新增这个包！
\usepackage[hidelinks]{hyperref}
\numberwithin{equation}{section}
\setlist[enumerate,1]{label=\textup{(\arabic*)},ref=\arabic*}
\hypersetup{
  pdftitle={Abelian composition factors of finite rational groups},
  pdfsubject={A proof of the prime bound for abelian composition factors of finite rational groups},
  pdfkeywords={Finite rational group, composition factor, Brauer character,
finite linear group}
}
\newtheorem{theorem}{Theorem}[section]
\newtheorem{lemma}[theorem]{Lemma}
\newtheorem{proposition}[theorem]{Proposition}
\newtheorem{corollary}[theorem]{Corollary}

\theoremstyle{definition}
\newtheorem{definition}[theorem]{Definition}
\theoremstyle{remark}
\newtheorem{remark}[theorem]{Remark}
\newcommand{\Q}{\mathbb Q}
\newcommand{\F}{\mathbb F}
\newcommand{\Fix}{\operatorname{Fix}}
\newcommand{\End}{\operatorname{End}}
\newcommand{\Hom}{\operatorname{Hom}}
\newcommand{\Aut}{\operatorname{Aut}}
\newcommand{\Out}{\operatorname{Out}}
\newcommand{\Sp}{\operatorname{Sp}}
\newcommand{\GL}{\operatorname{GL}}
\newcommand{\PGL}{\operatorname{PGL}}
\newcommand{\SL}{\operatorname{SL}}
\newcommand{\PSL}{\operatorname{PSL}}
\newcommand{\PSU}{\operatorname{PSU}}
\newcommand{\PSp}{\operatorname{PSp}}

\newcommand{\BrTr}{\operatorname{BrTr}}

\newcommand{\Irr}{\operatorname{Irr}}

\newcommand{\Sym}{\operatorname{Sym}}

\title{Abelian composition factors of finite rational groups}

\author{Junying Guo}
\affil{School of Mathematics and Statistics, Jiangxi Normal University, Nanchang, 330022, China}
\affil{\href{mailto:guojunying@jxnu.edu.cn}{guojunying@jxnu.edu.cn}}
\author{Yanjun Liu}
\affil{School of Mathematics and Statistics, Jiangxi Normal University, Nanchang, 330022, China}
\affil{\href{mailto:liuyanjun@pku.edu.cn}{liuyanjun@pku.edu.cn}}
\author{Huan Xiong}
\affil{Institute for Advanced Study in Mathematics, Harbin Institute of Technology, Heilongjiang 150001, China}
\affil{\href{mailto:huan.xiong@hit.edu.cn}{huan.xiong@hit.edu.cn}}

\date{}
\begin{document}
\maketitle

\begin{abstract}
Motivated by an observation concerning the $S_3$-conjecture,
we prove that every abelian composition factor of a finite
rational group has order $2$, $3$, or $5$.
Thompson previously restricted these orders to
$\{2,3,5,7,11\}$; we exclude $7$ and $11$, establishing
the bound that he proposed.
Together with the work of Feit and Seitz on nonabelian
composition factors, this completes the determination
of the possible composition factors of finite rational groups.
\end{abstract}

\Needspace{6\baselineskip}


\section{Introduction}

A finite group \(G\) is called \emph{rational} if
 every ordinary irreducible character of \(G\) is rational-valued.
 Equivalently, \(g^m\) is conjugate to \(g\) in \(G\) for every \(g\in G\) and every integer \(m\) satisfying \((m,o(g))=1\).
Using the classification of finite simple groups, Feit and Seitz \cite{FeitSeitz1989} showed that
 every nonabelian simple composition factor of a finite rational group
  is isomorphic to an alternating group $A_n$ with $n\geq 5$, or
 to one of the groups in \begin{equation}\label{eq:exceptional-list}
\mathcal E=\{\PSp_4(3),\Sp_6(2), {\rm O}_8^+(2),\PSL_3(4),\PSU_4(3)\}.
\end{equation}
For abelian composition factors, Gow \cite{Gow1976} proved that they have order $2$, $3$, or $5$
in the solvable case.
Heged\H us \cite{Hegedus2005} further showed that
the Sylow $5$-subgroup of a solvable rational group is normal
and elementary abelian, and gave a detailed description
of the structure of rational $\{2,5\}$-groups.
In general, Thompson \cite{Thompson2008} proved that
if $C_p$ occurs as a composition factor of a finite rational group, then $p\le 11$,
and suggested that the bound could be improved to $p\leq5$.
Motivated by an observation concerning the $S_3$-conjecture,
we establish this bound and complete the determination of the
possible composition factors of finite rational groups.

\begin{theorem}\label{thm:main}
Let $G$ be a finite rational group. If $G$ has an abelian
composition factor of prime order $p$, then $p\in\{2,3,5\}$.
Furthermore, the bound is sharp.
\end{theorem}

Each of the three primes in Theorem~\ref{thm:main} occurs. For instance,
the groups $C_2$ and $S_3$ account for $2$ and $3$, respectively,
while the rational Frobenius group $\F_5^2\rtimes Q_8$
accounts for $5$.

Our proof proceeds through a stronger result concerning linear
actions of rational groups. We say that a linear action of a
finite group $H$ on an $\F_p$-space $V$ has the
\emph{eigenvector property} if
\begin{equation}\label{eq:intro-eigenvector}
 \text{for every $0\ne v\in V$ and every $a\in\F_p^\times$,
 there exists $h\in H$ such that $hv=av$.}
\end{equation}

\begin{theorem}\label{thm:linear}
Suppose that a finite rational group acts linearly on a nonzero
finite-dimensional vector space over $\F_p$ with the
eigenvector property, where $p\leq 11$. Then $p\in\{2,3,5\}$.
\end{theorem}

For rational $2$-groups, Deb\'on \cite[Theorem~A]{Debon2026}
proved that a faithful action on a nonzero finite-dimensional $\F_p$-space
with the eigenvector property, where $p\geq5$, forces $p=5$, and determined the resulting
actions. Theorem~\ref{thm:linear} treats arbitrary finite rational
groups for characteristics $7$ and $11$.

To relate the two theorems, we use the least-counterexample
reduction through an elementary abelian normal subgroup, as in
Thompson's work \cite{Thompson2008}.
Suppose that
Theorem~\ref{thm:main} fails. By Thompson's theorem, it suffices
to consider $p\in\{7,11\}$. Fix such a prime $p$, and choose a
rational group $G$ of least order having $C_p$ as a composition
factor.

Let $V$ be a minimal normal subgroup of $G$. Since quotients
of rational groups are rational, the minimality of $G$ implies that
$G/V$ has no composition factor isomorphic to $C_p$.
Consequently, $C_p$ occurs as a composition factor of $V$, so
\[
 V\cong C_p^d
\]
for some $d\geq1$. We therefore identify $V$ with the additive
group of an $\F_p$-vector space of dimension $d$.

The conjugation action of $G$ on $V$ induces a faithful action
of
\[
 H=G/C_G(V).
\]
This action is irreducible because $V$ is minimal normal in $G$.
Moreover, $H$ is rational, being a quotient of $G$.
Rationality of $G$ also implies that $v$ and $av$ are conjugate
in $G$ for every $0\ne v\in V$ and every $a\in\F_p^\times$.
Thus the action of $H$ on $V$ has the eigenvector property.
Theorem~\ref{thm:linear} would therefore yield a contradiction,
so it implies Theorem~\ref{thm:main}.

%We record one further feature of this setup. Since $V$ is
%abelian, we have $V\leq C_G(V)$, and hence $H$ is also a quotient
%of $G/V$. In particular, $H$ has no composition factor
%isomorphic to $C_p$, although $p$ may still divide $|H|$.

The following sections develop the argument in this
minimal-counterexample setting. %At the end of the paper, we
%prove Theorem~\ref{thm:linear} in its full generality by applying
%the same arguments to a nonzero irreducible submodule and
%passing to the faithful image of the acting group.

\medskip
\noindent\textbf{Organization of the paper.}
Section~2 introduces admissible pairs through the scalar-fusion,
Brauer-rationality, and symplectic properties of the initial action.
Section~\ref{sec:quartics} constructs low-degree invariant polynomials
that obstruct certain module decompositions. Section~4 applies these
obstructions to abelian normal subgroups and reduces an admissible
pair to a $2$-terminal pair, preserving $\boldsymbol{(R^-)}$ when assumed.
Section~\ref{sec:generalized-fitting} determines the possible
structures of the Fitting subgroup of a $2$-terminal pair $(L,W)$.
Section~\ref{sec:extraspecial} shows that, in the extraspecial case,
this subgroup acts absolutely irreducibly on $W$, and excludes
this case under the additional condition $\boldsymbol{(R^-)}$,
which holds for the pairs arising in the proof of the main theorem.

In the remaining case, $E(L)\ne1$.
Sections~\ref{sec:layer} and~\ref{sec:local-actions}
decompose the module into component subspaces and construct the
associated almost simple local actions.
Section~\ref{sec:recognition} combines counting restrictions on block
actions with a realization of the local actions inside the original
rational action to determine the possible simple quotients of the
components.
Section~\ref{sec:component-structure} excludes the alternating and
exceptional possibilities.
Finally, Section~\ref{sec:terminal-contradiction} combines these
results to prove the main theorems.

\medskip
\noindent\textbf{Notation.}
Unless otherwise stated, we retain the above notation and
regard $V$ as a faithful irreducible $\F_p H$-module, written
additively. For $h\in H$, we denote its commuting $p$- and
$p'$-parts by $h_p$ and $h_{p'}$, respectively, so that
\[
 h=h_p h_{p'}=h_{p'}h_p.
\]
We shall also use the following elementary observation:
if $L\unlhd K\leq H$, then every $p'$-element of $K/L$
has a $p'$-lift in $K$. Indeed, the $p'$-part of any lift
is again a lift of the given element.

\section{Admissible pairs}
In this section, we establish several properties of the pair $(H,V)$,
labelled (SF), (CF), (R), and (I). These properties are then used to introduce the notion of an \textit{admissible pair}.
Under this definition, \((H,V)\) is automatically an admissible pair.

The first property, called \emph{scalar fusion} {\rm(SF)}, concerns
the fusion of each nonzero vector with all its nonzero scalar multiples.

\begin{lemma}\label{lem:scalar-fusion}
The $\F_p H$--module $V$ satisfies \\
 {\rm(SF)}: \ \ \ \ \ \ \ \ for $0\ne v\in V$
and $a\in\F_p^\times$, there is a $p'$-element $h\in H$
such that $hv=av$.
\end{lemma}
\begin{proof} Since $G$ is rational, $v$ and $av$ are conjugate in $G$.
Hence there is $t\in H$ such that $tv=av$.
The line $\F_p v$ is $t$-invariant, and the induced action of
$t$ on this line has order dividing $p-1$. Hence $t_p$ acts
trivially on $\F_p v$.
Thus,
\[av=tv=t_{p'}v,
\]
as desired.
\end{proof}

The second property is \emph{centralizer fusion (CF)},
meaning that scalar fusion on the fixed space of a \(p'\)-element can be realized within its centralizer.

\begin{lemma}\label{lem:centralizer-fusion}
The $\F_p H$--module $V$ satisfies\\
 {\rm(CF)}: \ \ for any $p'$-element $x\in H$, $w\in\Fix_V(x)$ and $a\in\F_p^\times$,
there is a $p'$-element $h\in C_H(x)$ such that $hw=aw$.
\end{lemma}
\begin{proof}
For $w=0$, take $h=1$. Otherwise let $\pi:G\to H$ be the
natural map and choose a $p'$-lift $y$ of $x$.
%The \(p\)-part of \(\widetilde x\) maps trivially because \(x\) is a \(p'\)-element;
%hence the commuting \(p'\)-part, say $y$,  of \(\widetilde x\) still maps to \(x\).
Since \(xw=w\), the elements \(w\) and \(y\) commute.
By the Chinese remainder theorem, we may choose an integer $m$ with
$m\equiv a\pmod p$ and $m\equiv1\pmod{o(y)}$.
Since $o(w)=p$ and $(p,o(y))=1$, we have $(m,o(yw))=1$ and
\[(yw)^m=y(aw).\]
Since $G$ is rational, there is  $k\in G$ with $k(yw)k^{-1}=y(aw)$.
Comparing the commuting $p$- and $p'$-parts of both sides yields
\[
kyk^{-1}=y,\qquad kwk^{-1}=aw.
\]
Hence $t=\pi(k)$ lies in $C_H(x)$ and sends $w$ to $aw$.
The proof of Lemma~\ref{lem:scalar-fusion} shows that $h=t_{p'}$ still
sends $w$ to $aw$. Clearly, $h\in C_H(x)$ (since $h$ is a power of $t$).
This proves the lemma.
\end{proof}

The third property concerns the Brauer character afforded by $V$.
For a $p'$-element $x\in H$, let
$\lambda_1,\ldots,\lambda_d\in\overline{\F}_p^\times$ be the
eigenvalues of the action of $x$ on $V$,  counted with multiplicity, and let
$\widehat\lambda_1,\ldots,\widehat\lambda_d$ be their respective
Teichm\"uller lifts to complex roots of unity. We call
\begin{equation}\label{eq:brauer-characteristic-polynomial}
P_x(T)=\prod_{i=1}^d(T-\widehat\lambda_i)
\end{equation}
the \emph{Brauer characteristic polynomial} of $x$, and write $\BrTr(x)=\sum_{i=1}^d \widehat{\lambda}_i$.
For every integer \(a\geq 1\), the \(a\)-th power sum of the roots of \(P_x(T)\) is
\[
\sum_{i=1}^d \widehat{\lambda}_i^{\,a}=\sum_{i=1}^d \widehat{\lambda_i^a}
=\BrTr(x^a),
\]
where the first equality follows from the compatibility of
Teichm\"uller lifts with taking powers.

\begin{lemma}\label{lem:rational-brauer}
Every irreducible Brauer character of $H$ in characteristic $p$
is integer-valued. In particular, $V$ satisfies \\
{\rm(R)}:\ \ \ \
for every $p'$-element $x\in H$, $P_x(T)\in\Q[T]$.
For each such $x$, this is equivalent to
$\BrTr(x^a)\in\mathbb Z$ for every integer $a\ge1$.
Galois-conjugate roots of $P_x(T)$ occur with equal multiplicity.
\end{lemma}

\begin{proof}
Let $n$ be the $p'$-part of $|H|$ and let $\varphi$ be an irreducible Brauer character of $H$ in characteristic $p$.
Then the values of $\varphi$ belong to $\mathbb{Q}_n=\Q(\zeta)$, where $\zeta$ is a primitive $n$th root of unity.
Recall that an automorphism $\sigma$ of $\mathbb{Q}_n$ has the form $\sigma(\zeta)=\zeta^b$ with $(b,n)=1$.

For every $p'$-element $h\in H$, we have
\[
\sigma(\varphi(h))=\varphi(h^b).
\]
Since $o(h)$ divides $n$, it follows that $(b,o(h))=1$.
By the rationality of $H$, $h^b$ is conjugate to $h$ in $H$.
Hence $\varphi(h)$ is fixed by every $\sigma$, and is rational.
Since $\varphi(h)$ is an algebraic integer, we deduce that $\varphi(h)\in\mathbb Z$.

The Brauer character of $V$ is a sum of irreducible Brauer
characters of $H$,  so it is integer-valued on $p'$-elements of $H$.
Since every power of $x$ is
again a $p'$-element, all power sums $s_a=\BrTr(x^a)$ are
integers. Write
\[
P_x(T)=T^d+c_1T^{d-1}+\cdots+c_d.
\]
Newton's identities yield, for $1\le j\le d$,
\[
j c_j+s_j+c_1s_{j-1}+\cdots+c_{j-1}s_1=0,
\]
so $c_j\in\Q$ by induction. Conversely, if $P_x(T)\in\Q[T]$,
its multiset of roots is Galois-invariant, and every power
sum is rational. Each is also an algebraic integer, hence
an integer. Finally, Galois invariance of the multiset of roots of  $P_x(T)$ leads to the
assertion about multiplicities.
\end{proof}

The final property concerns the absolute irreducibility of $V$
and the existence of an invariant symplectic form.
To show that the invariant form is alternating, we use an
eigenvector-property argument attributed to Gow in
\cite[Remark~2.1]{Debon2026}.

\begin{lemma}
\label{lem:three}\label{lem:symplectic}
For $p\in \{7,11\}$, the $\F_p H$-module $V$ satisfies \\
{\rm(I)}: it is absolutely
irreducible and carries an $H$-invariant nondegenerate
alternating bilinear form over $\F_p$.
\end{lemma}
\begin{proof}
Set $\overline V=V\otimes_{\F_p}\overline{\F}_p$, and let $W$
be an irreducible constituent with Brauer character $\varphi$.
For an irreducible module over a finite field, scalar extension
to the algebraic closure is a direct sum of pairwise
nonisomorphic irreducible modules forming a single orbit
under field Frobenius, each with multiplicity one;
see, for instance, \cite[Theorem 9.21]{I}.

A cyclotomic automorphism acting on the \(p'\)-roots of unity
has the form \(\zeta\mapsto\zeta^m\), with
\(m\) prime to the \(p'\)-part of \(|H|\).
By Lemma~\ref{lem:rational-brauer}, $\varphi$ is integer-valued.
So the cyclotomic automorphism \(\zeta\mapsto\zeta^p\) (corresponding to the field Frobenius)
fixes \(\varphi\), and hence fixes the isomorphism class of $W$.  %, since irreducible Brauer characters determine irreducible modules.
It follows that the Frobenius orbit has length one. Thus $\overline V=W$
is irreducible and $V$ is absolutely irreducible.

Complex conjugation is among the cyclotomic automorphisms
and fixes $\varphi$, so $\overline{V}$ and $\overline{V^*}$ have the same Brauer character.
  Hence $\overline{V}$ and $\overline{V^*}$ are isomorphic.
Since
\[
\Hom_{\F_pH}(V,V^*)\otimes_{\F_p}\overline{\F}_p
\cong\Hom_{\overline{\F}_pH}(\overline V,\overline V^*),
\]
 there is a nonzero
$\F_p H$-map $f:V\to V^*$ over $\F_p$. Therefore $f$ is an isomorphism by the
irreducibility of $V$. %In particular, $V$ is self-dual over $\F_p$.

Define
\[
B(v,w)=f(v)(w)
\qquad (v,w\in V).
\]
Then $B$ is a nonzero bilinear form on $V$.
Moreover, $B$ is $H$-invariant.
Indeed, for $h\in H$ we have $f(hv)=h\cdot f(v)$, and therefore
\[
B(hv,hw)
=f(hv)(hw)
=(h\cdot f(v))(hw)
=f(v)(w)
=B(v,w),
\]
where the action of $H$ on $V^*$ is given by
$(h\cdot \lambda)(w)=\lambda(h^{-1}w)$.
The radical of $B$ is an $H$-submodule of $V$. Since $B\neq 0$ and $V$ is
irreducible, the radical is zero, and hence $B$ is nondegenerate.

We next show that the space of $H$-invariant bilinear forms on $V$ is
one-dimensional. To an $H$-invariant bilinear form $C$ on $V$ associate
\[
f_C:V\longrightarrow V^*,\qquad
f_C(v)(w)=C(v,w).
\]
The $H$-invariance of $C$ is equivalent to $f_C$ being an $H$-module
homomorphism. Thus
\[
\{\text{$H$-invariant bilinear forms on $V$}\}
\cong \Hom_H(V,V^*).
\]
Since $V\cong V^*$ as $\F_p H$-modules and $V$ is absolutely irreducible, Schur's lemma says
$\End_H(V)=\F_p$. Therefore the space of $H$-invariant bilinear forms on $V$ is
one-dimensional.

Now define the transpose form $B^{\mathrm t}(v,w)=B(w,v)$,
which is also $H$-invariant. Hence there exists $c\in\F_p^\times$ such that
\[B^{\mathrm t}=cB.\] Transposing twice yields $c^2=1$.
 Since $p$ is odd, $c=\pm1$. Thus $B$ is either symmetric or
skew-symmetric. In odd characteristic, a skew-symmetric bilinear form is
alternating, so it only remains to exclude the symmetric case.

Suppose that $B$ is symmetric. Since $B$ is nonzero and $p$ is odd, there
exists $v\in V$ such that \[B(v,v)\neq 0.\]
Indeed, if $B(v,v)=0$ for every $v\in V$, then polarization yields
\[
2B(v,w)
=
B(v+w,v+w)-B(v,v)-B(w,w)=0
\]
for all $v,w\in V$, and hence $B=0$, a contradiction.

By (SF), applied with $a=2$, there exists a $p'$-element $h\in H$ such that $hv=2v$.
Since $B$ is $H$-invariant,
\[
B(v,v)=B(hv,hv)=B(2v,2v)=4B(v,v).
\]
Thus $3B(v,v)=0$,
which is impossible since $p\in \{7,11\}$ and $B(v,v)\neq0$. Therefore $B$
is not symmetric, meaning that $B$ is alternating.
Thus $V$ admits a nondegenerate $H$-invariant alternating bilinear form, and
so $V$ is symplectic.
\end{proof}

We call a faithful action satisfying {\rm(SF)}, {\rm(R)},
and {\rm(I)} \emph{admissible}, and call the corresponding
pair $(H,V)$ an \emph{admissible pair}.
Condition {\rm(CF)} is not part of this definition and will
be used only when explicitly assumed.
For the local component actions constructed later, we shall
verify {\rm(SF)} directly and prove that the restriction
to the component has integer-valued Brauer character.

\begin{remark}
\label{rem:faithful-image}
Whenever a subgroup $J$ stabilizes a subspace $U$, the group
\emph{induced by $J$ on $U$} means the faithful image
\[
                     J_U=J/C_J(U)\leq\GL(U).
\]

All local uses of the admissibility conditions refer to this induced
group. For an arbitrary subgroup $J\leq H$ stabilizing a subspace $U$ of $V$,
$(J_U,U)$ is not necessarily admissible.
Whenever admissibility of an induced action is needed, it will be
proved below.
\end{remark}

%\begin{remark}
%\label{rem:set-up}
%For our purposes, we shall need the central involution $z=-I$, and
% henceforth replace $H$ by $\langle H,-I\rangle$ if necessary. Specifically,
% $\langle H,-I\rangle$ coincides with $H$ if $-I\in H$, and is isomorphic to $H\times C_2$ otherwise. Therefore, this group is rational.
%
%Also, $-I$ preserves any bilinear form, so that
%\[
%B(-Iv,-Iw)=B(v,w).
%\]
%Hence $B$ is invariant under $\langle H,-I\rangle$, and the symplectic structure is preserved. Since any $K$-invariant subspace is automatically $H$-invariant
%and $\overline V=V\otimes_{\mathbb F_p}\overline{\mathbb F}_p$ has no non-trivial $H$-invariant subspaces
% over $\overline{\mathbb F}_p$, absolute irreducibility is retained upon passing to $K=\langle H,-I\rangle$.
%
%Property (SF) is inherited from $H$, as the required $p'$-elements already lie in $H$. Rationality yields property (R) as before.
%Thus we may always assume that $z=-I\in H$.
%This reduction applies to all admissible pairs considered in the paper.
%\end{remark}

\section{Low-degree invariant polynomials}\label{sec:quartics}

Throughout, polynomials are elements of polynomial rings, rather than
merely polynomial functions on finite vector spaces.
For a finite-dimensional space $W$ of dimension $n$, the coordinate linear functions $X_1,\ldots, X_n$ form
a basis of $W^*$. Products of $r$ coordinate functions span $\Sym^r(W^*)$, the
space of formal homogeneous polynomials of degree $r$. In coordinates,
every monomial has total degree $r$. For instance, for $W=\F_p^2$,
 homogeneous quartics on \(W\), which are polynomials
 in the coordinates $X_1,X_2$ with all monomials having total degree \(4\),
are linear combinations of $X_1^4,X_1^3X_2,X_1^2X_2^2,X_1X_2^3,X_2^4$.
The word \emph{formal} matters over finite fields, since $x^p-x$ is a nonzero
polynomial inducing the zero function on $\F_p$.

We shall use invariant polynomials to obtain
restrictions on symplectic actions satisfying scalar fusion.
%We first show that scalar fusion rules out nonzero homogeneous
%invariants of certain low degrees. We then construct quartic
%invariants from paired summands and tensor decompositions,
%thereby excluding these configurations. These obstructions
%prepare the way for the refinement arguments in the next section.
%
%The purpose of this section is to derive structural restrictions
%from scalar fusion by means of invariant polynomials.
The basic observation is that an invariant homogeneous polynomial
of degree $r$ satisfies $F(gw)=F(w)$ and $F(aw)=a^rF(w)$.
If scalar fusion provides an element $g$ with $gw=aw$, then
$(a^r-1)F(w)=0$. For certain low degrees, this forces the polynomial
to vanish. We use this observation to rule out certain
decompositions of the module by constructing nonzero invariant quartics from paired summands
and tensor products whose factors carry forms of opposite types.
These
restrictions will be used in the subsequent refinement
and structural arguments.

\begin{lemma}
\label{lem:quartic-scalar-principle}
Let $p$ be an odd prime, $W=\F_p^n$ and $L\leq\operatorname{GL}(W)$.
Suppose that for every $0\ne w\in W$ there is $g\in L$ with $gw=2w$.
If $1\leq r<p$ and $2^r\not\equiv1\pmod p$, then $L$  admits no nonzero invariant
homogeneous polynomial of degree $r$ on $W$.
\end{lemma}

\begin{proof} We first prove, by induction on $n$, that a polynomial
$F\in\F_p[X_1,\ldots,X_n]$ with $\deg_{X_i}F<p$ for every $i$
is zero if it vanishes at every point of $\F_p^n$.
Assume that $F$ vanishes at every such point.
For $n=1$, this is the well known fact that a
nonzero polynomial of degree less than $p$ cannot have $p$ distinct roots. For $n>1$, write
\[
F=\sum_{j=0}^{p-1}
F_j(X_1,\ldots,X_{n-1})X_n^j.
\]
Fix $(a_1,\ldots,a_{n-1})\in\mathbb F_p^{\,n-1}$. Then
$F(a_1,\ldots,a_{n-1},X_n)$
is a polynomial in $X_n$ of degree less than $p$ which vanishes at every
element of $\mathbb F_p$. Hence it is the zero polynomial, and so
$F_j(a_1,\ldots,a_{n-1})=0$
for every $j$. Each $F_j$ has degree less than $p$ in every remaining
variable, so the induction hypothesis gives $F_j=0$ and hence $F=0$.
In particular, a nonzero homogeneous polynomial of positive degree
$r<p$ has a nonzero value somewhere in $\F_p^n$. Since $F(0)=0$,
such a value occurs at a nonzero vector.

Now suppose that $F$ is a nonzero $L$-invariant homogeneous polynomial of
degree $r<p$ on $W$. By the first part, choose
$w\in W$ with $F(w)\ne0$. By assumption, there exists an element $g\in L$ such that $gw=2w$.
Since $F$ is $L$-invariant and homogeneous of degree $r$, we obtain
\[
F(w)=F(gw)=F(2w)=2^rF(w).
\]
Since $F(w)\ne0$, it follows that
$2^r\equiv1\pmod p$, a contradiction.
\end{proof}

\begin{remark}
For every prime $p\geq7$, scalar fusion excludes invariant quartics, since
$4<p$ and $2^4-1=15\ne0$ in $\mathbb F_p$.
The proof of Lemma~\ref{lem:quartic-scalar-principle} uses only
fusion with the scalar $2$; it does not require irreducibility
or the other admissibility conditions.
\end{remark}

Useful decompositions of a module are often visible only
after extending the field. We recall the standard fact
that taking invariants commutes with extension of the
ground field. Applied to homogeneous polynomials, this
allows us to construct invariants over an algebraic closure
and descend their existence to the original field.

\begin{lemma}
\label{lem:quartic-descent}
Let $k\subseteq K$ be a field extension, let $W$ be a finite-dimensional $k$-vector space, and let $L\leq\GL_k(W)$ be a
finite group. Write $W_K:=K\otimes_k W$, with the induced action of $L$ given
by
$g(a\otimes w)=a\otimes gw$, where $g\in L,\ a\in K$ and $w\in W$.
Then, for every integer $r\geq0$, the natural base-change map induces
an isomorphism of $K$-vector spaces
\[
\left(\operatorname{Sym}_k^r(W^*)^L\right)\otimes_k K
\cong
\operatorname{Sym}_K^r\bigl((W_K)^*\bigr)^L.
\]
Consequently,
\[
\operatorname{Sym}_K^r\bigl((W_K)^*\bigr)^L\neq0
\quad\Longleftrightarrow\quad
\operatorname{Sym}_k^r(W^*)^L\neq0.
\]
Equivalently, a nonzero $L$-invariant homogeneous polynomial of
degree $r$ exists on $W_K$ over $K$ if and only if one exists on
$W$ over $k$.
\end{lemma}

\begin{proof}
Set $M=\operatorname{Sym}^r(W^*)$, with the action
$(g\cdot F)(w)=F(g^{-1}w)$.  Thus $F$ is $L$-invariant if and only if
$g\cdot F-F=0$ for every $g\in L$.
Define
\[
 \Phi:M\longrightarrow\bigoplus_{g\in L}M,
 \qquad F\longmapsto\bigl((g-1)F\bigr)_{g\in L}.
\]
Then $M^L=\ker\Phi$. Write $M_K=M\otimes_k K$
and let
\[
\Phi_K=\Phi\otimes_k\operatorname{id}_K:
M_K\longrightarrow\bigoplus_{g\in L}M_K
\]
be the map obtained by extending scalars from $\Phi$.
In bases over $k$, the matrix representing $\Phi$ has entries in $k$. Its rank is unchanged after extending scalars to $K$,
since the rank of a matrix is the largest size of a nonzero minor, and a minor over $k$
is nonzero over $k$ exactly when it is nonzero over $K$.
The natural inclusion
\[
 (\ker\Phi)\otimes_k K\ \subseteq\ \ker(\Phi_K)
\]
is therefore an equality, by comparing the $K$-dimensions of the two sides.
Note that extension of coefficients gives a natural isomorphism
$M\otimes_k K
\cong
\operatorname{Sym}_K^r\bigl((W_K)^*\bigr)$.
This isomorphism is compatible with the $L$-actions, since extension of
coefficients commutes with substitution by the matrices of
elements of $L$. Under this identification, we have
\[
\ker\Phi_K
=
\operatorname{Sym}_K^r\bigl((W_K)^*\bigr)^L, \qquad
\bigl(\operatorname{Sym}_k^r(W^*)^L\bigr)\otimes_k K
\cong
\operatorname{Sym}_K^r\bigl((W_K)^*\bigr)^L.
\]
Finally, the last assertion follows from the fact that a
finite-dimensional $k$-space is nonzero if and only if its
scalar extension to $K$ is nonzero.
\end{proof}

\begin{remark}
The proof of Lemma \ref{lem:quartic-descent} uses only linear equations, so it applies even when
$\operatorname{char}(k)$ divides $|L|$ and requires no finiteness
or Galois assumption on $K/k$.
%Every invariant over $K$ is a $K$-linear combination of invariants
%defined over $k$.% Thus the lemma descends the existence of a
%nonzero invariant, although an individual invariant over $K$
%need not have coefficients in $k$.
Every $L$-invariant homogeneous polynomial of degree $r$ over
$K$ is a $K$-linear combination of $L$-invariant homogeneous
polynomials of degree $r$ defined over $k$.
Thus the lemma guarantees the existence of a nonzero invariant homogeneous polynomial
of the same degree over $k$, but does not assert that a given
invariant over $K$ itself has coefficients in $k$.
\end{remark}

The following lemma allows us to combine compatible homogeneous
polynomials on a family of summands permuted by the group
into an invariant polynomial on the whole module.
We shall use this construction to obtain quartic
obstructions from decompositions of the module.

\begin{lemma} \label{lem:quartic-packet-sum}
Let $L$ act linearly on a finite-dimensional vector space $W$
over a field $K$. Suppose that
\[
W=W_0\oplus\bigoplus_{i\in I}P_i,
\]
where $|I|<\infty$, each $P_i\neq 0$, $W_0$ is $L$-invariant, and $L$ permutes the subspaces $P_i$.
Write $gi$ for the index determined by $gP_i=P_{gi}$.
For each $i\in I$, let $Q_i$ be a homogeneous polynomial
of degree $r\geq1$ on $P_i$. Assume that
$Q_{gi}\circ(g|_{P_i})=Q_i$, where $g\in L$ and $i\in I$,
as an identity of polynomials.
Then the polynomial on $W$ defined by
\[
Q\left(w_0+\sum_{i\in I}v_i\right)
=\sum_{i\in I}Q_i(v_i),
\qquad
w_0\in W_0,\quad v_i\in P_i,
\]
is $L$-invariant and homogeneous of degree $r$.
Moreover, if some $Q_j$ is nonzero, then $Q$ is nonzero.
\end{lemma}

\begin{proof}
The direct-sum decomposition provides linear projections $W\to P_i$,
so the displayed expression defines a homogeneous polynomial of
degree $r$. For $g\in L$, the component $v_i$ is carried into $P_{gi}$,
and consequently
\[
 Q\left(gw_0+\sum_i gv_i\right)
 =\sum_iQ_{gi}(gv_i)=\sum_iQ_i(v_i).
\]
Thus $Q$ is $L$-invariant.
Finally, if  $Q_j\neq0$ for some $j\in I$, then $Q|_{P_j}=Q_j\neq0$, hence $Q\neq 0$.
\end{proof}

We now apply the preceding construction to a decomposition
into paired summands. The bilinear pairing on each pair
defines a quadratic polynomial that may change sign when
the two summands are interchanged. Squaring removes this
sign, and summing over the pairs produces a nonzero
invariant quartic polynomial. Lemma~\ref{lem:quartic-scalar-principle}, applied with $r=4$,
will therefore exclude such decompositions when $p\geq7$
and scalar fusion holds.

\begin{lemma}
\label{lem:quartic-dual-pairs}
Let $K$ be a field of characteristic different from \(2\), let $(W,B)$ be a
symplectic space, and let $L\leq\operatorname{Sp}(W,B)$. Suppose that
\[
 W=W_0\oplus\bigoplus_{i\in I}(U_i\oplus V_i),
\]
where $1\leq |I|<\infty $, $W_0$ is $L$-invariant, each $U_i,V_i$ is
nonzero, and the pairing $B:U_i\times V_i\to K$ is nondegenerate.
Suppose also that $L$ permutes the unordered pairs $\{U_i,V_i\}$.
Then
\begin{equation} \label{equ:quadratic}
Q\left(w_0+\sum_{i\in I}(u_i+v_i)\right)
=\sum_{i\in I}B(u_i,v_i)^2,
\qquad
w_0\in W_0,\quad u_i\in U_i,\quad v_i\in V_i,
\end{equation}
defines a nonzero $L$-invariant homogeneous polynomial
of degree 4 on $W$.
\end{lemma}

\begin{proof} On $P_i:=U_i\oplus V_i$, define
\[
 q_i(u+v)=B(u,v),\qquad Q_i(u+v)=B(u,v)^2,
\]  with $u\in U_i$ and $v\in V_i$.
The bilinearity of $B$ shows that $q_i$ is homogeneous of degree 2 and $Q_i$
is homogeneous of degree 4.

Suppose that $g\in L$ carries $\{U_i,V_i\}$ to $\{U_j,V_j\}$.
If $gU_i=U_j$ and $gV_i=V_j$, then
\[
 q_j(g(u+v))=B(gu,gv)=B(u,v)=q_i(u+v).
\]
If $gU_i=V_j$ and $gV_i=U_j$, then the $U_j$-component is $gv$ and
the $V_j$-component is $gu$. Since $B$ is alternating,
\[
 q_j(g(u+v))=B(gv,gu)=B(v,u)=-B(u,v).
\]
In either case, $Q_j(g(u+v))=Q_i(u+v)$. These calculations are
polynomial identities, since they use only bilinearity and form
preservation.

Choose $0\ne u\in U_i$. The nondegeneracy of $B$ shows that there is some $v\in V_i$ with
$B(u,v)\ne0$; rescale $v$ so that $B(u,v)=1$.
Then $Q_i(u+v)=1$, so $Q_i$ is nonzero.
Since $L$ permutes the summands $P_i$ and the polynomials $Q_i$
satisfy the required compatibility condition,
Lemma~\ref{lem:quartic-packet-sum} shows that (\ref{equ:quadratic}) is a nonzero
$L$-invariant homogeneous polynomial of degree 4 on $W$.
\end{proof}

\begin{remark} Transitivity on the pairs in Lemma \ref{lem:quartic-dual-pairs}  is unnecessary. If
$V_i$ is identified with $U_i^*$ via $v\mapsto B(-,v)$, the local
formula is precisely $Q_i(u+f)=f(u)^2$.
\end{remark}

The second construction concerns tensor products whose factors
carry nondegenerate bilinear forms of opposite types, one
symmetric and the other alternating. It produces a nonzero
homogeneous polynomial of degree $4$ whose transformation rule
ensures invariance under tensor products of similitudes with
reciprocal multipliers. Such transformations arise naturally
from the action on an isotypic component.
Together with Lemma~\ref{lem:quartic-scalar-principle},
this construction gives another restriction on symplectic
actions satisfying scalar fusion when $p\geq7$.

\begin{lemma}
\label{lem:quartic-tensor}
Let $K$ be a field of characteristic different from \(2\), and let
$X$ and $Y$ be finite-dimensional $K$-vector spaces with
$\dim X,\dim Y\geq2$.
Suppose that $X$ admits a nondegenerate symmetric bilinear form $b$
and that $Y$ admits a nondegenerate alternating bilinear form $c$.

For $v=\sum_i x_i\otimes y_i\in X\otimes_K Y$, let
$A_v:Y\to X$ be the linear map defined by
\[
 A_v(z)=\sum_i c(y_i,z)x_i.
\]
Let $A_v^\dagger:X\to Y$ be the unique linear map satisfying
\[
 c(z,A_v^\dagger x)=b(A_vz,x)
 \qquad\text{for all }z\in Y,\ x\in X.
\]
Then $Q(v)=\operatorname{tr}\bigl((A_v^\dagger A_v)^2\bigr)$
defines a nonzero homogeneous polynomial of degree 4 on
$X\otimes_K Y$.

Moreover, suppose that $g\in\operatorname{GL}(X)$ and
$h\in\operatorname{GL}(Y)$ satisfy
\[
 b(gx,gx')=\alpha b(x,x'),
 \qquad
 c(hy,hy')=\beta c(y,y')
\]
for all $x,x'\in X$ and $y,y'\in Y$, where
$\alpha,\beta\in K^\times$. Then
\[
 Q((g\otimes h)v)=(\alpha\beta)^2Q(v)
 \qquad\text{for all }v\in X\otimes_K Y.
\]
In particular, $Q$ is invariant under every such transformation
$g\otimes h$ for which $\alpha\beta=1$.
\end{lemma}

\begin{proof}
We first check that the construction is well defined.
The map
\[
 X\times Y\longrightarrow \operatorname{Hom}_K(Y,X),
 \qquad
 (x,y)\longmapsto \bigl(z\longmapsto c(y,z)x\bigr)
\]
is bilinear, and so induces a linear map
\[
 X\otimes_K Y\longrightarrow \operatorname{Hom}_K(Y,X),
 \qquad v\longmapsto A_v.
\]
Hence $A_v$ does not depend on the expression chosen for $v$ as a
sum of pure tensors, i.e., $A_v$ is well defined.
The nondegeneracy of  $c$ guarantees that $A_v^\dagger$ is well defined.
Moreover, both $A_v$ and $A_v^\dagger$ depend linearly on $v$.
After choosing bases of $X$ and $Y$, their matrix entries are
therefore linear polynomials in the coordinates of $v$.
It follows that $Q(v)=\operatorname{tr}\bigl((A_v^\dagger A_v)^2\bigr)$
is a homogeneous polynomial of degree 4, provided that it
is not the zero polynomial. We will show below that \(Q\) is not identically zero.

We next establish the transformation rule. Let
$A=A_v$ and $A'=A_{(g\otimes h)v}$.
The similitude identity for $h$ gives
$ c(hy,z)=\beta c(y,h^{-1}z)$.
Hence, writing $v=\sum_i x_i\otimes y_i$, we obtain
\[
 \begin{aligned}
 A'z
 &=\sum_i c(hy_i,z)gx_i\\
 &=\beta g\left(\sum_i c(y_i,h^{-1}z)x_i\right)
 =\beta gAh^{-1}z.
 \end{aligned}
\]
Thus $A'=\beta gAh^{-1}$.
For every $z\in Y$ and $x\in X$, we now have
\[
 \begin{aligned}
 b(A'z,x)
 &=\beta b(gAh^{-1}z,x)\\
 &=\alpha\beta b(Ah^{-1}z,g^{-1}x)\\
 &=\alpha\beta c(h^{-1}z,A^\dagger g^{-1}x)\\
 &=\alpha c(z,hA^\dagger g^{-1}x).
 \end{aligned}
\]
By uniqueness of the adjoint, we obtain $A'^\dagger=\alpha hA^\dagger g^{-1}$,
so $A'^\dagger A'
 =\alpha\beta\,h(A^\dagger A)h^{-1}$.
Therefore,  \[
 (A'^\dagger A')^2
 =(\alpha\beta)^2h(A^\dagger A)^2h^{-1}.
\]
Taking traces yields
\[
 Q((g\otimes h)v)=(\alpha\beta)^2Q(v).
\]
In particular, $Q$ is preserved whenever $\alpha\beta=1$.

It remains to prove that $Q$ is nonzero. We construct a vector
at which it has a nonzero value.
We may choose $e_1\in X$ with $b(e_1,e_1)\ne0$. This is possible
since if $b(x,x)=0$ for every $x$, then
\[
 2b(x,x')
 =b(x+x',x+x')-b(x,x)-b(x',x')=0
\]
for all $x,x'$. As $\operatorname{char}K\ne2$, this means
$b=0$, contradicting the nondegeneracy of $b$.

The orthogonal complement of $Ke_1$ is nondegenerate and has
positive dimension, since $\dim X\ge2$. Applying the same
argument in this complement, choose $e_2\in X$ such that
\[
 b(e_1,e_2)=0,\qquad
 a:=b(e_1,e_1)\ne0,\qquad
 d:=b(e_2,e_2)\ne0.
\]
By the nondegeneracy of $c$, we may also choose $f_1,f_2\in Y$ with $c(f_1,f_2)=1$.
Then  $c(f_2,f_1)=-1$. Set \[
 v_0=e_1\otimes f_1+e_2\otimes f_2,
 \qquad A=A_{v_0}.
\]
Since $c$ is alternating, we have $Af_1=-e_2$ and $Af_2=e_1$.
For arbitrary $z\in Y$, the orthogonality of $e_1,e_2$ yields
\[
 \begin{aligned}
 b(Az,e_1)&=a\,c(f_1,z)=c(z,-af_1), \qquad
 b(Az,e_2)&=d\,c(f_2,z)=c(z,-df_2).
 \end{aligned}
\]
Thus $
 A^\dagger e_1=-af_1$  and $A^\dagger e_2=-df_2$.
Denoting $T=A^\dagger A$, we obtain
\[Tf_1=df_2,\qquad Tf_2=-af_1, \qquad  T^2f_1=-adf_1,\qquad T^2f_2=-adf_2.
\]
The plane $F=\langle f_1,f_2\rangle_K$ is nondegenerate for $c$,
so
$Y=F\oplus F^\perp$.
If $z\in F^\perp$, then both $c(f_1,z)$ and $c(f_2,z)$ vanish, and therefore
$Az=0$. Thus $T$ vanishes on $F^\perp$, while $T^2$ acts as
$-ad$ times the identity on $F$. It follows that
\[
 Q(v_0)=\operatorname{tr}(T^2)=-2ad\ne0.
\]
This proves that $Q$ is nonzero.
\end{proof}

\begin{remark}
\label{rem:quartic-tensor-packets}
The construction in Lemma~\ref{lem:quartic-tensor} extends to
permuted tensor summands as follows. If the form types are
presented in the opposite order, interchange the factors.
Suppose that
\[
 W=W_0\oplus\bigoplus_{i\in I}(X_i\otimes_K Y_i),
\]
where $I$ is finite and nonempty, the group preserves $W_0$
and permutes the tensor summands, and each pair $(X_i,Y_i)$
satisfies the hypotheses of Lemma~\ref{lem:quartic-tensor}.
Choose nondegenerate symmetric forms $b_i$ on $X_i$ and
nondegenerate alternating forms $c_i$ on $Y_i$, and denote
the corresponding quartics by $Q_i$.

Assume that, whenever $\ell\in L$ sends
$X_i\otimes_K Y_i$ to $X_j\otimes_K Y_j$, its restriction
to $X_i\otimes_K Y_i$ has the form $g\otimes h$,
where $g:X_i\to X_j$ and $h:Y_i\to Y_j$ are linear
isomorphisms satisfying
\[
 b_j(gx,gx')=\alpha b_i(x,x'),\qquad
 c_j(hy,hy')=\beta c_i(y,y'),\qquad
 \alpha\beta=1
\]
for some $\alpha,\beta\in K^\times$.
The same adjoint calculation as in
Lemma~\ref{lem:quartic-tensor}, using the source and target
forms, yields
\[
 Q_j((g\otimes h)v)=Q_i(v).
\]
Thus Lemma~\ref{lem:quartic-packet-sum} shows that
\[
 Q\left(w_0+\sum_{i\in I}v_i\right)
   =\sum_{i\in I}Q_i(v_i),
 \qquad w_0\in W_0,\quad v_i\in X_i\otimes_K Y_i,
\]
is a nonzero invariant quartic on the whole space.
This compatibility is an explicit hypothesis; an abstract
vector-space tensor decomposition alone does not imply it.
Finally, the nonvanishing proof in
Lemma~\ref{lem:quartic-tensor} does not require the symmetric
form to contain a hyperbolic plane.
\end{remark}

The next lemma allows us to apply the polynomial obstruction
on the fixed space of a $p'$-element. This fixed space inherits
a nondegenerate symplectic form, and condition {\rm(CF)}
ensures that the action of the centralizer satisfies {\rm(SF)}.

\begin{lemma}\label{lem:localization}
Let $W$ be a symplectic space over $\F_p$, where $p$ is odd,
and let $J\leq\Sp(W)$ satisfy {\rm(CF)}.
Let $x\in J$ be a $p'$-element, and write $W_1=\Fix_W(x)$.
Then the following statements hold:
\begin{enumerate}
  \item The restriction of the symplectic form to $W_1$ is
nondegenerate.
  \item  $C_J(x)$ preserves $W_1$.
  \item The action of $C_J(x)$ on $W_1$ satisfies {\rm(SF)}
and admits no nonzero invariant homogeneous polynomial
of degree $1\leq r<p$ with $2^r\not\equiv1\pmod p$.
\end{enumerate}
\end{lemma}

\begin{proof} For (1), let $B$ be the $J$-invariant nondegenerate
alternating form on $W$. Write $\overline{W}=\overline{\F}_p\otimes_{\F_p}W$, and denote the scalar extension of $B$ by $\overline{B}$.
Since $x$ has order prime to $p$,
the action of $x$ on $\overline{W}$ is diagonalizable, and we may write
\[
    \overline{W}=\bigoplus_{\lambda}\overline{W}_{\lambda},
\]
where $\overline{W}_{\lambda}$ is the $\lambda$-eigenspace
of $x$. In particular, $\overline{W}_{1}
    =\overline{\F}_p\otimes_{\F_p}W_1$.

For $u\in\overline{W}_{\lambda}$ and
$v\in\overline{W}_{\mu}$,  the $J$-invariance of $B$ yields
\[
    \overline{B}(u,v)
    =\overline{B}(xu,xv)
    =\lambda\mu\,\overline{B}(u,v).
\]
Hence $\overline{B}(\overline{W}_{\lambda},
\overline{W}_{\mu})=0$ unless $\lambda\mu=1$.
In particular, $\overline{W}_1$ is orthogonal to every
$\overline{W}_{\lambda}$ with $\lambda\neq1$.
If $v\in\overline{W}_1$ is orthogonal to
$\overline{W}_1$, it is therefore orthogonal to all of
$\overline{W}$.
Since $\overline{B}$ is nondegenerate, we obtain
$v=0$. Therefore the restriction of $\overline{B}$ to
$\overline{W}_1$ is nondegenerate, and thus
the restriction of $B$ to $W_1$ is nondegenerate.

For (2), if $c\in C_J(x)$ and $v\in W_1$, then
\[
    x(cv)=c(xv)=cv.
\]
Thus $cv\in W_1$, and so $C_J(x)$ preserves $W_1$.

For (3), let $0\neq v\in W_1$ and
$a\in\F_p^\times$. By {\rm(CF)}, there exists a
$p'$-element $c\in C_J(x)$ such that $cv=av$.
Thus the action of $C_J(x)$ on $W_1$ satisfies {\rm(SF)}.
The polynomial assertion follows from Lemma~\ref{lem:quartic-scalar-principle}.
\end{proof}

\section{Abelian normal subgroups and refinement}

We now turn the polynomial constructions into information
about normal subgroups. In this section, we apply the quartic
obstructions to show that every abelian normal subgroup of
an admissible group is an elementary abelian $2$-group.
If such a subgroup acts nonscalarly, we pass to a common
eigenspace and the faithful action induced by its stabilizer.
We show that this refinement preserves admissibility and,
when assumed, condition $\boldsymbol{(R^-)}$ (defined below).
Iterating this process yields an admissible pair in which
every abelian normal subgroup acts by scalars.
We call such a pair \emph{$2$-terminal}.

\subsection{Abelian normal subgroups}

\begin{lemma}\label{lem:quartic-abelian-normal}
Let $p$ be an odd prime, let $W$ be a nonzero finite-dimensional
symplectic space over $\F_p$, with symplectic form $B$, and let
$L\leq\operatorname{Sp}(W)$ act irreducibly on $W$.
If $L$ has an abelian normal subgroup $A$ containing an element
of order greater than $2$, then $L$ admits a nonzero
invariant homogeneous quartic on $W$.

Consequently, if $p\geq7$ and every nonzero $w\in W$
is $L$-conjugate to $2w$, then every abelian normal
subgroup of $L$ is an elementary abelian 2-subgroup.
\end{lemma}

\begin{proof}
We first show that $p$ does not divide $|A|$.
Let $P$ be the Sylow $p$-subgroup of $A$.
Since $A$ is abelian, $P$ is characteristic in $A$ and hence
normal in $L$.

The action of $P$ on $W$ has a nonzero fixed vector, so
\[W^P=\{w\in W:uw=w\text{ for every }u\in P\}\ne\{0\}.\]
The normality of $P$ implies that $W^P$ is $L$-invariant.
Since $W$ is an irreducible $\F_p L$-module, we obtain $W^P=W$.
Hence $P$ acts trivially on $W$, and $P=1$ since
$L\leq\operatorname{Sp}(W)$ acts faithfully.
Thus $p\nmid|A|$.

Write $K=\overline{\F}_p$ and $W_K=K\otimes_{\F_p}W$,
and denote the extended symplectic form also by $B$.
Since $A$ is abelian and has order coprime to $p$,  the elements of $A$
are simultaneously diagonalizable on $W_K$. We therefore have
\[
 W_K=\bigoplus_{\lambda\in\Lambda}W_\lambda,
 \qquad
 W_\lambda=
 \{w\in W_K:aw=\lambda(a)w\text{ for all }a\in A\},
\]
where $\Lambda$ consists of the characters
$\lambda:A\to K^\times$ with $W_\lambda\ne0$.

For $u\in W_\lambda$, $v\in W_\mu$, and $a\in A$,
the invariance of $B$ yields
\[
 B(u,v)=B(au,av)=\lambda(a)\mu(a)B(u,v).
\]
Consequently,
\[
 B(W_\lambda,W_\mu)=0
 \qquad\text{unless }\mu=\lambda^{-1}.
\]
Since $B$ is nondegenerate on $W_K$, it follows that
$\lambda^{-1}\in\Lambda$ whenever $\lambda\in\Lambda$,
and that the pairing between $W_\lambda$ and
$W_{\lambda^{-1}}$ is nondegenerate.
Indeed, a vector in $W_\lambda$ orthogonal to
$W_{\lambda^{-1}}$ is orthogonal to every weight space,
and hence is zero; the same argument applies to the other
factor.

Observe that some $\lambda\in\Lambda$ satisfies
$\lambda\ne\lambda^{-1}$. Otherwise, every weight would take
its values in $\{1,-1\}$, and hence $a^2$ would act trivially
on $W_K$ for every $a\in A$.
Since $A$ acts faithfully on $W_K$, this would imply that
$a^2=1$ for every $a\in A$, contrary to the hypothesis that
$A$ contains an element of order greater than 2.

Let
\[
 W_0=\bigoplus_{\lambda\in\Lambda,~
                         \lambda=\lambda^{-1}}W_\lambda,
\]
and choose one representative $\lambda_i$ from each
unordered pair $\{\lambda,\lambda^{-1}\}$ with
$\lambda\ne\lambda^{-1}$. Then
\[
 W_K=W_0\oplus
 \bigoplus_{i\in I}
 \bigl(W_{\lambda_i}\oplus W_{\lambda_i^{-1}}\bigr),
\]
where $I$ is nonempty.
For $g\in L$, the normality of $A$ yields
\[
 gW_\lambda=W_{\lambda^g},
 \qquad
 \lambda^g(a)=\lambda(g^{-1}ag).
\]
Moreover,
\[
 (\lambda^{-1})^g=(\lambda^g)^{-1}.
\]
Thus $W_0$ is $L$-invariant, and $L$ permutes the unordered
pairs $\{W_{\lambda_i},W_{\lambda_i^{-1}}\}$.
All the hypotheses of Lemma~\ref{lem:quartic-dual-pairs}
are therefore satisfied.
Hence it yields a nonzero $L$-invariant homogeneous polynomial
of degree $4$ on $W_K$. By Lemma~\ref{lem:quartic-descent},
such a polynomial also exists on $W$ over $\F_p$.

Finally, suppose that $p\geq7$ and that every nonzero vector
$w\in W$ can be mapped to $2w$ by an element of $L$.
Since $4<p$ and
\[
 2^4-1=15\not\equiv0\pmod p,
\]
Lemma~\ref{lem:quartic-scalar-principle} excludes a nonzero
invariant homogeneous polynomial of degree 4.
By the first part, no abelian normal subgroup of $L$ can
contain an element of order greater than 2.
Every such subgroup is therefore an elementary abelian 2-subgroup.
\end{proof}

The following result collects the quartic obstructions in
the form used later. It excludes certain configurations of
paired constituents, certain tensor decompositions, and
abelian normal subgroups containing elements of order
greater than 2.

\begin{corollary}\label{lem:quartics}
Let $p\geq7$ be a prime, let $(W,B)$ be a nonzero
finite-dimensional symplectic space over $\F_p$, and let
$L\leq\operatorname{Sp}(W,B)$ act irreducibly on $W$.
Suppose that, for every $0\ne w\in W$, there exists $g\in L$
such that $gw=2w$.

Write $K=\overline{\F}_p$ and $W_K=K\otimes_{\F_p}W$, and
extend $B$ to $W_K$ by $K$-bilinearity.
Then $W_K$ admits no nonzero $L$-invariant homogeneous
polynomial of degree 4. Consequently, the following
statements hold.

\begin{enumerate}
\item
The space $W_K$ admits no $L$-preserved decomposition
satisfying the hypotheses of
Lemma~\ref{lem:quartic-dual-pairs}.

\item
The space $W_K$ admits neither a tensor decomposition
with an $L$-action by reciprocal similitudes as in
Lemma~\ref{lem:quartic-tensor}, nor a decomposition
into a nonempty permuted family of such tensor summands
as in Remark~\ref{rem:quartic-tensor-packets}.

\item
Every abelian normal subgroup of $L$ is an elementary
abelian $2$-subgroup.
\end{enumerate}
\end{corollary}

\begin{proof}
Since $p\geq7$, we have $4<p$ and
$2^4-1=15\ne0$ in $\F_p$. Lemma~\ref{lem:quartic-scalar-principle} therefore
excludes nonzero $L$-invariant homogeneous polynomials
of degree 4 on $W$.
By Lemma~\ref{lem:quartic-descent}, no such polynomial
exists on $W_K$ either.

Statements~{\rm(1)} and~{\rm(2)} now follow from
Lemma~\ref{lem:quartic-dual-pairs},
Lemma~\ref{lem:quartic-tensor}, and
Remark~\ref{rem:quartic-tensor-packets}, since
each of the indicated decompositions would produce
a nonzero $L$-invariant quartic on $W_K$.

For~{\rm(3)}, let $A$ be an abelian normal subgroup of $L$.
Since $L$ admits no nonzero invariant quartic on $W$,
Lemma~\ref{lem:quartic-abelian-normal} excludes elements
of order greater than 2 in $A$.
Therefore $A$ is an elementary abelian $2$-group.
\end{proof}

In view of Corollary~\ref{lem:quartics}{\rm(3)}, we distinguish
the case in which every abelian normal subgroup acts by scalars
(and hence does not split the space into smaller common eigenspaces).

\begin{definition}\label{def:2-terminal}
An admissible pair $(L,W)$ is called \emph{$2$-terminal}
if every normal elementary abelian $2$-subgroup of $L$
acts by scalars on $W$.
\end{definition}

\subsection{Reduction to a 2-terminal pair}

For a central involution $z\in Z(J)$ of a finite group $J$, write
\[
\Irr^-(J,z)=\{\chi\in\Irr(J):\chi(z)=-\chi(1)\},
\]
and say that $(J,z)$ satisfies $\boldsymbol{(R^-)}$ if every
character in $\Irr^-(J,z)$ is rational-valued. This condition concerns ordinary characters lying above the
central sign character. Under refinement, {\rm(R)} is checked
directly on the module, while $\boldsymbol{(R^-)}$ is preserved
by ordinary Clifford correspondence.

In this subsection, we show that every admissible pair can be reduced to a
$2$-terminal admissible pair. If a normal elementary abelian
$2$-subgroup acts nonscalarly, we pass to a common eigenspace
and the faithful action induced by its stabilizer. This
refinement preserves admissibility and strictly decreases
the dimension, so iteration must terminate. We also show
that condition $\boldsymbol{(R^-)}$, when assumed, is preserved throughout
the process.

The following lemma provides the rationality needed when passing to a common eigenspace and its stabilizer.

\begin{lemma}\label{lem:negative-clifford}
Let $A\lhd L$ be an elementary abelian $2$-subgroup, $\lambda\in\Irr(A)$, and $I=I_L(\lambda)$.
If $L$ is rational, then every character in $\Irr(I\mid\lambda)$ is rational-valued.

The same conclusion holds if there is a central involution
$z\in A\cap Z(L)$ such that $\lambda(z)=-1$ and
$(L,z)$ satisfies $\boldsymbol{(R^-)}$.
\end{lemma}

\begin{proof}
Since $A$ is an elementary abelian $2$-group, $\lambda$ takes values in $\{\pm1\}$ and hence is fixed by every cyclotomic
Galois automorphism. Consequently, the Clifford correspondence
\[
    \Irr(I\mid\lambda)\longrightarrow\Irr(L\mid\lambda),
    \qquad \psi\longmapsto\psi^L
\]
commutes with the action of the cyclotomic Galois group.
More precisely, for every $\psi\in\Irr(I\mid\lambda)$ and every
cyclotomic Galois automorphism $\sigma$, we have
$\psi^\sigma\in\Irr(I\mid\lambda)$ and $(\psi^\sigma)^L=(\psi^L)^\sigma$.

Fix $\psi\in\Irr(I\mid\lambda)$.
If $L$ is rational, then $\psi^L$ is rational-valued.
Thus, for every cyclotomic Galois automorphism $\sigma$,
\[
    (\psi^\sigma)^L=(\psi^L)^\sigma=\psi^L.
\]
The injectivity of the Clifford correspondence yields
$\psi^\sigma=\psi$, so $\psi$ is rational-valued.

For the second assertion, since $z\in Z(L)$ and
$\lambda(z)=-1$, we have
\[
    \Irr(L\mid\lambda)\subseteq\Irr^-(L,z).
\]
Hence $\psi^L$ is rational-valued by $\boldsymbol{(R^-)}$,
and the same argument shows that $\psi$ is rational-valued.
\end{proof}

\begin{lemma}\label{lem:sign-refinement}
Let $(L,W)$ be an admissible pair over $\mathbb F_p$,
where $p$ is an odd prime, and let $A\lhd L$
be an elementary abelian $2$-subgroup whose action on $W$ is nonscalar.
Let $U$ be a nonzero common eigenspace of $A$ on $W$,
and write
\[
 L_U=N_L(U)/C_L(U).
\]
Then $(L_U,U)$ is admissible and $0<\dim U<\dim W$.
If, in addition, $z=-I_W\in L$ and $(L,z)$ satisfies
$\boldsymbol{(R^-)}$, then $(L_U,-I_U)$ also satisfies
$\boldsymbol{(R^-)}$.
\end{lemma}

\begin{proof}
Let $k=\F_p$, and let $B$ be the $L$-invariant symplectic
form on $W$.
Since $A$ is an elementary abelian 2-subgroup and
$p$ is odd, we have
\begin{equation} \label{equ:orth-T}
 W=\bigoplus_{\mu\in\Omega}W_\mu,
 \qquad
 W_\mu=\{w\in W:aw=\mu(a)w\text{ for all }a\in A\},
\end{equation}
where $\Omega$ consists of the characters
$\mu:A\to\{1,-1\}$ occurring in $W$.
Write $U=W_\lambda$.

If $\mu\ne\nu$, choose $a\in A$ with
$\mu(a)\nu(a)=-1$.
For $u\in W_\mu$ and $v\in W_\nu$, the $L$-invariance of $B$ yields
\[
 B(u,v)=B(au,av)=\mu(a)\nu(a)B(u,v)=-B(u,v),
\]
so $B(u,v)=0$.
Thus the decomposition (\ref{equ:orth-T}) is orthogonal, and each
$W_\mu$ is nondegenerate.

Now the normality of $A$ implies that $L$ permutes these
weight spaces $W_\mu$, and the irreducibility of $W$ makes this
action transitive since the sum over any orbit
is $L$-invariant.
As $A$ is nonscalar, at least two weights occur.
Consequently,
\[
 0<\dim U<\dim W.
\]
Write $I=N_L(U)=I_L(\lambda)$.
The faithful induced group $L_U=I/C_I(U)$ preserves
the nondegenerate alternating form $B|_U$.

We verify the absolute irreducibility of $U$. Let $\overline{k}$ be an algebraic closure of $k$.
Assume that $D$ is a nonzero proper $I$-invariant subspace
of $\overline{k}\otimes_k U$. Then, for a set $T$
of representatives of the left cosets of $I$ in $L$,
\[
 \bigoplus_{t\in T}tD
\]
is a nonzero proper $L$-invariant subspace of
$\overline{k}\otimes_k W$.
Indeed, its summands lie in distinct weight spaces,
and $L$ permutes them.
This contradicts the absolute irreducibility of $W$.
Hence $U$ is absolutely irreducible for $L_U$.

For {\rm(SF)}, let $0\ne u\in U$ and $c\in k^\times$.
Choose a $p'$-element $g\in L$ with $gu=cu$.
Since $0\ne cu\in U\cap gU$ and distinct weight spaces
have zero intersection, $gU=U$.
Thus $g\in I$, and its image in $L_U$ is a
$p'$-element sending $u$ to $cu$.

For {\rm(R)}, let $x\in L_U$ be a $p'$-element,
and let $t\in I$ be a $p'$-lift of $x$.
The group $A\langle t\rangle$ has order prime to $p$,
so its representation on $W$ lifts to characteristic
zero with character equal to its Brauer character.
In this lift, the operator
\[
 e_\lambda=\frac1{|A|}
       \sum_{a\in A}\lambda(a)^{-1}a
\]
projects onto the $\lambda$-weight space.
Since $t$ stabilizes $\lambda$, it commutes with
$e_\lambda$.
The lifted weight space reduces to $U$, and therefore
\[
 \operatorname{BrTr}_U(x)
 =\operatorname{tr}(te_\lambda)
 =\frac1{|A|}\sum_{a\in A}
      \lambda(a)^{-1}\operatorname{BrTr}_W(ta).
\]
Every $ta$ is a $p'$-element.
By {\rm(R)} for $W$, the terms on the right are
integers, so $\operatorname{BrTr}_U(x)$ is rational.
Being a sum of roots of unity, it is also an
algebraic integer, and hence an integer.
The same applies to every power of $x$, proving
{\rm(R)} for $U$.
Thus $(L_U,U)$ is admissible.

Finally, suppose that $z=-I_W\in L$ and $(L,z)$
satisfies $\boldsymbol{(R^-)}$.
Replacing $A$ by $\langle A,z\rangle$ leaves its
weight spaces and their stabilizers unchanged.
We may therefore assume that $z\in A$, so
$\lambda(z)=-1$.
The image of $z$ in $L_U$ is $z_U=-I_U$.

Let $\theta\in\Irr^-(L_U,z_U)$ and inflate it to $I$.
Since the image of each $a\in A$ on $U$ is
$\lambda(a)I_U$, we have
\[
 \theta|_A=\theta(1)\lambda.
\]
Thus $\theta\in\Irr(I\mid\lambda)$, and
Lemma~\ref{lem:negative-clifford} implies that
$\theta$ is rational-valued.
Hence $(L_U,z_U)$ satisfies $\boldsymbol{(R^-)}$.
\end{proof}

The strict decrease in dimension ensures that the preceding
reduction terminates, as described in the following proposition.

\begin{proposition}
\label{prop:2-terminal-refinement}
Let $(L_0,W_0)$ be an admissible pair over $\F_p$,
where $p\geq7$.
Then successive applications of
Lemma~\ref{lem:sign-refinement} produce a finite
sequence of admissible pairs
\[
 (L_0,W_0),\ (L_1,W_1),\ \ldots,\ (L_m,W_m)
\]
such that $(L_m,W_m)$ is $2$-terminal.

If $z_0=-I_{W_0}\in L_0$ and $(L_0,z_0)$ satisfies
$\boldsymbol{(R^-)}$, then every pair in the sequence
also satisfies $\boldsymbol{(R^-)}$ with respect
to $z_i=-I_{W_i}$.

Moreover, every abelian normal subgroup of $L_m$
is contained in $\{I_{W_m},-I_{W_m}\}$.
\end{proposition}

\begin{proof}
Suppose that $(L_i,W_i)$ has been constructed.
If every normal elementary abelian $2$-subgroup
of $L_i$ acts by scalars, we are done.
Otherwise choose a nonscalar normal elementary
abelian $2$-subgroup $A_i\lhd L_i$, choose one
of its nonzero weight spaces $W_{i+1}$, and write
\[
 L_{i+1}
 =N_{L_i}(W_{i+1})/C_{L_i}(W_{i+1}).
\]
Lemma~\ref{lem:sign-refinement} shows that
$(L_{i+1},W_{i+1})$ is admissible and
\[
 0<\dim W_{i+1}<\dim W_i.
\]
It also preserves $\boldsymbol{(R^-)}$ whenever
that condition holds at the preceding stage.

Since the dimensions are positive integers,
the process terminates at some stage.
By the stopping condition, the final pair
$(L_m,W_m)$ is $2$-terminal.
If the initial pair is already $2$-terminal,
we take $m=0$.

Finally, Corollary~\ref{lem:quartics}{\rm(3)}
shows that every abelian normal subgroup of $L_m$
is an elementary abelian 2-subgroup.
The terminality of $(L_m,W_m)$ makes each such subgroup scalar.
Since scalar symplectic transformations in odd
characteristic are $I_{W_m}$ and $-I_{W_m}$,
the final assertion follows.
\end{proof}

\section{The Fitting structure of $2$-terminal admissible pairs}
\label{sec:generalized-fitting}

Throughout this section, let $(L,W)$ be a $2$-terminal
admissible pair over $\F_p$, where $p\geq7$. We study the Fitting subgroup
$F(L)$. The restrictions on abelian normal subgroups imply
that $F(L)=O_2(L)$. We then prove that either $F(L)$ is
extraspecial, or $|F(L)|\leq2$ and $E(L)\neq1$.
These alternatives guide the remainder of the proof.

\begin{lemma}\label{lem:terminal-fitting}
$F(L)=O_2(L)$.
Moreover, every characteristic abelian subgroup
of $O_2(L)$ has order at most 2.
\end{lemma}

\begin{proof}
Let $q$ be an odd prime, and suppose that
$Q=O_q(L)\ne1$.
Then $Z(Q)$ is nontrivial.
Since $Z(Q)$ is characteristic in $Q$ and
$Q\lhd L$, it is an abelian normal subgroup of $L$.
By Proposition \ref{prop:2-terminal-refinement}, we get that $|Z(Q)|\leq2$, which is impossible for a
nontrivial $q$-group. Hence $O_q(L)=1$ for every odd prime $q$, and
$F(L)=O_2(L)$.

If $A$ is a characteristic abelian subgroup of
$O_2(L)$, then $A\lhd L$.
Again by Proposition~\ref{prop:2-terminal-refinement}, we have
$|A|\leq2$, completing the proof.
\end{proof}

Recall that a finite $2$-group $R$ is
\emph{extraspecial} if
\[
 Z(R)=R'=\Phi(R)
 \quad\text{and}\quad |Z(R)|=2,
\]
where $\Phi(R)$ denotes the Frattini subgroup.

\begin{lemma}\label{lem:characteristic-abelian-2group}
Let $R$ be a finite $2$-group such that every
characteristic abelian subgroup of $R$ has order
at most 2.
Then either $|R|\leq2$ or $R$ is extraspecial.
\end{lemma}

\begin{proof}
If $R$ is abelian, then clearly $|R|\leq2$.
If $R$ is nonabelian, the conclusion follows from
a result of P. Hall \cite[p.~218,~Remark~6.4]{Berkovich2005}.
\end{proof}

Recall that a component of a finite group $L$ is a subnormal quasisimple
subgroup, and that the layer $E(L)$ is the product of all components.
Distinct components of $L$ commute, and $L$ permutes its components
by conjugation. In particular, $E(L)$ is a perfect normal
subgroup of $L$.

Also, recall that $F^*(L)=F(L)E(L)$ and $C_L(F^*(L))\leq F^*(L)$; see
\cite[\S 6.5]{Kurz}.

\begin{proposition}
\label{prop:terminal-first-alternatives}
We have $F(L)=O_2(L)$, and exactly one of the following holds:
\begin{enumerate}
\item
$O_2(L)$ is extraspecial and $Z(O_2(L))=\langle -I_W\rangle$.
\item
$O_2(L)\leq\{I_W,-I_W\}$ and $E(L)\ne1$.
\end{enumerate}
\end{proposition}

\begin{proof}  Write $R=O_2(L)$.
By Lemma~\ref{lem:terminal-fitting},
$F(L)=R$, and every characteristic abelian
subgroup of $R$ has order at most 2.
Lemma~\ref{lem:characteristic-abelian-2group}
therefore shows that either $|R|\leq2$
or $R$ is extraspecial.

If $R$ is extraspecial, then $|Z(R)|=2$, hence $Z(R)=\langle -I_W\rangle$,
proving {\rm(1)}.

Suppose that $|R|\leq2$. Then $R$ is abelian and normal in $L$, so
$R\leq\{I_W,-I_W\}$. If $E(L)=1$, then
\[
 F^*(L)=R\leq Z(L), \qquad
 L=C_L(F^*(L))\leq F^*(L)=R.
\]
Thus every element of $L$ acts as $I_W$ or $-I_W$.
This contradicts {\rm(SF)}, since $W\ne0$ and
$2\notin\{1,-1\}$ in $\F_p$.
Hence $E(L)\ne1$, proving {\rm(2)}.
\end{proof}

\section{The extraspecial case}
\label{sec:extraspecial}

\noindent\textbf{Assumptions.}
Throughout this section, let $p\geq7$ be a prime and let
$(L,W)$ be a $2$-terminal admissible pair over $k=\F_p$,
with symplectic form $B$.  We write $\Sp(W,B)$ for the group of all $k$-linear
automorphisms of $W$ preserving $B$.
Suppose that $R=O_2(L)$ is extraspecial of order $2^{1+2m}$,
where $m\geq1$.
By Proposition~\ref{prop:terminal-first-alternatives},
\[
 Z(R)=\langle z\rangle,\qquad z=-I_W.
\]

\subsection{Minus type and the irreducible action of the extraspecial Fitting subgroup}
\label{subsec:extraspecial}

Here we shall show that $R$ has minus type and acts
absolutely irreducibly on $W$.
Recall that an extraspecial group of order $2^{1+2m}$
has \emph{plus type} if it is a central product
of $m$ dihedral groups of order 8, and
\emph{minus type} if it is a central product of
one quaternion group of order 8 and $m-1$
dihedral groups of order 8.

The following result describes the dimension and form type of the (unique)
 faithful irreducible module of an extraspecial 2-group.

\begin{lemma}\label{lem:extraspecial-module}
Let $K$ be an algebraically closed field of odd
characteristic, and let $R$ be an extraspecial
$2$-group of order $2^{1+2m}$.
Then $R$ has a unique faithful irreducible
$K$-module $X$, up to isomorphism, and
\[
 \dim_K X=2^m.
\]
The space of $R$-invariant bilinear forms on $X$
is one-dimensional and is generated by a
nondegenerate form $b$.
The form $b$ is symmetric if $R$ has plus type
and alternating if $R$ has minus type.
\end{lemma}

\begin{proof}
The existence and uniqueness of the faithful irreducible
module $X$, its dimension $2^m$, and the asserted type
of its nondegenerate invariant bilinear form follow
from \cite[Section~3]{GlasbyHowlett1992}.
Since $K$ is algebraically closed, Schur's lemma shows that
\[
 \dim_K\operatorname{Hom}_{KR}(X,X^*)=1.
\]
Hence the space of $R$-invariant bilinear forms on $X$
is one-dimensional.
\end{proof}

\begin{lemma}\label{lem:extraspecial-multiplicity}
Write $K=\overline{k}$ and $W_K=K\otimes_k W$.
Let $X$ be the faithful irreducible $KR$-module.
Then there is a nonzero $K$-space $M$ such that
$W_K\cong X\otimes_K M$,
where $R$ acts trivially on $M$.

Under this identification, the extended
symplectic form has the form
$B_K=b\otimes c$,
where $b$ is the form in
Lemma~\ref{lem:extraspecial-module} and $c$ is
a nondegenerate bilinear form of the opposite
type.

Moreover, every $\ell\in L$ acts on $W_K$ as
$s_\ell\otimes t_\ell$
for suitable $s_\ell\in\operatorname{GL}(X)$
and $t_\ell\in\operatorname{GL}(M)$, and these
maps are similitudes of $b$ and $c$ with
reciprocal multipliers.
\end{lemma}

\begin{proof}
Since $p$ is odd, the restriction of $W_K$ to
$R$ is semisimple.
The central involution $z$ acts as $-I$ on
$W_K$, and hence on every irreducible constituent.
By Lemma~\ref{lem:extraspecial-module}, every
constituent is isomorphic to $X$.
Thus
\[
 W_K\cong X\otimes_K M
\]
for a nonzero multiplicity space $M$, with
$R$ acting only on the first factor.

For fixed $u,v\in M$, the map
\[
 (x,y)\longmapsto B_K(x\otimes u,y\otimes v)
\]
is an $R$-invariant bilinear form on $X$.
By Lemma~\ref{lem:extraspecial-module}, it is a
scalar multiple of $b$.
These scalars define a bilinear form $c$ on $M$
such that
\[
 B_K(x\otimes u,y\otimes v)=b(x,y)c(u,v).
\]
Moreover, the form $c$ is nondegenerate. Indeed, if $0\ne u\in M$
satisfied $c(u,v)=0$ for every $v\in M$, then, for any
$0\ne x\in X$, the nonzero vector $x\otimes u$ would be
orthogonal under $B_K$ to every pure tensor, and hence
to all of $X\otimes_K M$. This contradicts the
nondegeneracy of $B_K$.

Write
\[
 b(y,x)=\varepsilon b(x,y),
 \qquad \varepsilon\in\{1,-1\}.
\]
Since $B_K$ is alternating, we obtain
\[
 c(v,u)=-\varepsilon c(u,v).
\]
Thus $c$ is alternating when $b$ is symmetric,
and symmetric when $b$ is alternating.

Let $\rho:R\to\operatorname{GL}(X)$ denote the
representation on $X$, and fix $\ell\in L$.
The normality of $R$ implies that conjugation by
$\ell$ induces an automorphism of $R$. It follows that
the representation $r\longmapsto\rho(\ell r\ell^{-1})$ is faithful and irreducible, and hence is
equivalent to $\rho$.
Choose $s_\ell\in\operatorname{GL}(X)$ such that
\[
 s_\ell\rho(r)s_\ell^{-1}
 =\rho(\ell r\ell^{-1})
 \qquad(r\in R).
\]
It follows that
$(s_\ell^{-1}\otimes I_M)\ell$
commutes with $R$.
Since
\[
 \operatorname{End}_{KR}(X\otimes_K M)
 =I_X\otimes\operatorname{End}_K(M),
\]
there exists $t_\ell\in\operatorname{GL}(M)$
such that $\ell=s_\ell\otimes t_\ell$.

The form $(x,y)\mapsto b(s_\ell x,s_\ell y)$
is $R$-invariant.
Uniqueness of $b$ up to scalars gives
$\alpha_\ell\in K^\times$ with
\[
 b(s_\ell x,s_\ell y)=\alpha_\ell b(x,y).
\]
Since $\ell$ preserves $B_K=b\otimes c$, we
then have
\[
 c(t_\ell u,t_\ell v)
 =\alpha_\ell^{-1}c(u,v).
\]
Thus the multipliers are reciprocal.
\end{proof}

\begin{proposition}\label{prop:extraspecial-irreducible}
The group $R$ has minus type and acts absolutely
irreducibly on $W$.
In particular,
\[
 \dim_k W=2^m
 \qquad\text{and}\qquad
 C_L(R)=Z(R)=\langle -I_W\rangle.
\]
\end{proposition}

\begin{proof}
Write $K=\overline{k}$.
By Lemma~\ref{lem:extraspecial-multiplicity}, we have
\[
 W_K=X\otimes_K M,\qquad B_K=b\otimes c,
\]
where $X$ is the faithful irreducible $KR$-module with $\dim X=2^m\geq2$,
the forms $b$ and $c$ are nondegenerate and of
opposite types, and $L$ acts through reciprocal
similitudes on the two factors.
If $\dim M\geq2$, then, after interchanging the factors if necessary,
this is precisely the tensor configuration
excluded by Corollary~\ref{lem:quartics}{\rm(2)}.
Therefore $\dim M=1$.

A nondegenerate bilinear form on a one-dimensional
space in odd characteristic is symmetric.
Thus $c$ is symmetric and $b$ is alternating.
By Lemma~\ref{lem:extraspecial-module}, $R$ has
minus type.
Moreover, $W_K|_R\cong X$, so $R$ acts absolutely
irreducibly on $W$ and
\[
 \dim_k W=2^m.
\]

Finally, Schur's lemma over $K$ shows that every
$k$-linear transformation commuting with $R$
is scalar over $K$.
Its matrix has entries in $k$, so its scalar
belongs to $k$.
Hence
\[
 C_{\operatorname{GL}_k(W)}(R)=k^\times I_W.
\]
A scalar $aI_W$ preserves the nondegenerate
symplectic form only if $a^2=1$.
Since $Z(R)=\langle -I_W\rangle$, we conclude that
\[
 C_L(R)=\{I_W,-I_W\}=Z(R).
\]
\end{proof}

\begin{corollary}\label{cor:extraspecial-quotient}
There is an embedding
\[
 L/R\hookrightarrow O^-_{2m}(2).
\]
\end{corollary}

\begin{proof}
Write $Z=Z(R)=\langle z\rangle$ and $V=R/Z$.
By Proposition~\ref{prop:extraspecial-irreducible}, we have $C_L(R)=Z$.
Hence the conjugation action of $L$ on $R$ induces an embedding
$L/Z\hookrightarrow\operatorname{Aut}(R)$.
Under this embedding, $R/Z$ corresponds to $\operatorname{Inn}(R)$.
Therefore,
\[
 L/R\hookrightarrow\operatorname{Out}(R).
\]
 By \cite[Theorem 1(c)]{Winter1972}, $\operatorname{Out}(R)\cong O^-_{2m}(2)$,
 finishing the proof.
\end{proof}

\subsection{The extraspecial case in characteristic $p\geq11$}
\label{sec:degree-eight}

By Proposition~\ref{prop:extraspecial-irreducible}, $R$ acts
absolutely irreducibly on $W$. Since $R\leq L$ preserves
the symplectic form $B$, the pair $(R,W)$ satisfies {\rm(I)}.

We construct a nonzero invariant homogeneous polynomial of degree 8 and use
scalar fusion to exclude the extraspecial case when $p\geq11$.

\begin{lemma}
\label{lem:extraspecial-octic}
The normalizer $N_{\operatorname{Sp}(W,B)}(R)$ admits
a nonzero invariant homogeneous polynomial of degree $8$
on $W$.
\end{lemma}

\begin{proof}
Write $N:=N_{\operatorname{Sp}(W,B)}(R)$ and $|R|=2^{1+2m}$.
Since $W$ is absolutely irreducible, the action of $N$  on $W$ is faithful, and the characteristic of $k$
is odd, we have $Z:=Z(R)=\langle -I_W\rangle$.  Let
\[\mathcal S=\{rZ\in R/Z:r^2=-I_W\},\]
 and choose a representative $r_e$ for each
$e\in\mathcal S$. Define
\[
 F(w)=\sum_{e\in\mathcal S}B(w,r_ew)^4.
\]
We show that $F$ is well-defined, $N$-invariant, and nonzero.

The two representatives of a coset $rZ$ are $r$ and $-r$, and they have the same square.
Thus $\mathcal S$ is well-defined. Replacing $r_e$ by $-r_e$ changes $B(w,r_ew)$
by a sign but leaves its fourth power unchanged.
Therefore $F$ is independent of the chosen representatives, and so
$F$ is well-defined.
Since each $B(w,r_ew)$ is homogeneous of degree 2, we see that
$F$ is homogeneous of degree 8 if $F$ is nonzero.

Let $g\in N$. Conjugation by $g$ permutes the cosets in $\mathcal S$.
Since $g$ preserves $B$,
\[
 F(gw)
 =\sum_{e\in\mathcal S}B(w,g^{-1}r_egw)^4
 =F(w).
\]
Hence $F$ is $N$-invariant.

To prove that $F$ is nonzero, we extend scalars to an algebraic closure $K$ of $k$.
Since $R$ has minus type, it is a central product
of one quaternion group of order 8 and
$m-1$ dihedral groups of order 8.
The tensor product of their faithful
two-dimensional representations is faithful
and irreducible.
By Lemma~\ref{lem:extraspecial-module}, we may
therefore identify $W_K$ with this tensor-product
module.

Choose $\iota\in K$ with $\iota^2=-1$, and set
\[
 A=\begin{pmatrix}0&1\\-1&0\end{pmatrix},
 \qquad
 C=\begin{pmatrix}\iota&0\\0&-\iota\end{pmatrix},
 \qquad
 D=\begin{pmatrix}1&0\\0&-1\end{pmatrix},
 \qquad
 S=\begin{pmatrix}0&1\\1&0\end{pmatrix}.
\]
Thus $A^2=C^2=-I$ and $D^2=S^2=I$.
Since \[A^TAA=C^TAC=A,\] the quaternion group $\langle A,C\rangle$ preserves the
nondegenerate alternating form
$b_-(u,v)=u^TAv$ on $K^2$.
Similarly, since \[D^TD=S^TS=I,\] the dihedral
group $\langle D,S\rangle$ preserves the
nondegenerate symmetric form
$b(u,v)=u^Tv$ on $K^2$.
Here vectors are written as columns.
The tensor-product invariant form $b_-\otimes b^{\otimes(m-1)}$ has matrix
\[
 J=A\otimes I^{\otimes(m-1)}.
\]
By uniqueness of the invariant bilinear form
up to scalars, there exists $\gamma\in K^\times$
such that
\[
 B_K(u,w)=\gamma u^TJw.
\]

Notice that $\mathcal{R}_-:=\{I,A,C,AC\}$ and $\mathcal{R}_+:=\{I,D,S,DS\}$ are
sets of representatives for the cosets of $Z$
in the groups $\langle A,C \rangle$
and $\langle D,S\rangle$, respectively.
 Hence the matrices
\[
 r=T_0\otimes T_1\otimes\cdots\otimes T_{m-1},
\]
where $T_0\in \mathcal{R}_-$ and
$T_j\in \mathcal{R}_+$ for $j\geq1$, form a complete
set of representatives for the cosets of $Z(R)$
in $R$.

Let $e=(1,0)^T$ and $v=e^{\otimes m}$. Then $rv=T_0e\otimes T_1e\otimes\cdots\otimes T_{m-1}e$ and
\[
\begin{aligned}
 B_K(v,rv)
 &=\gamma b_-(e,T_0e)\prod_{j=1}^{m-1}b(e,T_je)\\
 &=\gamma (e^TAT_0e)\prod_{j=1}^{m-1}(e^TT_je),
\end{aligned}
\]
where $e^TMe$ is the $(1,1)$-entry of a matrix $M$.
Direct calculation yields
\[
\begin{array}{c|rrrr}
 T_0&I&A&C&AC\\ \hline
 e^TAT_0e&0&-1&0&-\iota
\end{array}
\qquad\text{and}\qquad
\begin{array}{c|rrrr}
 T_j&I&D&S&DS\\ \hline
 e^TT_je&1&1&0&0.
\end{array}
\]
Consequently, $B_K(v,rv)\ne0$ precisely when
$T_0\in\{A,AC\}$ and $T_j\in\{I,D\}$ with $1\leq j\leq m-1$.
There are exactly $2^m$ such representatives.
Since
\[
 A^2=(AC)^2=-I,\qquad I^2=D^2=I,
\]
each satisfies $r^2=-I_W$ and hence represents
a coset in $\mathcal S$.
For each of them,
\[
 B_K(v,rv)\in\{-\gamma,-\gamma\iota\},
 \qquad B_K(v,rv)^4=\gamma^4.
\]
Every other representative contributes zero.
Therefore
\[
 F_K(v)=2^m\gamma^4\ne0,
\]
since the characteristic of $k$ is odd.
When $m=1$, the empty product is interpreted
as 1, hence the same calculation applies.
Thus the scalar extension $F_K$ is nonzero.
Injectivity of scalar extension implies that
$F$ is nonzero over $k$, completing the proof.
\end{proof}

\begin{proposition}\label{prop:extraspecial-large-characteristic}
If $p\geq11$, then $O_2(L)$ is not extraspecial.
\end{proposition}

\begin{proof}
Assume, for a contradiction, that $O_2(L)$ is extraspecial.
By Lemma~\ref{lem:extraspecial-octic}, there is
a nonzero $L$-invariant homogeneous polynomial
$F$ of degree 8 on $W$.

Choose a generator $a$ of $\F_p^\times$.
Since $p-1>8$, we have $a^8\ne1$.
For every $0\ne w\in W$, condition {\rm(SF)}
provides $g\in L$ such that $gw=aw$. The $L$-invariance and homogeneity of $F$ yield
\[
 F(w)=F(gw)=F(aw)=a^8F(w).
\]
Hence $F(w)=0$ for every nonzero $w\in W$.
Also $F(0)=0$, so $F$ vanishes on all of $W$.
Since $8<p$, the polynomial vanishing argument in the first
part of the proof of Lemma~\ref{lem:quartic-scalar-principle}
implies that $F=0$, a contradiction.
\end{proof}

\begin{remark}
The preceding argument uses fusion with a generator
of $\F_p^\times$.
Fusion with the scalar $2$ alone does not suffice
uniformly for all primes $p\geq11$, since
$2^8$ can equal $1$ modulo $p$.

For the primes $p\in\{7,11\}$ occurring in our
application, Proposition~\ref{prop:extraspecial-large-characteristic}
eliminates $p=11$.
The extraspecial case in characteristic 7
remains to be considered.
\end{remark}

\subsection{The extraspecial case in characteristic $p=7$}
\label{sec:extraspecial-seven}

Throughout this subsection, let $p=7$.
We exclude the extraspecial case using scalar fusion when
$m\geq3$, and the additional condition $\boldsymbol{(R^-)}$
when $m=1,2$.

The invariant homogeneous polynomial of degree 8 used in the preceding subsection
does not yield the same contradiction in characteristic $7$, since
its degree exceeds the characteristic, so its vanishing
on all of $W$ does not imply that it is the zero polynomial.
We instead use integer-valued counts of nonzero quadratic
contractions.

\begin{lemma}\label{lem:extraspecial-seven-large}
Suppose that $p=7$ and $m\geq3$.
Then there exists $0\ne w\in W$ such that $gw\ne3w$
for every $g\in N_{\Sp(W,B)}(R)$.
Consequently, no subgroup of $N_{\Sp(W,B)}(R)$
satisfies {\rm(SF)} on $W$.
\end{lemma}

\begin{proof}
Write $N=N_{\Sp(W,B)}(R)$ and $Z=Z(R)=\langle-I_W\rangle$. We shall first construct a vector $v$ for which the number
\[
 d(v)=\bigl|\{rZ\in R/Z:r^2=-I_W,\ B(v,rv)\ne0\}\bigr|
\]
is a power of 2. The set counted by $d(v)$ is well defined, since replacing
$r$ by $-r$ changes neither $r^2$ nor the vanishing of $B(v,rv)$.
Over $\F_7$, let
\[
 I=\begin{pmatrix}1&0\\0&1\end{pmatrix},\quad
 J=\begin{pmatrix}0&-1\\1&0\end{pmatrix},\quad
 A=\begin{pmatrix}3&2\\2&4\end{pmatrix},\quad
 D=\begin{pmatrix}1&0\\0&-1\end{pmatrix},\quad
 X=JD=\begin{pmatrix}0&1\\1&0\end{pmatrix}.
\]
Since $J^2=A^2=-I$ and $JA=-AJ$, the matrices $J,A$ generate a quaternion group,
which preserves the alternating form with matrix $J$. The matrices $D,X$ generate a dihedral group of order 8, which  preserves
the symmetric form with matrix $I$. We use the central product model
\[
 R=Q_8*D_8*\cdots*D_8
\]
on $W=(\F_7^2)^{\otimes m}$, where the invariant alternating form $B'$ has matrix
$J\otimes I^{\otimes(m-1)}$.
This model affords the unique faithful absolutely irreducible
module of $R$; equivalence with the given module holds over
$\F_7$, since the intertwining linear equations commute with
extension of scalars.
After identifying the given module with this tensor model,
Schur's lemma shows that the space of $R$-invariant
bilinear forms is one-dimensional. Hence $B$ is a nonzero
scalar multiple of $B'$. Rescaling the form changes neither
its symplectic group nor the number $d(v)$, so we may
assume that $B=B'$.

Notice that representatives of $R/Z$ are
\[
 H\otimes M_1\otimes\cdots\otimes M_{m-1},
 \qquad H\in\{I,J,A,JA\},\quad M_i\in\{I,J,D,X\}.
\]
We first suppose $m=3$. Write
\[
 e_0=\begin{pmatrix}1\\0\end{pmatrix},\qquad
 e_1=\begin{pmatrix}0\\1\end{pmatrix},\qquad
 u=e_0+e_1=\begin{pmatrix}1\\1\end{pmatrix},
\]
and choose
\[
 s=u\otimes u,\qquad t=e_0\otimes e_0,\qquad
 v_3=e_0\otimes s+e_1\otimes t.
\]
For $M=M_1\otimes M_2$, where $M_i\in\{I,J,D,X\}$, set
\[
 a=s^{\mathsf T}Ms,\qquad b=s^{\mathsf T}Mt,
 \qquad c=t^{\mathsf T}Mt.
\]
The elementary identities
\[
\begin{array}{c|rrrr}
 Y&I&J&D&X\\\hline
 u^{\mathsf T}Yu&2&0&0&2\\
 u^{\mathsf T}Ye_0&1&1&1&1\\
 e_0^{\mathsf T}Ye_0&1&0&1&0
\end{array}
\]
yield $b=1$ for every $M$. Also,
\[ a=\begin{cases}4,&M_1,M_2\in\{I,X\},\\0,&\text{otherwise},\end{cases}
 \qquad
 c=\begin{cases}1,&M_1,M_2\in\{I,D\},\\0,&\text{otherwise}.\end{cases}
\]

There are six skew-symmetric choices of $M$, namely those
with exactly one factor equal to $J$. Each has $M^2=-I$.
For such $M$, the unique choice of first factor giving square
$-I$ is $H=I$, and
\[
 B(v_3,(I\otimes M)v_3)=-2b=-2\ne0.
\]
The other ten choices of $M$ are symmetric and have $M^2=I$.
They give square $-I$ precisely when $H\in\{J,A,JA\}$.
Direct multiplication of the two-by-two matrices yields
\[
\begin{aligned}
 B(v_3,(J\otimes M)v_3)&=-a-c,\\
 B(v_3,(A\otimes M)v_3)&=-2a+6+2c,\\
 B(v_3,(JA\otimes M)v_3)&=-3a-4-4c.
\end{aligned}
\]
For the four possible pairs $(a,c)$, these contractions
have the following values in $\F_7$:
\[
\begin{array}{c|ccc}
(a,c)&-a-c&-2a+6+2c&-3a-4-4c\\ \hline
(0,0)&0&6&3\\
(0,1)&6&1&6\\
(4,0)&3&5&5\\
(4,1)&2&0&1
\end{array}
\]
 Only their zeros matter. For the first expression, a zero
occurs precisely when $a=c=0$, which, among the symmetric
choices, means
\[
 M=J\otimes J,\qquad D\otimes X,\qquad X\otimes D.
\]
For the second expression, a zero occurs precisely when
$(a,c)=(4,1)$, which means $M=I\otimes I$.
The third never vanishes.
Thus, among the $6+3\cdot10=36$ cosets of square $-I$,
exactly four give zero contraction. Consequently,
\[
 d(v_3)=36-4=32.
\]

For general $m\geq3$, take
\[
 v_m=v_3\otimes e_0^{\otimes(m-3)}.
\]
In each additional dihedral factor,
\[
 e_0^{\mathsf T}Ie_0=e_0^{\mathsf T}De_0=1,
 \qquad e_0^{\mathsf T}Je_0=e_0^{\mathsf T}Xe_0=0.
\]
Hence a nonzero contraction requires every additional factor
to be $I$ or $D$. Both have square $I$, and both preserve
the contraction. Every coset counted by $d(v_3)$ therefore
gives exactly $2^{m-3}$ cosets counted by $d(v_m)$, and
there are no others. We obtain
\[
 d(v_m)=32\cdot2^{m-3}=2^{m+2}.
\]
In particular, this is not divisible by 3.

We now show that, for any $v\in W$, the existence of
$g\in N$ with $gv=3v$ forces $3\mid d(v)$.
For a coset $e=rZ$ with $r^2=-I_W$, define
\[
 q_e(v)=B(v,rv)^2.
\]
This is independent of the representative $r$,
since $B(v,(-r)v)^2=(-B(v,rv))^2=B(v,rv)^2$.
Conjugation by $g\in N$ permutes these cosets, and
\[
 q_{geg^{-1}}(gv)=q_e(v).
\]
If $gv=3v$, then $3^4=4$ in $\F_7$ yields
\[
 q_{geg^{-1}}(v)=2q_e(v).
\]
Thus conjugation by $g$ permutes the $d(v)$ cosets with
$q_e(v)\ne0$, multiplying their values by $2$.
Every cycle of this permutation has length divisible by
3, since $2$ has multiplicative order 3 in
$\F_7^\times$. It follows that $3\mid d(v)$.

Therefore, applying this to $v=v_m$ contradicts
$d(v_m)=2^{m+2}$. Since $v_m\ne0$, taking $w=v_m$
proves the first assertion. The final assertion follows
from {\rm(SF)} applied to the scalar $3\in\F_7^\times$.
\end{proof}

The two smallest extraspecial groups require the additional hypothesis
$\boldsymbol{(R^-)}$. The quotient restrictions and
Lemma~\ref{lem:gallagher-quotient} provide the information needed
to exclude these cases.

\begin{lemma}
\label{lem:gallagher-quotient}
Let $R\lhd J$, and let $z\in R\cap Z(J)$ be an involution.
Suppose that $(J,z)$ satisfies $\boldsymbol{(R^-)}$ and that
$\chi_R\in\Irr(R)$ for some $\chi\in\Irr^-(J,z)$.
Then $J/R$ is a rational group.
\end{lemma}

\begin{proof}
Let $\tau\in\Irr(J/R)$, and also denote its inflation to $J$
by $\tau$. Since $z\in R$, we have $\tau(z)=\tau(1)$.
As $\chi_R$ is irreducible, Gallagher's theorem
\cite[Corollary 6.17]{I} shows that $\chi\tau\in\Irr(J)$.
Moreover,
\[
    (\chi\tau)(z)
    =\chi(z)\tau(z)
    =-\chi(1)\tau(1)
    =-(\chi\tau)(1).
\]
Thus $\chi\tau\in\Irr^-(J,z)$, and both $\chi$ and
$\chi\tau$ are rational-valued by $\boldsymbol{(R^-)}$.

To show that $\tau$ is rational-valued, choose a unitary
representation $\rho$ of $J$ affording $\chi$ (see \cite[Theorem 4.17]{I}), and write
$d=\chi(1)$.
Since $\chi_R$ is irreducible, $\rho|_R$ is irreducible.

We first establish the matrix coefficient orthogonality
identity needed below. For $1\leq j,\ell\leq d$, let
$E_{j\ell}$ be the matrix whose $(j,\ell)$-entry is $1$
and whose other entries are 0, and set
\begin{equation}\label{eq:similar}
    T_{j\ell}
    =\sum_{r\in R}\rho(r)E_{j\ell}\rho(r)^{-1}.
\end{equation}
For every $s\in R$, replacing $r$ by $sr$ in the sum yields
$\rho(s)T_{j\ell}\rho(s)^{-1}=T_{j\ell}$.
Thus $T_{j\ell}$ commutes with $\rho(R)$.
By Schur's lemma, $T_{j\ell}=cI_d$ for some scalar $c$.
Taking traces in \eqref{eq:similar}, we have
\[
    dc
    =\operatorname{tr}(T_{j\ell})
    =|R|\operatorname{tr}(E_{j\ell})
    =|R|\delta_{j\ell},
\]
where $\delta$ denotes the Kronecker delta. Hence
\[
    T_{j\ell}=\frac{|R|}{d}\delta_{j\ell}I_d.
\]
Comparing the $(i,k)$-entries, and using
$\rho(r)^{-1}=\rho(r)^*$, we obtain
\begin{equation}\label{eq:orth}
    \sum_{r\in R}\rho(r)_{ij}
        \overline{\rho(r)_{k\ell}}
    =\frac{|R|}{d}\delta_{ik}\delta_{j\ell}.
\end{equation}

Now fix $g\in J$. For every $r\in R$, we have
\begin{equation}\label{eq:chi(g)}
    \chi(rg)
    =\operatorname{tr}(\rho(r)\rho(g))
    =\sum_{i,j=1}^{d}\rho(r)_{ij}\rho(g)_{ji}.
\end{equation}
Consequently, \eqref{eq:orth} and \eqref{eq:chi(g)} yield
\[
\begin{aligned}
    \sum_{r\in R}|\chi(rg)|^2
    &=
    \sum_{i,j,k,\ell=1}^{d}
        \rho(g)_{ji}\overline{\rho(g)_{\ell k}}
        \sum_{r\in R}\rho(r)_{ij}
            \overline{\rho(r)_{k\ell}}\\
    &=\frac{|R|}{d}\sum_{i,j=1}^{d}|\rho(g)_{ji}|^2\\
    &=|R|,
\end{aligned}
\]
where the last equality follows from the unitarity of $\rho(g)$.
In particular, there exists $r\in R$ such that
$\chi(rg)\neq0$. Since $\tau$ is inflated from $J/R$,
it is constant on the coset $Rg$. Therefore
\[
    \tau(g)
    =\tau(rg)
    =\frac{(\chi\tau)(rg)}{\chi(rg)}
    \in\mathbb{Q}.
\]
Since $g\in J$ was arbitrary, $\tau$
is rational-valued.
Hence
 every irreducible character of $J/R$ is
rational-valued, and thus $J/R$ is a rational group.
\end{proof}

\begin{remark}\label{rem:one-character}
Condition $\boldsymbol{(R^-)}$ is an additional hypothesis, since it does not
follow from admissibility. Indeed, a faithful two-dimensional
symplectic action of $\operatorname{SL}_2(3)\cong Q_8\rtimes C_3$
over $\F_7$ satisfies {\rm(SF)} and {\rm(R)}, but its negative
irreducible character family contains nonrational characters.
This is why we retain condition $\boldsymbol{(R^-)}$ explicitly throughout
the refinement process.
\end{remark}

\begin{lemma}\label{lem:extraspecial-seven-small} Let $p=7$.
Suppose that $m\in\{1,2\}$ and that $(L,z)$ satisfies
$\boldsymbol{(R^-)}$.
Then $L$ does not satisfy {\rm(SF)} on $W$.
\end{lemma}

\begin{proof}
Suppose, for a contradiction, that $L$ satisfies {\rm(SF)}.
Write $N=N_{\Sp(W,B)}(R)$ and $P=L/R$.
Since $R$ acts faithfully and absolutely irreducibly on $W$,
we have
\[
 \dim W=2^m,\qquad Z(R)=\langle z\rangle,
 \qquad z=-I_W.
\]

We first describe the action of $N$ on $R/Z(R)$.
The quadratic form defined by
$r^2=z^{q(rZ(R))}$ is nondegenerate and has minus type. It is clear that conjugation by $N$
preserves this form. Moreover, the kernel of this action
is $R$. Indeed, an automorphism of $R$ acting trivially
on $R/Z(R)$ is inner. Thus, after multiplication by an
element of $R$, an element of the kernel centralizes $R$.
By the absolute irreducibility of $W$ and Schur's lemma, $C_{\GL(W)}(R)=\F_7^\times I_W$.
Since a scalar isometry has square one, it follows that
\[
 C_{\Sp(W,B)}(R)=\{I_W,-I_W\}\leq R.
\]
Consequently,
$P=L/R\leq N/R\leq O^-_{2m}(2)$.

In the two ranks under consideration, we have
\[O^-_2(2)\cong S_3,\qquad O^-_4(2)\cong S_5.\]
The first identification follows because every nonzero vector
of the anisotropic plane over $\F_2$ is nonsingular.
For the second, use the quadratic space
\[
 V_0=\left\{(a_1,\ldots,a_5)\in\F_2^5:
                 \sum_{i=1}^5a_i=0\right\},
 \qquad
 q(a)=\frac{\operatorname{wt}(a)}2\pmod2.
\]
Its five nonzero singular vectors are the vectors of weight
4, and they span $V_0$. Every isometry therefore permutes
these five vectors faithfully. Conversely, coordinate
permutations give all permutations of these vectors, proving
the claimed identification.

In particular, $7\nmid |L|$. By the ordinary--Brauer correspondence for groups of order prime to $7$,
the given absolutely irreducible module lifts to an ordinary
irreducible character $\chi$ of $L$.
 This correspondence is compatible
with restriction to $R$, so $\chi_R$ is the faithful irreducible
extraspecial character of degree $2^m$.
 Thus
\[
 \chi_R\in\Irr(R),\qquad \chi(z)=-\chi(1).
\]
Since $(L,z)$ satisfies $\boldsymbol{(R^-)}$,
Lemma~\ref{lem:gallagher-quotient} shows that $P=L/R$
is rational.

We next prove that
\[
 P=O^-_{2m}(2).
\]

Suppose first that $m=1$. By {\rm(SF)}, for any nonzero
$w\in W$ there is $t\in L$ such that $tw=2w$.
Since $2$ has multiplicative order 3 in $\F_7^\times$,
the order of $t$ is divisible by 3.
Hence $3\mid |P|$, since $R$ is a $2$-group.
As $P\leq S_3$, either $P=C_3$ or $P=S_3$.
The group $C_3$ is not rational, so $P=S_3$.

Now suppose that $m=2$. We shall first count the elements
of order 3 in $L$ lying above an element of order
3 in $P$.

Let $\bar g\in P$ have order 3. By the Schur--Zassenhaus theorem, there is an element $g\in L$ of order 3 mapping to
$\bar g$.
Under the identification $P\leq S_5$,
the element $\bar g$ is a three-cycle.
Its fixed space on $V_0$ has dimension two, since
after relabelling the coordinates, a fixed vector has
the form
\[
 (a,a,a,b,c),\qquad a+b+c=0.
\]
Therefore $|C_{R/Z(R)}(g)|=4$.

Clearly, the natural map
$C_R(g)\longrightarrow C_{R/Z(R)}(g)$
is onto. Indeed, a $g$-invariant coset of $Z(R)$ contains
two elements, and an element of order 3 cannot
interchange them. Its kernel is $Z(R)$, so
\[
 |C_R(g)|=2\cdot4=8.
\]

Observe that every lift of $\bar g$ of order 3 generates a complement
to $R$ in the inverse image of $\langle\bar g\rangle$.
The conjugacy part of the Schur--Zassenhaus theorem shows that all
such lifts are $R$-conjugate to $g$.
Consequently, $\bar g$ has exactly
\[|R:C_R(g)|=32/8=4\]
lifts of order 3.

Since $P\leq S_5$ and $R$ is a $2$-group, the largest
power of 3 dividing $|L|$ is at most 3.
For every nonzero $w\in W$, choose $t\in L$ with $tw=2w$.
Write $o(t)=3s$, where $3\nmid s$, and choose an integer $a$
such that
\[
 a\equiv1\pmod3,\qquad a\equiv0\pmod s.
\]
Then $g=t^a$ has order 3 and $gw=2w$.
Thus the $2$-eigenspaces of the elements
of $L$ of order 3 cover $W\setminus\{0\}$.

We see that each such eigenspace is totally isotropic. Indeed, if
$gu=2u$ and $gv=2v$, then
\[
 B(u,v)=B(gu,gv)=4B(u,v),
\]
and hence $B(u,v)=0$.
Since $\dim W=4$, each such eigenspace has dimension at most $2$.

Let $n_3(P)$ denote the number of elements of order 3
in $P$. Every element of $L$ of order 3 maps to one
of order 3 in $P$, and each such element of $P$
has exactly four lifts of order 3. Therefore
\[
 7^4-1\leq4n_3(P)(7^2-1),
\]
which gives
$n_3(P)\geq13$. Notice that an intransitive subgroup of $S_5$ has at most eight
elements of order 3, since a subgroup fixing a point
lies in $S_4$, while an intransitive subgroup without
a fixed point lies in $S_3\times S_2$.
Hence $P$ is transitive, and $5\mid |P|$.

Choose $C=\langle c\rangle\leq P$ of order $5$.
Since $P$ is rational, $c$ is conjugate in $P$ to each of the four
generators of $C$.  Hence conjugation by elements of $N_P(C)$ induces
all automorphisms of $C$, and therefore
\[
 N_P(C)/C_P(C)\cong\Aut(C)\cong C_4.
\]
Since $C_{S_5}(C)=C$ and $C\leq P$, we have $C_P(C)=C$.  Thus
\[
 |N_P(C)|=|C|\,|\Aut(C)|=5\cdot4=20.
\]
As $|N_{S_5}(C)|=20$, it follows that
$N_P(C)=N_{S_5}(C)$.  This normalizer contains a $4$-cycle.  Indeed, identify the five points
with $\F_5$ and regard $C$ as the translation group
$C=\{x\mapsto x+a:a\in\F_5\}$.
Then the permutation $x\mapsto2x$ normalizes $C$ and, fixing $0$, acts
on the four nonzero elements as
$(1\,2\,4\,3)$.
Hence $N_P(C)$, and therefore $P$, contains a $4$-cycle.

We also have $3\mid |P|$, so $60\mid |P|$.
Since $P\leq S_5$, its order is $60$ or $120$.
In the former case, $P$ would be an index-two subgroup
of $S_5$, hence equal to $A_5$, contradicting the presence
of a four-cycle. Therefore $P=S_5$.

We have now proved in both cases that
$L/R=O^-_{2m}(2)$.
Together with $L\leq N$ and $N/R\leq O^-_{2m}(2)$,
this implies $L=N$.

Choose $b\in R$ with $b^2=z=-I_W$, and define
$h=2(I_W+b)$,
where the scalar $2$ belongs to $\F_7$.
Since $b$ is an isometry, its adjoint is
$b^\dagger=b^{-1}=-b$. Hence
\[
 h^\dagger h
 =4(I_W-b)(I_W+b)
 =8I_W=I_W,
 \qquad
 h^2=8b=b.
\]
Thus $h$ is an isometry. It also normalizes $R$.
For $r\in R$, either $r$ commutes with $b$, in which
case $hrh^{-1}=r$, or $rb=-br$, in which case
direct multiplication gives $hrh^{-1}=br$.
Therefore
$h\in N=L$.

Let $\rho:L\to\GL_{2^m}(\mathbb C)$ afford the character
$\chi$ introduced above. Choose representatives $r_e$
for the cosets $e\in R/Z(R)$. Their matrices $\rho(r_e)$
form a basis of $\End_{\mathbb C}(\mathbb C^{2^m})$.
To see this, recall that the faithful irreducible
extraspecial character vanishes outside $Z(R)$.
Consequently,
\[
 \operatorname{tr}\bigl(\rho(r_e)\rho(r_f)^{-1}\bigr)
 =
 \begin{cases}
 2^m,&e=f,\\
 0,&e\ne f.
 \end{cases}
\]
The matrices are therefore linearly independent, and
their number $2^{2m}$ equals the dimension of the
endomorphism space.

Conjugation by $\rho(h)$ permutes the one-dimensional spaces
spanned by the basis matrices. The line indexed by $rZ(R)$
is fixed precisely when $r$ commutes with $b$, or equivalently,
when $rZ(R)$ is orthogonal to $bZ(R)$ under the commutator form.
Thus the fixed basis lines are indexed by
$(bZ(R))^\perp$.
Since the commutator form is nondegenerate and $bZ(R)\ne0$,
this orthogonal complement has dimension $2m-1$ over $\F_2$.
Hence exactly $2^{2m-1}$ basis lines are fixed.

If $rZ(R)\in(bZ(R))^\perp$, then $r$ commutes with $b$,
and therefore with $h=2(I_W+b)$.
Consequently, conjugation by $\rho(h)$ fixes the corresponding
basis matrix, so its contribution to the trace is $+1$.
Every other basis line is moved and contributes zero.
It follows that
\[
 \operatorname{tr}\bigl(\operatorname{Ad}\rho(h)\bigr)
 =2^{2m-1}.
\]
On the other hand, the conjugation representation on the
endomorphism space is $\rho\otimes\rho^*$, so
\[
 |\chi(h)|^2
 =\chi(h)\chi(h^{-1})
 =2^{2m-1}.
\]

Since $\chi(z)=-\chi(1)$, condition
$\boldsymbol{(R^-)}$ implies that $\chi(h)\in\mathbb Q$.
But $2^{2m-1}$ is not a square in $\mathbb Q$,
because its exponent of $2$ is odd.
This contradiction proves that $L$ cannot satisfy
{\rm(SF)}.
\end{proof}

The following proposition excludes the extraspecial case for $O_2(L)$
under the additional hypothesis $\boldsymbol{(R^-)}$.

\begin{proposition}
\label{prop:extraspecial-excluded}
Let $(L,W)$ be a $2$-terminal admissible pair over $\F_p$,
where $p\in\{7,11\}$.
Suppose that $z=-I_W\in L$ and that $(L,z)$ satisfies
$\boldsymbol{(R^-)}$.
Then $O_2(L)$ is not extraspecial. Moreover,
\[
 F(L)=O_2(L)=\langle z\rangle
 \qquad\text{and}\qquad
 E(L)\ne1.
\]
\end{proposition}

\begin{proof}
Write $R=O_2(L)$, and suppose that $R$ is extraspecial.
If $p=11$, this contradicts
Proposition~\ref{prop:extraspecial-large-characteristic}.

Suppose therefore that $p=7$, and write
$|R|=2^{1+2m}$, where $m\geq1$.
By Proposition~\ref{prop:extraspecial-irreducible},
$R$ has minus type and acts absolutely irreducibly on $W$.
Since $R\lhd L$, we have
\[
 R\leq L\leq N_{\Sp(W)}(R).
\]
If $m\geq3$, Lemma~\ref{lem:extraspecial-seven-large}
shows that $L$ cannot satisfy {\rm(SF)}.
If $m\in\{1,2\}$, the same conclusion follows from
Lemma~\ref{lem:extraspecial-seven-small}, using
$\boldsymbol{(R^-)}$.
Both conclusions contradict admissibility.
Thus $R$ is not extraspecial.

By
Proposition~\ref{prop:terminal-first-alternatives}, we have
\[
 F(L)=O_2(L)\leq\{I_W,-I_W\},
 \qquad E(L)\ne1.
\]
Since $z=-I_W$ is a central involution of $L$,
the subgroup $\langle z\rangle$ is a normal $2$-subgroup
of $L$ and hence is contained in $O_2(L)$.
Therefore
\[
 F(L)=O_2(L)=\langle z\rangle,
\]
as required.
\end{proof}

\begin{remark}\label{rem:R}
As will be shown in Section~\ref{sec:terminal-contradiction},
the $2$-terminal pairs arising in the proof of the main theorems
satisfy $\boldsymbol{(R^-)}$.  Hence
Proposition~\ref{prop:extraspecial-excluded} applies to these pairs,
and the extraspecial case for $O_2(L)$ is thereby completely excluded.
\end{remark}

\section{The layer, its irreducible constituents, and component subspaces}
\label{sec:layer}

\noindent\textbf{Assumptions.}
Throughout this section, let $p\in\{7,11\}$ and let $(L,W)$
be a $2$-terminal admissible pair over $k=\F_p$,
with symplectic form $B$.
Assume that $z=-I_W\in L$ and that $(L,z)$ satisfies
$\boldsymbol{(R^-)}$.
Set
\[
 E=E(L),\qquad K=\overline{k},\qquad W_K=K\otimes_k W,
\]
and let $B_K$ denote the $K$-bilinear extension of $B$ to $W_K$.
By Proposition~\ref{prop:extraspecial-excluded},
\[
 F(L)=\langle z\rangle,\qquad E\ne1.
\]

We consider the restriction of $W_K$ to the layer $E$. The following
lemma shows that it is a multiplicity-free orthogonal sum of
nontrivial irreducible symplectic modules, permuted transitively by $L$.

\begin{lemma}\label{lem:layer-constituents}
In the notation above, the restriction of $W_K$ to $E$
is a direct sum
\[
 W_K=V_1\perp\cdots\perp V_t
\]
of pairwise nonisomorphic nontrivial irreducible $KE$-modules.
The restriction of $B_K$ to each $V_i$ is nondegenerate.
Moreover, $L$ permutes the subspaces $V_1,\ldots,V_t$
transitively.
\end{lemma}

\begin{proof}
We first show that $W_K$ is semisimple as an $E$-module.
Choose an irreducible $KE$-submodule $U$ of $W_K$.
Since $E\lhd L$, every translate $gU$, with $g\in L$,
is again an irreducible $KE$-submodule.
The sum
\[
 \sum_{g\in L}gU
\]
is a nonzero $KL$-submodule of $W_K$. Therefore the absolute irreducibility of $W$ implies that
this sum is $W_K$. Hence $W_K$ is a sum of simple
submodules, and so is semisimple, which proves the assertion.

Write
\[
 W_K=\bigoplus_{\lambda\in\Lambda}V_\lambda
\]
for its decomposition into $E$-isotypic components.
Thus, for each $\lambda$, there is an irreducible
$KE$-module $X_\lambda$ such that $V_\lambda$ is a direct
sum of copies of $X_\lambda$, and the modules $X_\lambda$
are pairwise nonisomorphic.  Since $E$ is normal in $L$, it follows that $L$ permutes these
isotypic components. The sum of the components in any
one $L$-orbit is an $L$-invariant subspace.
Since $W_K$ is irreducible for $L$, the permutation
action on $\Lambda$ is transitive.

It is easy to see that no $X_\lambda$ is trivial. Indeed, if one $X_\lambda$ is trivial,
then all its $L$-conjugates are also trivial.  Hence $E$ acts trivially on
$W_K$, contrary to $E\ne1$ and faithfulness of the
action. Since $E$ is perfect, every one-dimensional
$KE$-module is trivial. Consequently,
\[
 \dim_K X_\lambda\geq2
 \qquad\text{for every }\lambda\in\Lambda.
\]

We next prove that every $X_\lambda$ is self-dual.
The $E$-invariance of $B_K$ implies that
\[
 B_K(V_\lambda,V_\mu)=0
 \quad\text{unless}\quad
 X_\mu\cong X_\lambda^*.
\]
Indeed, a nonzero pairing would induce a nonzero
$E$-homomorphism from $V_\lambda$ to $V_\mu^*$,
which is possible only when their irreducible
constituents are isomorphic.

Suppose that some $X_\lambda$ is not self-dual. Since the property of self-duality
is preserved under conjugation by $L$,  the transitivity of $L$ on $\Lambda$ implies that none of the
$X_\lambda$ is self-dual. Then the nondegeneracy of $B_K$ pairs the isotypic components
in distinct dual pairs
\[
 \{V_\lambda,V_{\lambda^*}\},
\]
with a nondegenerate pairing between the two members
of each pair, and $L$ permutes these unordered
pairs. By Lemma~\ref{lem:quartic-dual-pairs}, there is a nonzero $L$-invariant homogeneous polynomial
of degree 4 on $W_K$.
This contradicts Corollary~\ref{lem:quartics}.
Hence every $X_\lambda$ is self-dual.

It follows that distinct isotypic components are
orthogonal. Since $B_K$ is nondegenerate on $W_K$,
its restriction to each $V_\lambda$ is nondegenerate.

Fix one isotypic component $V=V_\lambda$, and write
$X=X_\lambda$. Let
$M=\Hom_{KE}(X,V)$.
The evaluation map gives an $E$-module isomorphism
\[
 X\otimes_K M\longrightarrow V,\qquad
 x\otimes f\longmapsto f(x),
\]
where $E$ acts trivially on $M$.
In particular, $\dim_K M$ is the multiplicity of
$X$ in $V$.

Since $X$ is irreducible and self-dual, Schur's lemma
shows that the space of $E$-invariant bilinear forms
on $X$ is one-dimensional. Let $b$ be a nonzero such
form. Then $b$ is nondegenerate and is either symmetric
or alternating.

Under the identification $V=X\otimes_K M$, the form
$B_K|_V$ has the form
$B_K|_V=b\otimes c$
for a bilinear form $c$ on $M$.
To see this, for fixed $u,v\in M$, the map
\[
 (x,y)\longmapsto B_K(x\otimes u,y\otimes v)
\]
is an $E$-invariant bilinear form on $X$, hence is a
scalar multiple of $b$. These scalars define $c$. Moreover, the nondegeneracy of $B_K|_V$ implies that $c$ is
nondegenerate. Since $B_K|_V$ is alternating,
the forms $b$ and $c$ have opposite types.

We claim that $\dim_K M=1$.
Suppose otherwise, and write $H=\{g\in L:gV=V\}$.
For each $g\in H$, the action of $g$ on $V$ has the form
\[
 a_g\otimes h_g,
 \qquad a_g\in\GL_K(X),\quad h_g\in\GL_K(M).
\]
Indeed, $g$ induces an automorphism of $E$ preserving
the isomorphism class of $X$. Choose an intertwiner
$a_g$ implementing this automorphism on $X$.
After composing the action of $g$ with
$a_g^{-1}\otimes I_M$, the resulting map commutes
with $E$ and hence has the form $I_X\otimes h_g$,
by Schur's lemma.

The space of \(E\)-invariant bilinear forms on \(X\) is one-dimensional, so
there is a scalar $\alpha_g\in K^\times$ such that
\[
 b(a_gx,a_gy)=\alpha_g b(x,y).
\]
Since $g$ preserves $b\otimes c$, we also have
\[
 c(h_gu,h_gv)=\alpha_g^{-1}c(u,v).
\]
Thus $H$ acts on $X\otimes_K M$ by reciprocal
similitudes of the two forms.
Both factors have dimension at least 2, and their
forms have opposite types. After interchanging the
factors if necessary, Lemma~\ref{lem:quartic-tensor}
shows that there is a nonzero $H$-invariant homogeneous
polynomial $Q$ of degree 4 on $V$.

We now transport $Q$ to the other isotypic components.
Explicitly, if $V_\mu=gV$, define
\[
 Q_\mu(gv)=Q(v)\qquad(v\in V).
\]
This definition is independent of the choice of $g$,
because $Q$ is $H$-invariant.
The transported polynomials are compatible with the
permutation action of $L$ on the isotypic components.
By Lemma~\ref{lem:quartic-packet-sum}, there is a nonzero
$L$-invariant homogeneous polynomial of degree 4 on $W_K$, again contradicting
Corollary~\ref{lem:quartics}.

Therefore $\dim_K M=1$. It follows that each
isotypic component is irreducible and occurs with
multiplicity one. Its restricted form is already
known to be nondegenerate and alternating.
Relabelling the isotypic components as \(V_1,\ldots,V_t\), we obtain the desired conclusion.
\end{proof}

We also need an invariant quartic for an odd number of alternating
tensor factors. Together with Lemma~\ref{lem:quartic-tensor},
this will show that exactly one component acts nontrivially
on each irreducible constituent of $W_K$.

\begin{lemma}
\label{lem:odd-tensor-quartic}
Let $K$ be a field of characteristic greater than 3,
and let $(X_i,b_i)$, $1\leq i\leq s$, be nonzero symplectic
spaces, where $s\geq3$ is odd. Write
\[
 X=\bigotimes_{i=1}^sX_i,\qquad
 B=\bigotimes_{i=1}^sb_i.
\]
Then $X$ admits a nonzero homogeneous polynomial $Q$
of degree 4 with the following property:

Suppose that $\pi\in S_s$ and that
$a_i:X_i\to X_{\pi(i)}$ are linear isomorphisms satisfying
\[
 b_{\pi(i)}(a_ix,a_iy)=\mu_i b_i(x,y),
 \qquad \mu_i\in K^\times.
\]
Let $g$ be the induced map on $X$, obtained by applying
the maps $a_i$ and placing their images in the positions
specified by $\pi$. Then
\[
 Q(gv)=\left(\prod_{i=1}^s\mu_i\right)^2Q(v).
\]
In particular, $Q$ is invariant under every such $g$
that preserves $B$.
\end{lemma}

\begin{proof}
For each $i$, define 4-linear forms on $X_i$ by
\[
\begin{aligned}
 p_{1,i}(x_1,x_2,x_3,x_4)
   &=b_i(x_1,x_2)b_i(x_3,x_4),\\
 p_{2,i}(x_1,x_2,x_3,x_4)
   &=b_i(x_1,x_3)b_i(x_2,x_4),\\
 p_{3,i}(x_1,x_2,x_3,x_4)
   &=b_i(x_1,x_4)b_i(x_2,x_3).
\end{aligned}
\]
Denote
\[
 d_{1,i}=p_{1,i}-p_{3,i},\qquad
 d_{2,i}=-p_{1,i}-p_{2,i},\qquad
 d_{3,i}=p_{2,i}+p_{3,i}.
\]
Then $d_{1,i}+d_{2,i}+d_{3,i}=0$.

Any permutation of the four vectors $x_1,x_2,x_3,x_4$ merely permutes
these three forms; in particular, no sign changes occur.
Indeed, each of the transpositions $(12)$ and $(34)$
interchanges $d_{1,i}$ and $d_{2,i}$ and fixes $d_{3,i}$,
whereas $(13)$ interchanges $d_{2,i}$ and $d_{3,i}$
and fixes $d_{1,i}$. These transpositions generate $S_4$,
and their action on the three forms is independent of $i$.

We also need to show that $d_{1,i}$ and $d_{2,i}$ are linearly
independent. Choose $e,f\in X_i$ with $b_i(e,f)=1$.
Their values on two quadruples are
\[
\begin{array}{c|cc}
 &d_{1,i}&d_{2,i}\\ \hline
 (e,f,e,f)&2&-1\\
 (e,e,f,f)&-1&-1
\end{array}
\]
and the determinant is $-3\ne0$.
This proves the required independence.

For decomposable vectors
\[
 v_j=x_{1j}\otimes\cdots\otimes x_{sj}
 \qquad(1\leq j\leq4),
\]
define
\[
 T(v_1,v_2,v_3,v_4)
 =
 \sum_{a=1}^3
 \prod_{i=1}^s
 d_{a,i}(x_{i1},x_{i2},x_{i3},x_{i4}),
\]
and extend $T$ multilinearly to $X^4$.
By the simultaneous permutation rule established above, any
permutation of $v_1,v_2,v_3,v_4$ induces the same permutation of the
three summands in the displayed sum.  Hence their sum is unchanged,
and $T$ is a symmetric $4$-linear form on $X$.

We claim that $T\ne0$.
Under the natural rearrangement of tensor factors,
$T$ corresponds to
\[
 \bigotimes_{i=1}^s d_{1,i}
 +\bigotimes_{i=1}^s d_{2,i}
 +\bigotimes_{i=1}^s d_{3,i}.
\]
Using $d_{3,i}=-d_{1,i}-d_{2,i}$ and the oddness of $s$, the last
tensor product becomes
\[
 -\bigotimes_{i=1}^s(d_{1,i}+d_{2,i}).
\]
Consider the elementary tensor
\[
 d_{2,1}\otimes d_{1,2}\otimes\cdots\otimes d_{1,s}.
\]
It occurs exactly once in the expansion of the last tensor product,
by choosing $d_{2,1}$ from the first factor and $d_{1,i}$ from every
remaining factor, and hence occurs there with coefficient $-1$.
It occurs in neither of the first two tensor products
\[
 \bigotimes_{i=1}^s d_{1,i}
 \qquad\text{and}\qquad
 \bigotimes_{i=1}^s d_{2,i}.
\]
Since $d_{1,i}$ and $d_{2,i}$ are linearly independent for every $i$,
the corresponding elementary tensors are linearly independent.
Therefore the displayed tensor has coefficient $-1$ in $T$, and in
particular $T\ne0$.

Define
\[
 Q(v)=T(v,v,v,v).
\]
This is a homogeneous polynomial of degree 4.
It is nonzero by polarization, since
\[
 24T(v_1,v_2,v_3,v_4)
 =
 \sum_{I\subseteq\{1,2,3,4\}}
 (-1)^{4-|I|}
 Q\left(\sum_{j\in I}v_j\right),
\]
and $24$ is invertible in $K$.

Finally, suppose that $g$ acts on the tensor factors through
similitudes $a_i$ with multipliers $\mu_i$.  Since each $p_{a,i}$ is
a product of two values of $b_i$, we have
\[
 p_{a,\pi(i)}(a_ix_1,a_ix_2,a_ix_3,a_ix_4)
   =\mu_i^2p_{a,i}(x_1,x_2,x_3,x_4).
\]
The same identity holds with $p$ replaced by $d$.
Moreover, any permutation of the tensor factors only rearranges the factors within
each product in the definition of $T$, and hence leaves $T$ unchanged.
It follows that
\[
 T(gv_1,gv_2,gv_3,gv_4)
 =
 \left(\prod_{i=1}^s\mu_i\right)^2
 T(v_1,v_2,v_3,v_4).
\]
Thus the quartic $Q(v)=T(v,v,v,v)$ satisfies the asserted
transformation rule.

On the other hand, the tensor-product form $B$ transforms according to
\[
 B(gv,gw)
 =
 \left(\prod_{i=1}^s\mu_i\right)B(v,w).
\]
Hence, if $g$ preserves $B$, then $\prod_{i=1}^s\mu_i=1$.
The preceding transformation rule for $T$ then shows that
$Q(gv)=Q(v)$,
so $Q$ is invariant under every such $g$.
\end{proof}

\begin{remark}
The construction does not involve averaging over
permutations of the tensor factors.
In particular, it remains valid when the characteristic
of $K$ divides $s!$.
\end{remark}

We can now apply this construction to the layer $E=E(L)$ and show
that only one component of $L$ can act nontrivially on each
irreducible constituent of $W_K\!\downarrow_E$.

\begin{lemma}\label{lem:single-component-action}
Let $T_1,\ldots,T_r$ be the components of $L$.
For every irreducible $KE$-submodule $V$ of $W_K$,
there is exactly one index $i$ such that $T_i$
acts nontrivially on $V$.
\end{lemma}

\begin{proof}
By Lemma~\ref{lem:layer-constituents}, the irreducible
$KE$-submodules of $W_K$ are pairwise nonisomorphic,
are nontrivial, and form an orthogonal direct-sum
decomposition. The restriction of $B_K$ to each
of them is nondegenerate.

Fix one such submodule $V$, and write $H=\{g\in L:gV=V\}$.
Since distinct components commute and together generate $E$, the map
\[
 T_1\times\cdots\times T_r\longrightarrow E,\qquad
 (t_1,\ldots,t_r)\longmapsto t_1\cdots t_r,
\]
is a surjective homomorphism.
Inflating $V$ along this epimorphism, we may regard it as an
irreducible $K(T_1\times\cdots\times T_r)$-module.  Hence
\[
 V\cong X_1\otimes\cdots\otimes X_r,
\]
with each $X_i$ an irreducible $KT_i$-module.

Let
\[
 S=\{i:T_i\text{ acts nontrivially on }V\}.
\]
The set $S$ is nonempty, since $V$ is a nontrivial
$E$-module. For $i\notin S$, the corresponding $KT_i$-module $X_i$ is
one-dimensional and trivial. Thus
\[
        V\cong\bigotimes_{i\in S}X_i.
\]
Each $T_i$ is perfect, so a nontrivial irreducible
$KT_i$-module cannot have dimension one. Hence
\[
 \dim_KX_i\geq2\qquad(i\in S).
\]

Since $V$ is self-dual, every $X_i$ is self-dual.
Indeed, the restriction of $V$ to $T_i$ is a direct sum
of copies of $X_i$, whereas the restriction of $V^*$ is
a direct sum of copies of $X_i^*$.
The isomorphism $V\cong V^*$ therefore implies
$X_i\cong X_i^*$.

Choose a nonzero $T_i$-invariant bilinear form $b_i$
on each $X_i$. By Schur's lemma, $b_i$ is nondegenerate
and is either symmetric or alternating.
The product form $\bigotimes_{i\in S}b_i$ is a nonzero
$E$-invariant form on $V$.
Since the invariant form on $V$ is unique up to a nonzero scalar,
$B_K|_V$ is a scalar multiple of $\bigotimes_{i\in S}b_i$.
After rescaling one of the forms $b_i$, we may assume that
\[
 B_K|_V=\bigotimes_{i\in S}b_i.
\]
In particular, the product of the form types is
alternating.

The group $H$ permutes the components indexed by $S$, and hence
permutes the corresponding tensor factors $X_i$.  For $h\in H$,
choose an intertwiner between each $X_i$ and the factor to which it is
sent by $h$.  The tensor product of these intertwiners, followed by
the corresponding permutation of the tensor factors, induces the same
action on $E$ as $h$.  Comparing this map with the action of $h$, we
see that the two maps differ by an endomorphism centralizing the
irreducible action of $E$ on $V$.  By Schur's lemma, this endomorphism
is a scalar.  Absorbing the scalar into one of the intertwiners, we
may therefore realize the action of $h$ as a tensor product of the
chosen intertwiners followed by the corresponding permutation of the
factors.

Since each $b_i$ is unique up to scalars, these intertwiners are
similitudes between the corresponding forms.  If their multipliers
are $\mu_i$, then the multiplier on
\[
 B_K|_V=\bigotimes_{i\in S}b_i
\]
is $\prod_{i\in S}\mu_i$.  Since $H$ preserves $B_K|_V$, it follows
that  \[\prod_{i\in S}\mu_i=1.\]

Suppose first that $H$ has at least two orbits on $S$.
Choose one orbit $S_1$, and write $S_2=S\setminus S_1$.
Then
\[
 V=Y_1\otimes Y_2,\qquad
 Y_j=\bigotimes_{i\in S_j}X_i.
\]
Both $Y_1$ and $Y_2$ have dimension at least two.
Their product forms have opposite types, since their
tensor product is $B_K|_V$, which is alternating.
Moreover, $H$ preserves this two-factor decomposition
and acts by reciprocal similitudes. By Lemma~\ref{lem:quartic-tensor},
$H$ admits a nonzero invariant homogeneous polynomial of
degree 4 on $V$.

Suppose instead that $H$ is transitive on $S$
and that $|S|>1$. Since conjugate factors have the same
form type, all the forms $b_i$ with $i\in S$
have the same type. They cannot all be symmetric, because their product
is alternating. Thus they are all alternating,
and $|S|$ is odd. Since $|S|>1$, we have $|S|\geq3$. Again by Lemma~\ref{lem:odd-tensor-quartic},
$H$ admits a
nonzero invariant homogeneous polynomial of
degree 4 on $V$.

In either case, we may transport this polynomial to each
$L$-conjugate of $V$; the resulting polynomial is well defined because
the original one is $H$-invariant.  By
Lemma~\ref{lem:layer-constituents}, these conjugate subspaces are
precisely the irreducible $E$-summands of $W_K$.
Lemma~\ref{lem:quartic-packet-sum} therefore yields a nonzero
$L$-invariant homogeneous polynomial of degree $4$ on $W_K$,
contrary to Corollary~\ref{lem:quartics}.

Thus $|S|=1$, as required.
\end{proof}

The final proposition of this section collects the constituents
associated with each component and descends the resulting subspaces to
the original field.  The resulting orthogonal decomposition allows us
to study the components separately while retaining the action of $L$
permuting the corresponding subspaces.

\begin{proposition}
\label{prop:component-decomposition}
Let $T_1,\ldots,T_r$ be the components of $L$, and set
\[
 U_i=[T_i,W]
 =
 \operatorname{span}_{k}
 \{tw-w:t\in T_i,\ w\in W\},
 \qquad k=\F_p.
\]
Then the following statements hold.
\begin{enumerate}
\item
Each $U_i$ is nonzero, $
 W=U_1\perp\cdots\perp U_r$, and the restriction of $B$ to each $U_i$ is nondegenerate.

\item
The component $T_i$ acts faithfully on $U_i$,
and every $T_j$ with $j\ne i$ acts trivially on $U_i$.

\item
The group $L$ permutes the components transitively.
Moreover, for $g\in L$,
\[
 gT_i g^{-1}=T_j
 \quad\Longrightarrow\quad
 gU_i=U_j.
\]

\item The layer is the direct product $E(L)=T_1\times\cdots\times T_r$.
\end{enumerate}
\end{proposition}

\begin{proof}
Denote $E=E(L)$ and $K=\overline{k}$.
By Lemma~\ref{lem:layer-constituents}, we have
\[
 W_K=V_1\perp\cdots\perp V_t,
\]
where the $V_a$ are pairwise nonisomorphic nontrivial
irreducible $KE$-modules and the restriction of $B_K$
to each $V_a$ is nondegenerate.
By Lemma~\ref{lem:single-component-action}, exactly
one component acts nontrivially on each $V_a$.

For each $i$, define
\[
 \widetilde U_i
 =
 \bigoplus_{\substack{1\leq a\leq t\\
 T_i\text{ acts nontrivially on }V_a}}V_a.
\]
Every summand $V_a$ occurs in exactly one of these
subspaces, so
\[
 W_K=\widetilde U_1\perp\cdots\perp\widetilde U_r.
\]
Each $\widetilde U_i$ is nonzero.
Indeed, if $\widetilde U_i=0$, then $T_i$ acts trivially
on every $V_a$, and hence on $W_K$.
This contradicts the faithfulness of the action.
Each $\widetilde U_i$ is also nondegenerate, being
an orthogonal sum of nondegenerate subspaces.

We next identify these subspaces over $k$.
Since $T_i\lhd E$, the subspace $[T_i,V_a]$ is
$E$-invariant. Explicitly, for $e\in E$, $t\in T_i$,
and $v\in V_a$,
\[
 e(tv-v)=(ete^{-1})(ev)-ev\in[T_i,V_a].
\]
If $T_i$ acts trivially on $V_a$, then $[T_i,V_a]=0$.
Otherwise, $[T_i,V_a]$ is a nonzero $E$-submodule
of the irreducible module $V_a$, so
$[T_i,V_a]=V_a$.
Consequently,
$[T_i,W_K]=\widetilde U_i$.

On the other hand,
\[
 U_i=[T_i,W]=\sum_{t\in T_i}(t-I_W)W.
\]
Extension of scalars commutes with these images
and their finite sum. Therefore, under the natural
identification inside $W_K$,
\[
 K\otimes_k U_i
 =
 \sum_{t\in T_i}(t-I_{W_K})W_K
 =
 [T_i,W_K]
 =
 \widetilde U_i.
\]
In particular, each $U_i$ is nonzero.

Consider the addition map
\[
 \bigoplus_{i=1}^r U_i\longrightarrow W.
\]
After extending scalars to $K$, it becomes the
isomorphism
\[
 \bigoplus_{i=1}^r\widetilde U_i\longrightarrow W_K.
\]
Since extension of scalars preserves the rank of
a linear map, the original addition map is also
an isomorphism. Thus
$W=U_1\oplus\cdots\oplus U_r$.
The orthogonality of the $\widetilde U_i$ implies
the orthogonality of the $U_i$.
If $u$ lies in the radical of $B|_{U_i}$, then $u$ is orthogonal to
$U_i$.  Since the subspaces $U_j$ are mutually orthogonal, it is also
orthogonal to every $U_j$ with $j\ne i$, and hence to all of $W$.
The nondegeneracy of $B$ therefore shows that $u=0$.  Thus $B|_{U_i}$ is
nondegenerate, proving~{\rm(1)}.

If $j\ne i$, then $T_j$ acts trivially on each
irreducible summand contained in $\widetilde U_i$.
Hence it acts trivially on $\widetilde U_i$ and
therefore on $U_i$.
Now let $t\in T_i$ act trivially on $U_i$.
The same observation shows that $T_i$ acts trivially
on every $U_j$ with $j\ne i$.
Thus $t$ acts trivially on all of $W$.
Since the action of $L$ on $W$ is faithful, it follows that $t=1$.
This proves~{\rm(2)}.

For $g\in L$, we have $g(tw-w)=(gtg^{-1})(gw)-gw$, so by the definition of the commutator subspace,
\[
 g[T_i,W]=[gT_i g^{-1},W].
\]
Thus conjugation of components corresponds to
permutation of their associated subspaces.

Suppose that the components have more than one
$L$-conjugacy class. Choose one such class, indexed
by a nonempty proper subset
$I\subsetneq\{1,\ldots,r\}$.
Then $\bigoplus_{i\in I}U_i$
is a nonzero proper $L$-invariant subspace of $W$.
This contradicts the irreducibility of $W$.
Hence the components form one conjugacy class,
proving (3).

Finally, since distinct components commute, multiplication
defines a surjective homomorphism
\[
 T_1\times\cdots\times T_r\longrightarrow E.
\]
Suppose that $t_1\cdots t_r=1$, where $t_i\in T_i$.
On $U_i$, all factors except $t_i$ act trivially by~{\rm(2)}.
Hence $t_i$ acts trivially on $U_i$.  Since $T_i$ acts faithfully
on $U_i$, we have $t_i=1$ for every $i$.
 Thus the multiplication
map is injective, and hence is an isomorphism, proving~{\rm(4)}.
\end{proof}

\section{Local actions of components}
\label{sec:local-actions}

We now pass to the stabilizer of a component subspace and consider its
faithful induced action.  The following lemma records the properties
inherited by this action and identifies an elementary abelian
centralizer whose eigenspaces will provide the next local reduction.

\begin{lemma}
\label{lem:component-subspace-action}
Let $T=T_i$ be a component of $L$, with associated subspace
$U=U_i=[T,W]$ as in Proposition~\ref{prop:component-decomposition},
and set $H=N_L(T)$.
Let \[\pi:H\longrightarrow\GL(U)\] be the homomorphism induced by restriction to $U$, and set
$J=\pi(H)$ and $\overline T=\pi(T)$.
Then the following statements hold.
\begin{enumerate}
\item
The group $H$ is the setwise stabilizer of $U$ in $L$.
The group $J\leq\Sp(U)$ acts faithfully and irreducibly
on $U$.

\item
The restriction $\pi|_T$ is injective.
In particular, $\overline T\cong T$ is a quasisimple
normal subgroup of $J$.

\item
The action of $J$ on $U$ satisfies {\rm(SF)}.

\item
For every $p'$-element $t\in T$,
\[
 \BrTr_U(\pi(t))
 =
 \BrTr_W(t)-(\dim W-\dim U)\in\mathbb Z.
\]

\item
The centralizer $C_J(\overline T)$ is a normal
elementary abelian $2$-subgroup of $J$.
\end{enumerate}
\end{lemma}

\begin{proof}
By Proposition~\ref{prop:component-decomposition},
conjugation of components corresponds to permutation
of their associated subspaces via \[
 g[T,W]=[gTg^{-1},W]\qquad(g\in L).
\]
Thus $H=N_L(T)$ stabilizes $U$.
Conversely, suppose that $gU=U$.
Write $gTg^{-1}=T_j$. Then $gU=U_j$, so $U_j=U_i$.
The subspaces in the component decomposition are
nonzero and form a direct sum; hence $j=i$.
Therefore $g\in N_L(T)$.
This proves that $H$ is exactly the setwise stabilizer
of $U$.

Every element of $H$ preserves the restriction of $B$
to $U$, which is nondegenerate.
Consequently, $J\leq\Sp(U)$, and the action of $J$
is faithful by its definition as an induced group.

For the irreducibility of $U$, suppose that $0\ne A<U$ is an $H$-invariant
$k$-subspace, where $k=\F_p$.
Since $L$ permutes the $r$ component subspaces
transitively, choose $g_1,\ldots,g_r\in L$ such that
$g_jU=U_j$.
The subspace
\[
 A_0=\bigoplus_{j=1}^r g_jA
\]
is $L$-invariant. Indeed, if $g\in L$ and
$gg_jU=g_\ell U$, then
$g_\ell^{-1}gg_j\in H$, and so $gg_jA=g_\ell A$.
Moreover,
\[
 0<\dim A_0=r\dim A<r\dim U=\dim W.
\]
This contradicts the irreducibility of $W$.
Hence $J$ acts irreducibly on $U$, proving (1).

By Proposition~\ref{prop:component-decomposition},
$T$ acts faithfully on $U$, so $\pi|_T$ is injective.
Also $T\lhd H$, and therefore $\overline T\lhd J$.
Since $T$ is quasisimple, so is $\overline T$.
This proves (2).

To prove (3), let $0\ne u\in U$ and $a\in k^\times$.
By {\rm(SF)} for $L$ on $W$, there is a $p'$-element
$g\in L$ with $gu=au$.
Since $g$ permutes the component subspaces, $gu$
belongs to one of the $U_j$.
But $gu=au$ is a nonzero vector of $U$.
The directness of the component decomposition forces
$gU=U$, so $g\in H$.
Its image $\pi(g)$ is a $p'$-element of $J$ satisfying
$\pi(g)u=au$.
Thus $J$ satisfies {\rm(SF)} on $U$.

For (4), let $t\in T$ be a $p'$-element.
The component $T$ acts trivially on every $U_j$
other than $U$. Therefore
\[
 \BrTr_W(t)
 =
 \BrTr_U(\pi(t))
 +\sum_{j\ne i}\dim U_j.
\]
Condition {\rm(R)} for $(L,W)$ implies that
$\BrTr_W(t)\in\mathbb Z$, and hence
\[
 \BrTr_U(\pi(t))
 =
 \BrTr_W(t)-(\dim W-\dim U)\in\mathbb Z.
\]

Finally, write  $C=C_J(\overline T)$.
Since $\overline T\lhd J$, its centralizer $C$
is normal in $J$.
Extend scalars to $K=\overline{k}$.
By the preceding results, $U_K$ is an orthogonal sum
\[
 U_K=V_1\perp\cdots\perp V_s
\]
of pairwise nonisomorphic irreducible $K\overline T$-modules,
each carrying the nondegenerate alternating form
obtained by restricting $B_K$.
Indeed, these are exactly the irreducible $KE(L)$-summands
associated with $T$; all other components act trivially
on them.

Let $c\in C$. Since $c$ commutes with $\overline T$,
it preserves each $V_a$ and acts on it by a scalar
$\varepsilon_a(c)\in K^\times$, by Schur's lemma.
Because $c$ is an isometry,
\[
 B_K(cx,cy)=B_K(x,y)
 \qquad(x,y\in V_a).
\]
By the nondegeneracy of $B_K|_{V_a}$, we obtain $\varepsilon_a(c)^2=1$.
Thus $\varepsilon_a(c)\in\{1,-1\}$.
Consider the homomorphism
\[
 C\longrightarrow\{1,-1\}^s,\qquad
 c\longmapsto
 \bigl(\varepsilon_1(c),\ldots,\varepsilon_s(c)\bigr).
\]
If $c$ lies in its kernel, then $c$ acts trivially on every summand of
$U_K$, and hence on $U$.  Since $J$ acts faithfully on $U$, we have
$c=1$.  Thus the homomorphism is injective, and $C$ is an elementary
abelian $2$-group.  This proves~{\rm(5)}.
\end{proof}

Before refining the action on a component subspace, we consider
the case in which a component has nontrivial center.
In this case, there is only one component, and the whole group
already has an almost simple quotient by the scalar signs.

\begin{lemma}
\label{lem:unique-centered-component}
Suppose that a component $T$ of $L$ has nontrivial center.
Then $T$ is the unique component of $L$, and
\[
 E(L)=T,\qquad
 Z(T)=C_L(T)=\langle z\rangle.
\]
Moreover, conjugation induces
an embedding $S\leq L/\langle z\rangle\leq\Aut(S)$,
where $S=T/Z(T)$.
\end{lemma}

\begin{proof}
Since distinct components commute, every element
of $Z(T)$ commutes with every component of $L$.
Consequently, $Z(T)\leq Z(E(L))$.
The subgroup $Z(E(L))$ is characteristic in $E(L)$
and hence normal in $L$. Therefore
\[
 Z(T)\leq Z(E(L))\leq F(L)=\langle z\rangle.
\]
Since $Z(T)\ne1$, it follows that
$Z(T)=\langle z\rangle$.

Suppose that $T'$ is another component.
By Proposition~\ref{prop:component-decomposition},
the associated subspace
\[
 U'=[T',W]
\]
is nonzero, and $T$ acts trivially on $U'$.
In particular, $z\in T$ acts trivially on $U'$.
But $z=-I_W$, so it acts as $-I_{U'}$.
Since the characteristic is odd, this is impossible.
Thus $T$ is the unique component, and
$E(L)=T$. In particular, $T\lhd L$.

The component decomposition now has only one summand,
so
$[T,W]=W$.
Applying Lemma~\ref{lem:component-subspace-action}
with $U=W$ and $N_L(T)=L$, we see that $C_L(T)$
is a normal elementary abelian $2$-subgroup of $L$.
Hence
\[
 C_L(T)\leq O_2(L)=\langle z\rangle.
\]
The reverse inclusion holds because $z$ is central
in $L$. Therefore
$C_L(T)=\langle z\rangle=Z(T)$.

It remains to identify the projective action of $L$.
Since $T\lhd L$, conjugation induces a homomorphism
\[
 \alpha:L\longrightarrow\Aut(S),
 \qquad S=T/Z(T).
\]
We claim that $\ker\alpha=C_L(T)$. Clearly, $C_L(T)\leq\ker\alpha$.
Conversely, suppose that $g\in\ker\alpha$.
For every $t\in T$, there is $c(t)\in Z(T)$ such that
$gtg^{-1}=t\,c(t)$.
Since the elements $c(t)$ are central, we have $c(tu)=c(t)c(u)$ for $t,u\in T$.
Thus $c:T\to Z(T)$ is a homomorphism.
The group $T$ is perfect and $Z(T)$ is abelian,
so $c$ is trivial.
Therefore $g$ centralizes $T$, proving the claim.

It follows that
$L/\langle z\rangle\hookrightarrow\Aut(S)$.
The image of $T$ under $\alpha$ is the group of
inner automorphisms of $S$, naturally isomorphic
to $S$. Hence
\[
 S\leq L/\langle z\rangle\leq\Aut(S),
\]
as required.
\end{proof}

We now decompose $U$ into common eigenspaces of
$C_J(\overline T)$.  On each such eigenspace, the resulting local
action has scalar centralizer, and hence its quotient by the scalar
signs is almost simple.  This connects the local symplectic action
with the automorphism group of the simple quotient of the component.

\begin{proposition}
\label{prop:local-component-action}
Let $J$ and $\overline T$ be as in
Lemma~\ref{lem:component-subspace-action}, and set
$C=C_J(\overline T)$.
Let $Y$ be a nonzero common eigenspace of $C$ on $U$,
and let $J_Y$ be the group induced on $Y$ by its
setwise stabilizer in $J$.
Write $T_Y$ for the group induced by $\overline T$ on $Y$.
Then the following statements hold.
\begin{enumerate}
\item
The restriction of $B$ to $Y$ is nondegenerate,
and $J_Y\leq\Sp(Y)$ acts faithfully and irreducibly
on $Y$. Its action satisfies {\rm(SF)}.

\item
The action of $\overline T$ on $Y$ is faithful.
In particular, $T_Y\cong T$ is a quasisimple normal
subgroup of $J_Y$.

\item
We have $C_{J_Y}(T_Y)=\langle-I_Y\rangle$. If $S=T/Z(T)$, then conjugation induces an embedding
$S\leq J_Y/\langle-I_Y\rangle\leq\Aut(S)$.

\item
For every $p'$-element $t\in T$, the Brauer trace of
its action on $Y$ is an integer.
\end{enumerate}
\end{proposition}

\begin{proof}
By Lemma~\ref{lem:component-subspace-action},
$C$ is a normal elementary abelian $2$-subgroup of $J$.
Since $p$ is odd, the action of $C$ is simultaneously
diagonalizable over $k=\F_p$, with eigenvalues in
$\{1,-1\}$. Thus
\[
 U=\bigoplus_{\lambda\in\Lambda}Y_\lambda,
 \qquad
 Y_\lambda=\{u\in U:cu=\lambda(c)u
                 \text{ for every }c\in C\},
\]
where $\Lambda$ consists of the homomorphisms
$\lambda:C\to\{1,-1\}$ with $Y_\lambda\ne0$.
Fix $\lambda\in\Lambda$ and write $Y=Y_\lambda$.

We claim that distinct eigenspaces are orthogonal.
Indeed, if $\lambda\ne\mu$, choose $c\in C$ such that
$\lambda(c)\ne\mu(c)$. For $u\in Y_\lambda$ and
$v\in Y_\mu$, we have
\[
 B(u,v)=B(cu,cv)
       =\lambda(c)\mu(c)B(u,v)
       =-B(u,v).
\]
Hence $B(u,v)=0$, proving the claim.
Since $B|_U$ is nondegenerate, its restriction to
each $Y_\lambda$ is nondegenerate.

Observe that the normality of $C$ implies that $J$ permutes its
nonzero common eigenspaces.
This permutation action is transitive, since the sum
of the eigenspaces in any orbit is a $J$-invariant
subspace of the irreducible module $U$.

Let  $D=\{j\in J:jY=Y\}$,
and let $\rho:D\to\GL(Y)$ be restriction, so that
$J_Y=\rho(D)$.
The action of $D$ on $Y$ is irreducible.
Indeed, if $0\ne A<Y$ is $D$-invariant, then the $J$-conjugates of $A$ lie in the corresponding common
eigenspaces, and their sum is a nonzero proper $J$-invariant subspace
of $U$, contradicting the irreducibility of $U$.
Hence the induced
action of $J_Y$ on $Y$ is irreducible.  It is faithful by definition,
and, since it preserves the restriction of the symplectic form,
\[
 J_Y\leq\Sp(Y).
\]

Let $0\ne y\in Y$ and $a\in k^\times$. By the scalar fusion on $U$, there is a $p'$-element $j\in J$
such that $jy=ay$.
Because $J$ permutes the common eigenspaces and
$jy$ is a nonzero vector of $Y$, we have $jY=Y$.
Thus $\rho(j)\in J_Y$ is a $p'$-element sending
$y$ to $ay$. This proves (1).

Since $\overline T$ centralizes $C$, it preserves every common
eigenspace $Y_\lambda$.  Hence $\overline T\leq D$, and therefore
\[
 T_Y=\rho(\overline T)\unlhd J_Y.
\]
It remains to show that $\rho|_{\overline T}$ is injective.  We first
note that
\[
 Z(T)\leq Z(E(L))\leq F(L)=\langle z\rangle.
\]
Indeed, distinct components commute, so $Z(T)\leq Z(E(L))$, while
$Z(E(L))$ is an abelian normal subgroup of $L$ and hence lies in
$F(L)$.  Thus $Z(T)$ is either trivial or $\langle z\rangle$.

The action of $\overline T$ on $Y$ is nontrivial.  To see this, extend
scalars to $K$.  The space $U_K$ is a multiplicity-free sum of
nontrivial irreducible $K\overline T$-modules and therefore has no
nonzero $\overline T$-fixed vectors.
Since $0\ne Y_K\subseteq U_K$, the action of $\overline T$
on $Y$ cannot be trivial.

The kernel of $\overline T$ on $Y$ is therefore a proper normal
subgroup of the quasisimple group $\overline T$, and so it is contained
in $Z(\overline T)$.  By the preceding observation, the only possible
nontrivial central element is the image of $z$, and this acts on $Y$
as $-I_Y$.  It cannot lie in the kernel.  Hence the kernel is trivial,
so $\rho|_{\overline T}$ is injective.  This proves~{\rm(2)}.

We now determine the centralizer of $T_Y$ in $J_Y$.
Let $\rho(d)\in C_{J_Y}(T_Y)$, with $d\in D$.
For each $t\in\overline T$, the commutator $[d,t]$ lies in
$\overline T$.  Since $\rho(d)$ centralizes $T_Y=\rho(\overline T)$,
the element $[d,t]$ acts trivially on $Y$.  The faithfulness of
$\overline T$ on $Y$, proved above, therefore shows that $[d,t]=1$.
Thus $d\in C_J(\overline T)=C$.  By the definition of the common
eigenspace $Y$, every element of $C$ acts on $Y$ as a scalar in
$\{1,-1\}$.  Hence $C_{J_Y}(T_Y)\leq\{I_Y,-I_Y\}$.
On the other hand, $z=-I_W$ induces $-I_U\in C$, and hence
$-I_Y\in J_Y$.  Therefore
\[
 C_{J_Y}(T_Y)=\langle -I_Y\rangle.
\]

Identify $T_Y$ with $T$, and write $S=T/Z(T)$.
We have a homomorphism
$J_Y\longrightarrow\Aut(S)$ via conjugation.
Suppose that $a\in J_Y$ lies in its kernel.
Then for every $t\in T_Y$,
$ata^{-1}=t\,c(t)$ for some $c(t)\in Z(T_Y)$.
Since the elements $c(t)$ are central,
\[
\begin{aligned}
 a(tu)a^{-1}
   &=(ata^{-1})(aua^{-1})\\
   &=t\,c(t)\,u\,c(u)
    =tu\,c(t)c(u),
\end{aligned}
\]
and hence $c(tu)=c(t)c(u)$.  Thus $t\mapsto c(t)$ is a homomorphism
from $T_Y$ to the abelian group $Z(T_Y)$.  Since $T_Y$ is perfect,
this homomorphism is trivial.
Hence $a$ centralizes $T_Y$.
It follows that the kernel of the conjugation map
is precisely $\langle-I_Y\rangle$.

The image of $T_Y$ under conjugation is ${\rm Inn}(S)\cong S$.
Consequently,
\[
S\leq J_Y/\langle-I_Y\rangle\leq\Aut(S),
\]
which proves~{\rm(3)}.

It remains to prove the Brauer trace assertion.
Fix a $p'$-element $t\in T$, and write
$\overline t=\pi(t)\in\overline T$.
For each $c\in C$, choose a preimage $h\in H$
under $\pi$.
Since $c$ centralizes $\overline T$, for every
$x\in T$ the commutator $[h,x]$ belongs to $T$
and acts trivially on $U$.
Since the action of $T$ on $U$ is faithful, we have
$[h,x]=1$. Hence $h$ centralizes $T$.

Let $\widetilde c$ be the $p'$-part of $h$.
Since $\pi(h)=c$ has order at most 2 and $p$ is odd, we have $\pi(\widetilde c)=c$.
Also $\widetilde c$ centralizes $T$, and so $t\widetilde c$ is a $p'$-element.

Write
\[
 U'=\bigoplus_{j\ne i}U_j.
\]
The elements $\widetilde c$ and $t\widetilde c$
preserve both $U$ and $U'$, and they have the same
action on $U'$, since $T$ acts trivially there.
Therefore
\[
 \BrTr_U(\overline t c)-\BrTr_U(c)
 =
 \BrTr_W(t\widetilde c)-\BrTr_W(\widetilde c)
 \in\mathbb Z.
\]

Here the integrality follows from {\rm(R)} on $W$.
Since $c$ has only the eigenvalues $1$ and $-1$,
its Brauer trace on $U$ is an integer. Hence $\BrTr_U(\overline t c)\in\mathbb Z$
for any $c\in C$.

Note that the group $C\langle\overline t\rangle$ has order
prime to $p$, since $C$ centralizes $\overline t$.
Applying the ordinary character projector for the
sign character $\lambda$ yields
\[
 \BrTr_Y(t)
 =
 \frac1{|C|}
 \sum_{c\in C}\lambda(c)\BrTr_U(\overline t c).
\]
Here $\BrTr_Y(t)$ denotes the Brauer trace of the
action induced by $t$ on $Y$.
The right-hand side is rational, since every
$\lambda(c)$ equals $1$ or $-1$ and all the traces
in the sum are integers. The left-hand side is a sum of roots
of unity and hence an algebraic integer. Therefore
\[\BrTr_Y(t)\in\mathbb Z.
\]
This proves (4).
\end{proof}

The following lemma gives a uniform lower bound on the dimension of the
local space $Y$, which will be used in the counting argument below.

\begin{lemma}\label{lem:local-dimension} Let $Y$ be a local symplectic space obtained in
Proposition~\ref{prop:local-component-action}. Then
\[
 \dim_{\F_p}Y\geq4.
\]
\end{lemma}

\begin{proof}
Since $Y$ is a nonzero symplectic space, its dimension
is a positive even integer. Suppose, for a contradiction,
that $\dim_{\F_p}Y=2$.

Identify the faithfully acting component $T_Y$ with $T$.
Then
\[
 T\leq\Sp(Y)=\SL_2(p).
\]
For every $p'$-element $t\in T$, its Brauer trace on $Y$
is an integer, by
Proposition~\ref{prop:local-component-action}.

We first determine the possible orders of such elements.
Let $t\in T$ be a $p'$-element.
It is diagonalizable over $\overline{\F}_p$, and its
eigenvalues are $\lambda,\lambda^{-1}$, since its
determinant is one. If $\zeta$ is the complex
Teichm\"uller lift of $\lambda$, then
\[
 \BrTr_Y(t)=\zeta+\zeta^{-1}\in\mathbb Z.
\]
This is a real number in $[-2,2]$, so
$\zeta+\zeta^{-1}\in\{-2,-1,0,1,2\}$.
The corresponding equations
\[
 \zeta^2-(\zeta+\zeta^{-1})\zeta+1=0
\]
show that the corresponding orders of $\zeta$ are
$2,3,4,6$, and $1$, respectively.
Since Teichm\"uller lifting preserves orders, and
$t$ is diagonalizable with eigenvalues
$\lambda,\lambda^{-1}$, it follows that
\begin{equation}\label{eq:pprime-orders-two-dim}
o(t)\in\{1,2,3,4,6\}.
\end{equation}

Suppose first that $p\nmid |T|$.
By Cauchy's theorem and \eqref{eq:pprime-orders-two-dim}, every prime divisor
of $|T|$ is either 2 or 3.
Burnside's $p^aq^b$ theorem then says that $T$
is solvable, contradicting the fact that $T$
is quasisimple.

We may therefore assume that $p\mid |T|$.
Choose a subgroup $P\leq T$ of order $p$.
A nonidentity element of $P$ is a nontrivial unipotent
matrix on the two-dimensional space $Y$.
After choosing a suitable basis, we have
\[
 P=
 \left\{
 u(a)=
 \begin{pmatrix}1&a\\0&1\end{pmatrix}
 :a\in\F_p
 \right\}.
\]
Indeed, the powers of any nonidentity upper unipotent
matrix give all these matrices because the field is
the prime field $\F_p$.
The unique line fixed by $P$ is
\[
 \ell=\langle e_1\rangle,\qquad
 e_1=(1,0)^T.
\]

The group $T$ does not stabilize $\ell$.
Otherwise $T$ would lie in the upper triangular
subgroup of $\SL_2(p)$, which is solvable, again
contradicting quasisimplicity.
Choose $g\in T$ such that $g\ell\ne\ell$.
The subgroup $gPg^{-1}$ has order $p$ and fixes
the line $g\ell$.

The group $P$ acts transitively on the $p$ lines
different from $\ell$. These lines are
\[
 \langle(x,1)^T\rangle,\qquad x\in\F_p,
\]
and
\[
 u(a)\langle(x,1)^T\rangle
   =\langle(x+a,1)^T\rangle.
\]
Hence some $u\in P$ sends $g\ell$ to
$\langle e_2\rangle$, where $e_2=(0,1)^T$.
It follows that $T$ contains a subgroup of order $p$
fixing $\langle e_2\rangle$.

The stabilizer of $\langle e_2\rangle$ in $\SL_2(p)$
consists of the matrices
\[
 \begin{pmatrix}a&0\\b&a^{-1}\end{pmatrix},
 \qquad a\in\F_p^\times,\quad b\in\F_p.
\]
Its unique subgroup of order $p$ is
\[
 P^-=
 \left\{
 \begin{pmatrix}1&0\\b&1\end{pmatrix}
 :b\in\F_p
 \right\}.
\]
Thus $P^-\leq T$.
The upper and lower elementary matrices generate
$\SL_2(p)$, so
$T=\SL_2(p)$.

Finally, regard $\F_{p^2}$ as a two-dimensional vector space over
$\F_p$.  The norm-one subgroup
\[
 \{\alpha\in\F_{p^2}^{\times}:\alpha^{p+1}=1\}
\]
is cyclic of order $p+1$.  Choose a generator $\alpha$.
Multiplication by $\alpha$ defines an $\F_p$-linear transformation of
order $p+1$, whose determinant is
\[
 N_{\F_{p^2}/\F_p}(\alpha)=\alpha^{p+1}=1.
\]
After choosing an $\F_p$-basis of $\F_{p^2}$, this transformation
therefore gives an element of $\SL_2(p)=T$ of order $p+1$.

For $p=7$ and $11$, this order is $8$ and $12$, respectively.
In either case it is prime to $p$ and is not among the orders allowed
by \eqref{eq:pprime-orders-two-dim}, a contradiction.

Therefore $\dim_{\F_p}Y\ne2$, and hence
$\dim_{\F_p}Y\geq4$.
\end{proof}

We next use scalar fusion and the symplectic form to obtain a lower
bound on the order of the acting group modulo scalar signs.

\begin{lemma}
\label{lem:projective-count}
Let $(Y,B)$ be a nonzero symplectic space of dimension
$2n$ over $\F_p$, where $p\geq7$, and let
$J\leq\Sp(Y,B)$ satisfy {\rm(SF)}.
Write $Z_0=J\cap\{I_Y,-I_Y\}$.  Then $p^n+1\leq |J/Z_0|$.
\end{lemma}

\begin{proof}
For $g\in J$ and $a\in\F_p$, write
$E_a(g)=\{v\in Y:gv=av\}$.
If $u,v\in E_2(g)$, then
\[
 B(u,v)=B(gu,gv)=4B(u,v).
\]
Since $p\ne3$, this implies $B(u,v)=0$.
Thus $E_2(g)$ is totally isotropic and
$\dim E_2(g)\leq n$.
In particular, $E_2(g)$ contains at most $p^n-1$
nonzero vectors.

Suppose first that $-I_Y\notin J$.
For every nonzero $v\in Y$, condition {\rm(SF)},
applied to the scalar $2$, shows that there is an element $g\in J$
such that $gv=2v$.
Consequently,
\[
 Y\setminus\{0\}
 =
 \bigcup_{g\in J}\bigl(E_2(g)\setminus\{0\}\bigr).
\]
Counting this union gives
$p^{2n}-1\leq |J|(p^n-1)$.
Here $Z_0=1$, so division by $p^n-1$ proves
the required inequality.

Now suppose that $-I_Y\in J$.
We count the contributions of $g$ and $-g$ together.
Since
\[
 E_2(-g)=E_{-2}(g),
\]
the nonzero vectors covered by the coset
$gZ_0=\{g,-g\}$ form the set
\[
 \bigl(E_2(g)\cup E_{-2}(g)\bigr)\setminus\{0\}.
\]
Write $r=\dim E_2(g)$ and $s=\dim E_{-2}(g)$.
Both eigenspaces are totally isotropic.
They are also orthogonal to one another:
if $u\in E_2(g)$ and $v\in E_{-2}(g)$, then
\[
 B(u,v)=B(gu,gv)=-4B(u,v),
\]
so $B(u,v)=0$, since $p\ne5$.
Moreover, the eigenvalues $2$ and $-2$ are distinct,
so the two eigenspaces intersect trivially.
Their direct sum is therefore totally isotropic,
and hence
$r+s\leq n$.

The number of nonzero vectors covered by this coset
is consequently
\[
 (p^r-1)+(p^s-1)
 =p^r+p^s-2
 \leq p^{r+s}-1
 \leq p^n-1.
\]
The first inequality follows from
\[
 p^{r+s}-1-(p^r+p^s-2)
 =(p^r-1)(p^s-1)\geq0.
\]

By {\rm(SF)}, the sets associated with all cosets
of $Z_0$ cover $Y\setminus\{0\}$.
Thus
\[
 p^{2n}-1\leq |J/Z_0|(p^n-1).
\]
Dividing by $p^n-1$ gives
$p^n+1\leq |J/Z_0|$,
as required.
\end{proof}

An application of Lemma~\ref{lem:projective-count} to $J_Y$ yields a bound on the
local dimension in terms of the order of the automorphism group of the
associated simple group.

\begin{corollary}\label{cor:local-dimension-bound}
Let $Y$ and $T$ be as in
Proposition~\ref{prop:local-component-action}, and set
$d=\dim_{\F_p}Y$ and $S=T/Z(T)$.
Then $d$ is even, $d\geq4$, and
\[
 p^{d/2}+1
 \leq |J_Y/\langle-I_Y\rangle|
 \leq |\Aut(S)|.
\]
In particular,
$d\leq 2\log_p\bigl(|\Aut(S)|-1\bigr)$.
\end{corollary}

\begin{proof}
The space $Y$ is symplectic, so $d$ is even, and
Lemma~\ref{lem:local-dimension} gives $d\geq4$.
By Proposition~\ref{prop:local-component-action},
the group $J_Y$ satisfies {\rm(SF)} on $Y$,
contains $-I_Y$, and satisfies
$J_Y/\langle-I_Y\rangle\leq\Aut(S)$.
The asserted inequalities now follow from
Lemma~\ref{lem:projective-count}.
\end{proof}

\section{Recognition of the local simple group}
\label{sec:recognition}

In this section we determine the possible simple quotients
of the components by using the action of the original
rational group on the full family of blocks.

\subsection{Scalar fusion on collections of blocks}

The following lemma translates {\rm(SF)} into a simultaneous fusion
condition on arbitrary collections of orthogonal summands, after
factoring out the scalar signs on these summands.

\begin{lemma}\label{lem:fusion-on-supports}
Let $G\leq\Sp(V,B)$ satisfy {\rm(SF)} over $\F_p$, where $p$ is odd.
Suppose that $G$ permutes the nonzero nondegenerate summands of
\[
 V=\mathop{\perp}_{\omega\in\Omega}V_\omega.
\]
Write $X_\omega=(V_\omega\setminus\{0\})/\{\pm1\}$, with
$[v]=\{v,-v\}$ and $a[v]=[av]$.
Let $D\leq G$ consist of the elements acting as $\pm I$ on each summand.
Then the following statements hold.
\begin{enumerate}
\item $D$ is a normal elementary abelian $2$-subgroup of $G$, and $Q=G/D$ acts faithfully
on $\coprod_{\omega\in\Omega}X_\omega$, with $qX_\omega=X_{q\omega}$.
\item For every nonempty $\Sigma\subseteq\Omega$, every tuple
$(x_\omega)\in\prod_{\omega\in\Sigma}X_\omega$, and every
$a\in\F_p^\times$, there is a $p'$-element $q\in Q$ satisfies
\[
 q\Sigma=\Sigma,\qquad qx_\omega=a x_{q\omega}
 \quad(\omega\in\Sigma).
\]
\end{enumerate}
\end{lemma}

\begin{proof}
The group $G$ acts on the disjoint union
\[
 \mathcal X=\coprod_{\omega\in\Omega}X_\omega
\]
by
$g[v]=[gv]$. This is well defined because $g(-v)=-gv$.

We show that its kernel is precisely $D$.
Every element of $D$ clearly fixes every element of $\mathcal X$.
Conversely, suppose that $g$ fixes every element of $\mathcal X$.
Then $g$ stabilizes each $V_\omega$, and $gv\in\{v,-v\}$ for every $0\ne v\in V_\omega$.
Choose a basis $e_1,\ldots,e_d$ of $V_\omega$ and write
$ge_i=\varepsilon_i e_i$, where $\varepsilon_i\in\{1,-1\}$.
For $i\ne j$, the element $[e_i+e_j]\in X_\omega$ is also fixed.
Hence
\[
 \varepsilon_i e_i+\varepsilon_j e_j
 \in\{e_i+e_j,-e_i-e_j\},
\]
and comparison of coefficients yields
$\varepsilon_i=\varepsilon_j$.
Thus $g$ acts as a single sign on $V_\omega$.
Since this holds for every $\omega$, we have $g\in D$.
Therefore the kernel is $D$, and hence $D\lhd G$ and $Q=G/D$
acts faithfully on $\mathcal X$.

Restriction to the summands embeds $D$ in
$\prod_{\omega\in\Omega}\{I_{V_\omega},-I_{V_\omega}\}$.
Hence $D$ is an elementary abelian 2-subgroup,
proving (1).

To prove~{\rm(2)}, choose, for each $\omega\in\Sigma$, a nonzero vector
$v_\omega\in V_\omega$ such that $x_\omega=[v_\omega]$, and set
\[
 v=\sum_{\omega\in\Sigma}v_\omega.
\]
The directness of the decomposition gives $v\ne0$.  By {\rm(SF)}, there
is a $p'$-element $g\in G$ such that $gv=av$.  The nonzero components
of $gv$ lie in the subspaces indexed by $g\Sigma$, while those of
$av$ lie in the subspaces indexed by $\Sigma$.  Hence
$g\Sigma=\Sigma$.

For each $\omega\in\Sigma$, comparison of the components in
$V_{g\omega}$ yields $gv_\omega=a v_{g\omega}$.
Passing to the corresponding elements of $X_{g\omega}$ and writing
$\bar g=gD\in Q$, we obtain
$\bar g x_\omega=a\,x_{\bar g\omega}$.
Since the order of $\bar g$ divides the order of $g$, the element
$\bar g$ is a $p'$-element.  This proves~{\rm(2)}.
\end{proof}

\begin{remark}
The sets $X_\omega$ consist of nonzero vectors modulo
sign, not of one-dimensional subspaces.
This distinction is essential.  Multiplication by $2$
is nontrivial on $X_\omega$ when $p\geq7$, whereas
every nonzero scalar acts trivially on the set of
one-dimensional subspaces.
\end{remark}

By bounding the number of tuples that a single element sends to their
$2$-multiples, the following lemma gives a basic counting inequality.

\begin{lemma}\label{lem:count-fused-labelings}
Retain the notation of Lemma~\ref{lem:fusion-on-supports},
and suppose that $p\geq7$.
Let $\varnothing\ne\Sigma\subseteq\Omega$, and assume
that every block indexed by $\Sigma$ has dimension $d$.
Write $b=|\Sigma|$.

For a fixed $q\in Q$ stabilizing $\Sigma$, the number of tuples
\[
 (x_\omega)_{\omega\in\Sigma}
 \in\prod_{\omega\in\Sigma}X_\omega
\]
satisfying $qx_\omega=2x_{q\omega}$ for every $\omega\in\Sigma$
is at most $p^{db/2}-1$.
Consequently, if $Q_\Sigma$ denotes the setwise
stabilizer of $\Sigma$ in $Q$, then
\[
 \left(\frac{p^d-1}{2}\right)^b
 \leq |Q_\Sigma|\bigl(p^{db/2}-1\bigr).
\]
\end{lemma}

\begin{proof}
Choose a representative $g\in G$ of $q$.
Since $q$ stabilizes $\Sigma$, the element $g$ preserves
the nondegenerate symplectic subspace
\[
 V_\Sigma=\mathop{\perp}_{\omega\in\Sigma}V_\omega.
\]

For each choice of signs $\varepsilon=(\varepsilon_\omega)_{\omega\in\Sigma}$, where $\varepsilon_\omega\in\{1,-1\}$,
let $s_\varepsilon$ act as
$\varepsilon_\omega I_{V_\omega}$ on $V_\omega$.
The map $s_\varepsilon$ is an isometry of $V_\Sigma$,
so $h_\varepsilon=s_\varepsilon g|_{V_\Sigma}$
is also an isometry.
These maps need not belong to $G$; only their
preservation of the form is needed.

The $2$-eigenspace of $h_\varepsilon$ is totally
isotropic. Indeed, for vectors $u,v$ in this eigenspace,
\[
 B(u,v)=B(h_\varepsilon u,h_\varepsilon v)=4B(u,v),
\]
and hence $B(u,v)=0$.
Its dimension is therefore at most $db/2$, so it
contains at most $p^{db/2}-1$ nonzero vectors.

Let $\mathcal A_q$ be the set of tuples $(x_\omega)_{\omega\in\Sigma}
 \in\prod_{\omega\in\Sigma}X_\omega$ satisfying
$qx_\omega=2x_{q\omega}$ for every $\omega\in\Sigma$.
Each tuple in $\mathcal A_q$ has exactly $2^b$ choices of vector
representatives
\[
 (v_\omega)_{\omega\in\Sigma},
 \qquad 0\ne v_\omega\in V_\omega,
 \qquad x_\omega=[v_\omega].
\]
For each such choice, the defining relation for $\mathcal A_q$ implies
\[
 gv_\omega=\delta_{q\omega}\,2v_{q\omega}
 \qquad(\omega\in\Sigma)
\]
for uniquely determined signs
$\delta_{q\omega}\in\{1,-1\}$.
Consequently, there is exactly one choice of $\varepsilon$ for which
\[
 h_\varepsilon
 \left(\sum_{\omega\in\Sigma}v_\omega\right)
 =
 2\sum_{\omega\in\Sigma}v_\omega.
\]

Conversely, a vector with a nonzero component in every summand
$V_\omega$, $\omega\in\Sigma$, satisfying this equation determines a
tuple in $\mathcal A_q$.  Counting pairs consisting of a choice of
signs and such a vector, we obtain
\[
 2^b|\mathcal A_q|
 \leq 2^b\bigl(p^{db/2}-1\bigr),
\]
which proves the first assertion after cancelling $2^b$.

There are
\[
 \prod_{\omega\in\Sigma}|X_\omega|
 =
 \left(\frac{p^d-1}{2}\right)^b
\]
tuples in $\prod_{\omega\in\Sigma}X_\omega$.
By Lemma~\ref{lem:fusion-on-supports}, each of them satisfies the
required relation for some element of $Q_\Sigma$.
Summing the preceding bound over $Q_\Sigma$, we get
\[
 \left(\frac{p^d-1}{2}\right)^b
 \leq |Q_\Sigma|\bigl(p^{db/2}-1\bigr),
\]
as required.
\end{proof}

The following lemma allows us to pass the fusion property through a
normal subgroup whose block stabilizers act trivially on the
corresponding sets $X_\omega$, while preserving the induced local actions.

\begin{lemma}
\label{lem:block-quotient-descent}
Let a finite group $Q$ act on a finite set $\Omega$ and on
$\coprod_{\omega\in\Omega}X_\omega$, with $qX_\omega=X_{q\omega}$.
Suppose that multiplication by $2$ on these sets commutes with  the action of  $Q$,
and that every tuple $(x_\omega)\in\prod_{\omega\in\Sigma}X_\omega$,
for every nonempty $\Sigma\subseteq\Omega$, admits $q\in Q$ satisfying
\[
 q\Sigma=\Sigma,\qquad qx_\omega=2x_{q\omega}
 \quad(\omega\in\Sigma).
\]
Let $N\lhd Q$, and suppose that $N_\omega:=N\cap Q_\omega$ acts trivially on $X_\omega$ for every $\omega\in\Omega$.
Then there is a family of sets
$\{\widehat X_\delta:\delta\in\Omega/N\}$
on which $Q/N$ acts, with the following properties.
\begin{enumerate}
\item
For each $\delta=N\omega$, after choosing the representative $\omega$,
the set $\widehat X_\delta$ is naturally identified with $X_\omega$.

\item
Under this identification, the permutation group induced on
$\widehat X_\delta$ by its stabilizer in $Q/N$ is exactly the group induced by
$Q_\omega$ on $X_\omega$.

\item
The family $\{\widehat X_\delta:\delta\in\Omega/N\}$ satisfies the same fusion
property.
\end{enumerate}
If, for a fixed prime $p$, the fusion elements in $Q$ can always
be chosen to be $p'$-elements, the same is true in $Q/N$.
\end{lemma}

\begin{proof}
For each $N$-orbit $\delta\subseteq\Omega$, let $\widehat X_\delta$
be the set of all $N$-equivariant tuples
\[
 s=(s_\eta)_{\eta\in\delta},
 \qquad s_\eta\in X_\eta,
\]
where $N$-equivariance means that
$s_{n\eta}=n s_\eta$ for any $n\in N$ and $\eta\in\delta$.

Fix $\omega\in\delta$.  Every $x\in X_\omega$ determines a unique
element of $\widehat X_\delta$.  Namely, define
\[
 s_{n\omega}=nx\qquad(n\in N).
\]
This is well defined.  If $n_1\omega=n_2\omega$, then
$n_2^{-1}n_1\in N_\omega$, which acts trivially on $X_\omega$.
Hence $n_1x=n_2x$. Since every element of $\delta$ has the form $n\omega$ for some
$n\in N$, the value of $s$ at $\omega$ determines all its other
values.  Hence the extension is unique.

Thus evaluation at $\omega$ induces a bijection
\[
 \operatorname{ev}_\omega:\widehat X_\delta\longrightarrow X_\omega,
 \qquad s\longmapsto s_\omega.
\]
Define multiplication by $2$ on $\widehat X_\delta$ coordinatewise.
It preserves $N$-equivariance and, under the above bijection,
corresponds to multiplication by $2$ on $X_\omega$.  This proves~{\rm(1)}.

We next define the action of $Q$.  If $s\in\widehat X_\delta$ and
$q\in Q$, put
\[
 (qs)_\eta=q\,s_{q^{-1}\eta}
 \qquad(\eta\in q\delta).
\]
Since $N\lhd Q$, the tuple $qs$ is again $N$-equivariant.  Indeed, if
$n\in N$, then
\[
\begin{aligned}
 (qs)_{n\eta}
 &=q\,s_{q^{-1}n\eta}\\
 &=q\,s_{(q^{-1}nq)(q^{-1}\eta)}\\
 &=q(q^{-1}nq)s_{q^{-1}\eta}\\
 &=n(qs)_\eta.
\end{aligned}
\]
Moreover, $N$ acts trivially on every $\widehat X_\delta$, since for
$n\in N$,
\[
 (ns)_\eta=n\,s_{n^{-1}\eta}=s_\eta.
\]
The action of $Q$ therefore induces an action of $Q/N$ on the family
$\{\widehat X_\delta:\delta\in\Omega/N\}$.

We now determine the action associated with a fixed
$\delta=N\omega$.  Let
\[
 Q_\delta=\{q\in Q:q\delta=\delta\}.
\]
Then $Q_\delta=NQ_\omega$.
Indeed, if $q\delta=\delta$, then $q\omega=n\omega$ for some $n\in N$,
so $n^{-1}q\in Q_\omega$.  The reverse inclusion is immediate.

Since $N$ acts trivially on $\widehat X_\delta$, the group induced by
$Q_\delta$ on $\widehat X_\delta$ is the same as that induced by
$Q_\omega$.  For $h\in Q_\omega$,
\[
 \operatorname{ev}_\omega(hs)
 =(hs)_\omega
 =h s_\omega
 =h\,\operatorname{ev}_\omega(s).
\]
Thus evaluation at $\omega$ identifies this action with the original
action of $Q_\omega$ on $X_\omega$.  Since the stabilizer of $\delta$
in $Q/N$ is $Q_\delta/N$, this proves~{\rm(2)}.

Finally, let $\mathcal D$ be a nonempty collection of $N$-orbits, and
choose $s_\delta\in\widehat X_\delta$ for each $\delta\in\mathcal D$.
Set
\[
 \Sigma=\bigcup_{\delta\in\mathcal D}\delta.
\]
For $\omega\in\Sigma$, let $\delta$ be the unique member of
$\mathcal D$ containing $\omega$, and define
\[
 x_\omega=(s_\delta)_\omega\in X_\omega.
\]
By the original fusion property, there exists $q\in Q$ such that
\[
 q\Sigma=\Sigma,\qquad
 qx_\omega=2x_{q\omega}
 \quad(\omega\in\Sigma).
\]
Since $N\lhd Q$, the element $q$ sends $N$-orbits to $N$-orbits and
therefore permutes $\mathcal D$.  For $\delta\in\mathcal D$ and
$\eta\in q\delta$, we have
\[
 (q s_\delta)_\eta
 =q x_{q^{-1}\eta}
 =2x_\eta
 =(2s_{q\delta})_\eta.
\]
Hence
\[
 q s_\delta=2s_{q\delta}
 \qquad(\delta\in\mathcal D).
\]
Passing to the image $qN\in Q/N$, we conclude that the family
$\{\widehat X_\delta:\delta\in\Omega/N\}$ satisfies the same fusion
property.  This proves~{\rm(3)}. For the final assertion, choose $q$ to be a $p'$-element.
The order of $qN$ divides $o(q)$, so $qN$ is again a $p'$-element.
\end{proof}
\begin{remark}
The sets $\widehat X_\delta$ constructed above are sets of
$N$-equivariant tuples; they need not arise from a direct-sum
decomposition of the original vector space.  After choosing
$\omega\in\delta$, however, evaluation at $\omega$ identifies
$\widehat X_\delta$ with $X_\omega$, and under this identification
the induced permutation groups are the same.

If $Q$ is rational, then so is $Q/N$.  If the action of $Q/N$ on the
family $\{\widehat X_\delta:\delta\in\Omega/N\}$ has a nontrivial
kernel, we may factor out this kernel.  This does not change the
induced permutation groups on the individual sets $\widehat X_\delta$
or the fusion property.
\end{remark}

\subsection{Restrictions on transitive normal subgroups}

In this subsection, we apply the preceding counting bounds to
restrict the possible nonabelian simple direct factors of a
transitive normal subgroup and determine its action in the
alternating case.

We shall also use the following lower bound for the degree of a
faithful transitive action of a direct power of a nonabelian simple
group.

\begin{lemma}
\label{lem:direct-power-bound}
Let $S$ be a nonabelian finite simple group, and let $r\geq1$.
Write
\[
 \mu(S)=\min\{[S:M]:M<S\}.
\]
If $S^r$ acts faithfully and transitively on a set of size $b$, then $b\geq \mu(S)^r$.
\end{lemma}

\begin{proof}
Let $\Omega$ be the set on which $S^r$ acts, and fix
$\omega\in\Omega$. Write $H=(S^r)_\omega$.
Since the action is transitive and faithful, we have
\[
 b=|\Omega|=[S^r:H], \qquad
 \operatorname{Core}_{S^r}(H)
 =\bigcap_{g\in S^r}H^g=1.
\]
Hence $H$ is a core-free subgroup of $S^r$. By
\cite[Lemma~2.1(i)]{CannonHoltUnger2019}, the smallest index of a
core-free subgroup of $S^r$ is $\mu(S)^r$. Thus
$b=[S^r:H]\geq\mu(S)^r$, as required.
\end{proof}

The next lemma combines the preceding counting bounds with a bound
on the order of the ambient group obtained from a transitive normal
subgroup.

\begin{lemma}
\label{lem:normal-subgroup-count}
Let $p\geq7$ and $|\Omega|=b$. For each $\omega\in\Omega$, let
$V_\omega$ be a symplectic $\F_p$-space of dimension $d$, and write
$X_\omega=(V_\omega\setminus\{0\})/\{\pm1\}$.
Let a finite group $Q$ act faithfully on
$\coprod_{\omega\in\Omega}X_\omega$, permuting the sets $X_\omega$
by maps induced by linear isometries.
Suppose that the following conditions hold:

\begin{enumerate}
\item Every tuple $(x_\omega)\in\prod_{\omega\in\Omega}X_\omega$
admits $q\in Q$ with $qx_\omega=2x_{q\omega}$ for all $\omega$.
\item A normal subgroup $N\cong S^r$ acts faithfully and transitively
on $\Omega$, where $S$ is nonabelian simple.
\item $C_{A_\omega}(B_\omega)=1$, where $A_\omega$ and $B_\omega$
are induced on $X_\omega$ by $Q_\omega$ and $N_\omega$, respectively.
\end{enumerate}
Then
\[
 \left(\frac{p^d-1}{2}\right)^b
 \leq b\,|\Aut(S)|^r r!\bigl(p^{db/2}-1\bigr).
\]
In particular,
\[
 \left(\frac{p^d-1}{2p^{d/2}}\right)^b
 <b\,|\Aut(S)|^r r!.
\]
\end{lemma}

\begin{proof}
We first bound $|Q|$.
Write
$C=C_Q(N)$, so that $Q/C\hookrightarrow\Aut(N)\cong \Aut(S)\wr S_r$.
We claim that the action of $C$ on $\Omega$
has trivial point stabilizers.

Let $c\in C$ fix $\omega$.
For every $n\in N_\omega$, the elements $c$ and $n$
commute. Hence the permutation induced by $c$
on $X_\omega$ centralizes $B_\omega$.
By hypothesis,
$C_{A_\omega}(B_\omega)=1$,
so $c$ acts trivially on $X_\omega$.

Now let $\eta\in\Omega$.  Since $N$ acts transitively on $\Omega$,
there exists $n\in N$ such that $n\omega=\eta$.
Since $c$ commutes with $n$ and fixes $\omega$,
\[
 c\eta=cn\omega=nc\omega=n\omega=\eta.
\]
Moreover, every element of $X_\eta$ has the form $nx$
for some $x\in X_\omega$. Therefore
\[
 c(nx)=n(cx)=nx.
\]
Hence $c$ fixes every element of $\mathcal X$.
Since the action of $Q$ on $\mathcal X$ is faithful,
it follows that $c=1$.
Consequently, $C_\omega=1$ and $|C|=|C\omega|\leq b$.
Therefore
\[
 |Q|\leq b\,|\Aut(S)|^r r!.
\]

The number of tuples in $\prod_{\omega\in\Omega}X_\omega$ is
\[
 \prod_{\omega\in\Omega}|X_\omega|
 =
 \left(\frac{p^d-1}{2}\right)^b.
\]
Fix $q\in Q$. We show that at most $p^{db/2}-1$ of these tuples
satisfy
\[
 qx_\omega=2x_{q\omega}
 \qquad(\omega\in\Omega).
\]

For each $\omega\in\Omega$, choose a linear isometry
$T_\omega:V_\omega\longrightarrow V_{q\omega}$ inducing the
action of $q$ on $X_\omega$.
Given a tuple $(x_\omega)_{\omega\in\Omega}$ satisfying the
displayed condition, choose a nonzero vector
$v_\omega\in V_\omega$ representing $x_\omega$ for each
$\omega\in\Omega$.
Then
\[
 T_\omega v_\omega
 =2\varepsilon_\omega v_{q\omega}
 \qquad(\omega\in\Omega)
\]
for uniquely determined signs $\varepsilon_\omega\in\{1,-1\}$.

For each of the $2^b$ choices of signs
$\varepsilon_\omega\in\{1,-1\}$, define a linear isometry $T$
of the symplectic space
\[
 (V,B)=\mathop{\perp}_{\omega\in\Omega}V_\omega
\]
by $T|_{V_\omega}=\varepsilon_\omega T_\omega$.
For the signs determined by the chosen representatives, we have
\[
 T\left(\sum_{\omega\in\Omega}v_\omega\right)
 =2\sum_{\omega\in\Omega}v_\omega.
\]
For every choice of signs, the $2$-eigenspace of $T$ is
totally isotropic. Indeed, if $Tu=2u$ and $Tv=2v$, then
\[
 B(u,v)=B(Tu,Tv)=4B(u,v),
\]
so $B(u,v)=0$ because $p\geq7$.
Its dimension is therefore at most $db/2$, and it contains
at most $p^{db/2}-1$ nonzero vectors.
Thus there are at most $2^b(p^{db/2}-1)$ possible pairs
consisting of a choice of signs and a nonzero vector satisfying
the corresponding equation.
Each tuple $(x_\omega)_{\omega\in\Omega}$ under consideration
gives exactly $2^b$ distinct such pairs, one for each choice
of its representatives $v_\omega$.
Different tuples give disjoint sets of pairs.
Hence at most $p^{db/2}-1$ tuples satisfy the displayed condition.

By hypothesis, every tuple in $\prod_{\omega\in\Omega}X_\omega$
satisfies $qx_\omega=2x_{q\omega}$ for all $\omega\in\Omega$
for at least one element $q\in Q$.
Therefore
\[
 \left(\frac{p^d-1}{2}\right)^b
 \leq |Q|\bigl(p^{db/2}-1\bigr)
 \leq b\,|\Aut(S)|^r r!\,
          \bigl(p^{db/2}-1\bigr).
\]
Dividing by $p^{db/2}$, we obtain
\[
 \left(\frac{p^d-1}{2p^{d/2}}\right)^b
 \leq b\,|\Aut(S)|^r r!
       \left(1-p^{-db/2}\right),
\]
which implies the final strict inequality.
\end{proof}

The following corollary gives numerical bounds independent of
$p$ and $d$, which will restrict the possible simple factors
of a transitive normal subgroup.

\begin{corollary}\label{cor:normal-subgroup-numerical-bound}
Under the hypotheses of
Lemma~\ref{lem:normal-subgroup-count}, suppose that
$d\geq4$. Then
\[
 24^b<b\,|\Aut(S)|^r r!,
 \qquad
 b\geq \mu(S)^r.
\]
\end{corollary}

\begin{proof}
Since $p\geq7$ and $d\geq4$,
\[
 \frac{p^d-1}{2p^{d/2}}
 =
 \frac{p^{d/2}-p^{-d/2}}2
 \geq\frac{49-49^{-1}}2
 =\frac{1200}{49}>24.
\]
The first inequality follows from
Lemma~\ref{lem:normal-subgroup-count}.
The second follows from
Lemma~\ref{lem:direct-power-bound}, since
$N\cong S^r$ acts faithfully and transitively on
the $b$ blocks.
\end{proof}

For alternating groups, we use the following standard bound on subgroup indices.

\begin{lemma}\label{lem:alternating-index}
Let $n\geq9$, and let $M<A_n$.  Then either $M$ is the stabilizer of
a point in the natural action of $A_n$, or
\[
[A_n:M]\geq\binom{n}{2}.
\]
\end{lemma}

\begin{proof}
Apply \cite[Theorem~5.2A]{DixonMortimer} with $r=2$.
Its exceptional pairs $(n,[A_n:M])$ are
$(5,6)$, $(6,6)$, $(6,15)$, $(7,15)$, $(8,15)$, and $(9,120)$;
none satisfies $n\geq9$ and $[A_n:M]<\binom n2$.
The alternative preserving two equal blocks, for $n=2m\geq10$,
has index at least $\tfrac12\binom{2m}{m}>\binom{2m}{2}$,
so it is excluded as well.
Thus a proper subgroup of index less than $\binom n2$ contains
the pointwise stabilizer of one point, a subgroup isomorphic to
$A_{n-1}$. Since this subgroup is maximal in $A_n$, $M$ is that
point stabilizer.
\end{proof}

We first apply the numerical bounds to alternating groups,
showing that $N$ must be a single alternating group acting naturally
on $\Omega$.

\begin{lemma}
\label{lem:alternating-block-action}
Under the hypotheses of
Lemma~\ref{lem:normal-subgroup-count}, suppose that
$d\geq4$ and $N\cong A_n^r$, where $n\geq5$.
Then $r=1$, $n\geq21$, and the action of $N\cong A_n$
on $\Omega$ is equivalent to its natural action.
In particular,
\[
 |\Omega|=n,\qquad N_\omega\cong A_{n-1}.
\]
\end{lemma}

\begin{proof}
Write $b=|\Omega|$.
By Corollary~\ref{cor:normal-subgroup-numerical-bound},
\begin{equation}\label{eq:alt-count-initial}
24^b<b\,|\Aut(A_n)|^r r!.
\end{equation}
We shall compare this with the lower bounds on $b$.

First, $\mu(A_n)=n$. Indeed, a natural point stabilizer has
index $n$. Conversely, since $A_n$ is simple, its action on
the cosets of any proper subgroup is faithful. An index
$m<n$ would therefore give an embedding $A_n\hookrightarrow S_m$,
contrary to \[
 |A_n|=\frac{n!}{2}>(n-1)!\geq m!.
\]
Therefore Lemma~\ref{lem:direct-power-bound} gives
\begin{equation}\label{eq:alt-permutation-degree}
b\geq n^r.
\end{equation}

For $n\geq5$, we have $|\Aut(A_n)|\leq2n!$.
Indeed, $\Aut(A_n)\cong S_n$ when $n\ne6$, whereas
$|\Aut(A_6)|=1440=2\cdot6!$.
Consequently, \eqref{eq:alt-count-initial} implies
\begin{equation}\label{eq:alt-count-factorial}
24^b<b(2n!)^r r!.
\end{equation}

We first exclude $r\geq2$.
Induction first on $n$ when $r=2$, and then on $r$ for each
fixed $n\geq5$, shows that
\begin{equation}\label{eq:alt-power-inequality}
24^{n^r}>n^r(2n!)^r r!
\qquad(n\geq5,\ r\geq2).
\end{equation}

The function $24^a/a$ is strictly increasing on
the positive integers, since
\[
 \frac{24^{a+1}/(a+1)}{24^a/a}
 =\frac{24a}{a+1}>1.
\]
Thus \eqref{eq:alt-permutation-degree} and \eqref{eq:alt-power-inequality} imply
\[
 \frac{24^b}{b}
 \geq\frac{24^{n^r}}{n^r}
 >(2n!)^r r!,
\]
contradicting \eqref{eq:alt-count-factorial}.
Hence $r=1$.

Now \eqref{eq:alt-permutation-degree} gives $b\geq n$, and \eqref{eq:alt-count-factorial} becomes
\begin{equation}\label{eq:alt-count-one-factor}
24^b<2b\cdot n!.
\end{equation}
If $5\leq n\leq20$, then
\begin{equation}\label{eq:alt-small-degree}
24^n>2n\cdot n!.
\end{equation}
The inequality holds at $n=5$, and
the ratio of its left-hand side to its right-hand side
is multiplied
by
\[
 \frac{24n}{(n+1)^2}>1
 \qquad(5\leq n\leq19).
\]
Thus \eqref{eq:alt-small-degree} holds throughout the stated range.
Since $b\geq n$ and $24^a/a$ is increasing,
\eqref{eq:alt-small-degree} contradicts \eqref{eq:alt-count-one-factor}.
Therefore $n\geq21$.

Identify $N$ with $A_n$, and let $M=N_\omega$ be the stabilizer
of $\omega$ in the action on $\Omega$.
If $M$ is a point stabilizer in the natural action of $A_n$,
then the two transitive actions are equivalent, and the
conclusion follows.

Suppose otherwise. Lemma~\ref{lem:alternating-index} shows that
\[
 b=[A_n:M]\geq
 b_0:=\binom{n}{2}.
\]
We also have
\begin{equation}\label{eq:alt-large-index-estimate}
24^{\binom{n}{2}}>2n^2n!
\qquad(n\geq21).
\end{equation}
In fact, the inequality already holds at $n=5$. The ratio of
the left-hand side to the right-hand side is multiplied by
\[
 \frac{24^n n^2}{(n+1)^3}
 \geq\frac{24^n}{8n}>1
 \qquad(n\geq5)
\]
at the next step, proving the claim by induction.  Therefore
\[
 \frac{24^b}{b}
 \geq\frac{24^{b_0}}{b_0}
 \geq\frac{24^{b_0}}{n^2}
 >2n!,
\]
again contradicting \eqref{eq:alt-count-one-factor}.
Thus $M$ is a natural point stabilizer,
and the proof is complete.
\end{proof}

The following corollary identifies the local socle as $A_{n-1}$
when the normal subgroup $A_n$ acts naturally. This explains the
alternating possibility in the recognition theorem.

\begin{corollary}\label{cor:alternating-local-socle}
Under the hypotheses of Lemma~\ref{lem:alternating-block-action},
suppose that $A_\omega$ is almost simple with socle $C$ and that
the image of $N_\omega$ contains $C$. Then $C\cong A_{n-1}$.
\end{corollary}

\begin{proof}
By Lemma~\ref{lem:alternating-block-action},
$N_\omega\cong A_{n-1}$, where $n\geq21$.
Its image $B_\omega$ in $A_\omega$ is nontrivial because it
contains $C$. Simplicity therefore gives $B_\omega\cong A_{n-1}$.
Since $C\lhd A_\omega$ and $C\leq B_\omega$, the subgroup $C$
is a nontrivial normal subgroup of the simple group $B_\omega$.
Hence $C=B_\omega$.
\end{proof}

\begin{lemma}\label{lem:bounded-simple-exclusion}
Let $N\cong S^r$ be a transitive normal subgroup of $Q$
as in Lemma~\ref{lem:normal-subgroup-count}, with $d\geq4$.
If $\mu(S)\geq8$, then
\[
 |\Aut(S)|>\frac{24^8}{8}.
\]
\end{lemma}

\begin{proof}
Write $b=|\Omega|$ and put
\[
 M=\frac{24^8}{8}.
\]
Suppose, for a contradiction, that $|\Aut(S)|\leq M$.
Since $\mu(S)\geq8$, it follows from Corollary~\ref{cor:normal-subgroup-numerical-bound}
that
\[
 b\geq8^r,
 \qquad
 24^b<b\,|\Aut(S)|^r r!\leq bM^r r!.
\]
On the other hand, induction on $r$ shows that
\[
 24^{8^r}\geq8^rM^r r!
 \qquad(r\geq1).
\]
Since the function $24^a/a$ is strictly increasing on the positive
integers, the inequality $b\geq8^r$ implies
\[
 \frac{24^b}{b}
 \geq\frac{24^{8^r}}{8^r}
 \geq M^r r!,
\]
contrary to $24^b<bM^r r!$.
Therefore $|\Aut(S)|>24^8/8$.
\end{proof}

We now apply the preceding lemma to exclude the groups in
$\mathcal E$ as simple direct factors of a transitive normal subgroup.

\begin{corollary}\label{cor:exceptional-normal-subgroup-excluded}
Let $N\cong S^r$ be a transitive normal subgroup of $Q$
as in Lemma~\ref{lem:normal-subgroup-count}, with $d\geq4$.
Then $S\notin\mathcal E$.
\end{corollary}

\begin{proof}
The orders and outer automorphism groups in the ATLAS
\cite{Atlas1985,AtlasOnline} give
\[
\begin{array}{c|r|r|r}
 S&|S|&|\Out(S)|&|\Aut(S)|\\ \hline
 \PSp_4(3)&25\,920&2&51\,840\\
 \Sp_6(2)&1\,451\,520&1&1\,451\,520\\
{\rm O}_8^+(2)&174\,182\,400&6&1\,045\,094\,400\\
 \PSL_3(4)&20\,160&12&241\,920\\
 \PSU_4(3)&3\,265\,920&8&26\,127\,360
\end{array}
\]
Suppose that $S\in\mathcal E$.
Every group in the table has order greater than $7!$,
so it cannot embed in $S_7$.
Since the action of a simple group on the cosets of any
proper subgroup is faithful, this implies $\mu(S)\geq8$.
Lemma~\ref{lem:bounded-simple-exclusion} therefore gives
\[
 |\Aut(S)|>\frac{24^8}{8}=13\,759\,414\,272,
\]
contrary to the table.
Thus $S\notin\mathcal E$.
\end{proof}

\subsection{Recognition and global realization}

Our aim is to determine the possible simple groups $T/Z(T)$ arising
from the components of the terminal group $L$.  We already know that
$T/Z(T)$ is the socle of an almost simple group acting on a local
space $Y$, after scalar signs have been factored out.  However, this
local group is obtained from a stabilizer, and a subgroup of a rational
group need not be rational.  We therefore cannot apply the
Feit--Seitz theorem directly to the local group.

The argument has two steps. First, we prove an abstract
recognition result: if an almost simple local action arises
inside a rational group satisfying the fusion condition,
then its socle is either an alternating group or isomorphic
to a composition factor of the rational group.
Second, we place the local
space $Y$ back into a decomposition of the original module.  This
provides the rational group and the fusion condition needed for the
first step.

\begin{proposition}
\label{prop:local-simple-recognition}
Let $Q$ be a finite rational group acting transitively on a finite
nonempty set $\Omega$. For each $\omega\in\Omega$, let $V_\omega$
be a symplectic $\F_p$-space of dimension $d$, where $p\geq7$ and
$d\geq4$, and write $X_\omega=(V_\omega\setminus\{0\})/\{\pm1\}$.
Assume that the following conditions hold:
\begin{enumerate}
\item $Q$ acts faithfully on $\coprod_{\omega\in\Omega}X_\omega$,
with each map $X_\omega\to X_{q\omega}$ induced by a linear isometry.
\item For every nonempty $\Sigma\subseteq\Omega$ and every tuple
$(x_\omega)\in\prod_{\omega\in\Sigma}X_\omega$, some $p'$-element
$q\in Q$ satisfies
\[
 q\Sigma=\Sigma,\qquad qx_\omega=2x_{q\omega}
 \quad(\omega\in\Sigma),
\]
where $2[v]=[2v]$.
\item The group $A_\omega$ induced by $Q_\omega$ on $X_\omega$
is almost simple with socle isomorphic to $C$.
\end{enumerate}
Then $C$ is alternating or isomorphic to a composition factor of $Q$.
In particular, $C$ is alternating or belongs to $\mathcal E$
in~\eqref{eq:exceptional-list}.
\end{proposition}
\begin{proof}
We use induction on $|Q|$.  The trivial group cannot satisfy the
hypothesis on the local socle, so there is no initial case to check.
Choose a minimal nontrivial normal subgroup $N$ of $Q$, fix
$\omega\in\Omega$, and abbreviate $A_\omega$ to $A$.
We identify the socle of $A$ with $C$, so that
\[
 C\leq A\leq\Aut(C),\qquad C_A(C)=1.
\]

We shall repeatedly use the following observation: a normal subgroup
$B$ of $A$ is either trivial or contains $C$.  Indeed,
$B\cap C\lhd C$.  If $B\cap C=1$, then
$[B,C]\leq B\cap C=1$, and hence $B\leq C_A(C)=1$.
Otherwise, since $C$ is simple, we have $C\leq B$.
In particular, this observation applies to the image of
$N_\omega=N\cap Q_\omega$ in $A$.

\medskip
\noindent\textbf{Case 1: $N_\omega$ acts trivially on $X_\omega$.}

Since $Q$ is transitive on $\Omega$ and $N\lhd Q$, the stabilizer
$N_\eta$ acts trivially on $X_\eta$ for every $\eta\in\Omega$.
Lemma~\ref{lem:block-quotient-descent} therefore allows us to replace
the index set $\Omega$ by the set $\Omega/N$ of $N$-orbits and to
replace $Q$ by $Q/N$.

By Lemma~\ref{lem:block-quotient-descent}, each set
$\widehat X_\delta$ can be identified with $X_\eta$
for a chosen representative $\eta\in\delta$, with the
same local permutation group and multiplication by $2$.
The fusion elements in the quotient are images of
$p'$-elements of $Q$, so they are still $p'$-elements.

The maps between the sets $\widehat X_\delta$ also
come from isometries.
Indeed, for chosen representatives $\eta\in\delta$
and $\eta'\in q\delta$, choose $n\in N$ such that
$nq\eta=\eta'$. Under the above identifications, the
map induced by $qN$ is induced by $nq$, and hence
comes from an isometry.

The group $Q/N$ acts transitively on $\Omega/N$.
Let $Q'$ be its image in the permutation group on
$\coprod_{\delta\in\Omega/N}\widehat X_\delta$.
Then $Q'$ is rational and $|Q'|<|Q|$.
The local permutation groups and the fusion condition
are unchanged, so the induction hypothesis applies to $Q'$.
Hence $C$ is alternating or isomorphic to a composition
factor of $Q'$. Since $Q'$ is a quotient of $Q$,
every composition factor of $Q'$ is also a composition
factor of $Q$. This proves the assertion in this case.

\medskip
For the rest of the proof, the image of $N_\omega$ in $A$ is
nontrivial and hence contains $C$.  In particular, $N$ is
nonabelian.  The structure of finite minimal normal subgroups gives
\[
 N\cong S^\ell
\]
for a nonabelian finite simple group $S$ and some $\ell\geq1$.
Since $N\lhd Q$, the group $S$ is a composition factor of $Q$.

Let $\Delta$ be the $N$-orbit containing $\omega$, and put
\[
 \Delta=N\omega,\qquad
 K=\ker\bigl(N\longrightarrow\operatorname{Sym}(\Delta)\bigr).
\]
Thus every element of $K$ fixes each point of $\Delta$
and preserves each set $X_\eta$, with $\eta\in\Delta$.
However, its action on $X_\eta$ need not be trivial.

Since $Q_\Delta$ normalizes $N$ and preserves $\Delta$,
it also normalizes $K$.
Moreover, $Q_\omega\leq Q_\Delta$ and $K\leq N_\omega$.
Hence the image of $K$ in $A$ is normal in $A$,
so it is either trivial or contains $C$.
We consider these two possibilities separately.

\medskip
\noindent\textbf{Case 2: the image of $K$ in $A$ contains $C$.}

Let $B$ be this image. Since $C\lhd A$ and $C\leq B$,
the simple group $C$ is a composition factor of $B$,
and hence of $K$.
On the other hand, $K\lhd N\cong S^\ell$, so $K$ is
the direct product of some of the simple direct factors
of $N$. Thus all composition factors of $K$ are
isomorphic to $S$, and consequently $C\cong S$.
Since $S$ is a composition factor of $Q$, the conclusion
follows in this case.

\medskip
\noindent\textbf{Case 3: $K$ acts trivially on $X_\omega$.}

Since $K\lhd N$ and $N$ is transitive on $\Delta$,
$K$ acts trivially on $X_\eta$ for every $\eta\in\Delta$.
Write
\[
 \mathcal X_\Delta=\coprod_{\eta\in\Delta}X_\eta,
\]
and let $H$ be the kernel of the action of $Q_\Delta$
on $\mathcal X_\Delta$. Then \[
 \Lambda=Q_\Delta/H
\]
acts faithfully on $\mathcal X_\Delta$.

We claim that $N\cap H=K$.
Since $K$ acts trivially on $\mathcal X_\Delta$,
we have $K\leq N\cap H$.
Conversely, every element of $N\cap H$ fixes all points
of $\mathcal X_\Delta$, and therefore fixes every
$\eta\in\Delta$. Hence it belongs to $K$.

Thus the image of $N$ in $\Lambda$ is
\[
 \overline N=NH/H\cong N/K\cong S^r
\]
for some $r\geq1$. Here $N/K$ is nontrivial because
the image of $N_\omega$ on $X_\omega$ contains $C$,
whereas $K$ acts trivially.
By the definition of $K$, the action of $\overline N$
on $\Delta$ is faithful. It is also transitive,
and $\overline N$ is normal in $\Lambda$.

We now check the remaining hypotheses of
Lemma~\ref{lem:normal-subgroup-count}.
Since $Q_\omega\leq Q_\Delta$, the group induced by
$\Lambda_\omega$ on $X_\omega$ is still $A$.
Likewise, $\overline N_\omega$ and $N_\omega$ induce
the same group $B$ on $X_\omega$. As $C\leq B$, we have
\[
 C_A(B)\leq C_A(C)=1.
\]

The fusion hypothesis, applied with $\Sigma=\Delta$,
provides an element of $Q_\Delta$ that fuses any given
tuple in $\prod_{\eta\in\Delta}X_\eta$ to its
$2$-multiple. Its image in $\Lambda$ has the same action.
Moreover, the maps between the sets $X_\eta$ are still
induced by isometries, and the corresponding spaces
have dimension $d\geq4$.
Thus Lemma~\ref{lem:normal-subgroup-count} and its
consequences apply to $\overline N\lhd\Lambda$.

Although $\Lambda$ need not be rational, $S$ is a
composition factor of the rational group $Q$.
By the Feit--Seitz theorem
\cite[Theorem~B]{FeitSeitz1989}, $S$ is therefore
alternating or belongs to $\mathcal E$.
Corollary~\ref{cor:exceptional-normal-subgroup-excluded}
rules out the latter possibility.
Hence $S\cong A_n$ for some $n\geq5$.

By Lemma~\ref{lem:alternating-block-action}, we have
$r=1$ and $n\geq21$, and $\overline N\cong A_n$
acts naturally on $\Delta$.
Hence $\overline N_\omega\cong A_{n-1}$.
Since its image in $A$ contains $C$, it follows from
Corollary~\ref{cor:alternating-local-socle} that
\[
 C\cong A_{n-1}.
\]
Thus $C$ is alternating, completing the induction.

Finally, whenever $C$ is a composition factor of $Q$, the
Feit--Seitz theorem again says that $C$ is alternating or belongs
to $\mathcal E$.  This proves the last assertion.
\end{proof}

The next lemma shows how an orthogonal decomposition of
one block extends to a decomposition of the whole space
by taking its translates. It also shows that the group
induced by the stabilizer of each smaller block remains
unchanged.

\begin{lemma}
\label{lem:refine-global-blocks}
Let $G\leq\Sp(V)$ transitively permute the nonzero nondegenerate
summands of $V=\mathop{\perp}_{\omega\in\Omega}V_\omega$.
Fix $U=V_{\omega_0}$, and let $L$ be the group induced by $G_U$ on $U$.
If $L$ transitively permutes the nonzero nondegenerate summands of
$U=Y_1\perp\cdots\perp Y_s$, then:

\begin{enumerate}
\item The distinct $G$-translates of $Y_1$ form an orthogonal
decomposition of $V$, permuted transitively by $G$.
\item $G_{Y_1}$ and $L_{Y_1}$ induce the same group on $Y_1$.
\end{enumerate}
\end{lemma}

\begin{proof}
For each $\omega\in\Omega$, choose $g_\omega\in G$ with
$g_\omega U=V_\omega$, and transport the decomposition of $U$:
\[
 V_\omega=g_\omega Y_1\perp\cdots\perp g_\omega Y_s.
\]
These decompositions together give an orthogonal decomposition of
$V$ into nonzero nondegenerate subspaces.

The collection inside $V_\omega$ does not depend on the choice of
$g_\omega$.  Indeed, if $g'_\omega U=V_\omega$, then
$g'_\omega=g_\omega h$ for some $h\in G_U$.
The restriction of $h$ to $U$ belongs to $L$, so it permutes the
$Y_j$.  Thus $g'_\omega$ and $g_\omega$ give the same collection.
It follows also that $G$ permutes all the transported subspaces.
Since $G$ is transitive on the original blocks and $L$
is transitive on the $Y_j$, $G$ is transitive on the
subspaces $g_\omega Y_j$. They are
therefore precisely the distinct translates of $Y_1$.

It remains to show that $G_{Y_1}$ and $L_{Y_1}$
induce the same group on $Y_1$. If $gY_1=Y_1$, then
both $U$ and $gU$ contain the nonzero subspace $Y_1$.
Distinct original blocks have zero intersection, so $gU=U$.
Hence $G_{Y_1}\leq G_U$.
Let $\pi:G_U\to L$ be restriction to $U$.  For $g\in G_U$,
\[
 gY_1=Y_1\quad\Longleftrightarrow\quad\pi(g)Y_1=Y_1.
\]
Thus $G_{Y_1}=\pi^{-1}(L_{Y_1})$. The surjectivity of $\pi$ now shows that these two stabilizers induce
exactly the same transformations on $Y_1$.
\end{proof}

Our construction involves three successive refinements:
the sign eigenspaces leading to $W$, the component
subspaces of $W$, and the common eigenspaces of the
component centralizer in a chosen component subspace $U$.
By applying Lemma~\ref{lem:refine-global-blocks} at each
stage, we obtain an orthogonal decomposition of $V_0$
whose summands include $Y$ and are permuted transitively
by $G_0$. The chosen spaces satisfy
\[
 Y\leq U\leq W\leq V_0,
\]
and their stabilizers in $G_0$ induce $J_Y$, $J$, $L$,
and $G_0$, respectively.

 Although $J_Y$ need not be rational,
we realize its action modulo sign as a local action
of a rational quotient of $G_0$, with the required
fusion condition.

\begin{proposition} \label{prop:global-local-realization}
Let $G_0\leq\Sp(V_0)$ be a finite rational group satisfying
{\rm(SF)} over $\F_p$, where $p\in\{7,11\}$.
Suppose that $(L,W)$ is a $2$-terminal pair obtained from
$(G_0,V_0)$ by successively passing to sign eigenspaces
and taking the group induced by the stabilizer at each step.
Assume that $(L,W)$ satisfies the hypotheses of
Proposition~\ref{prop:local-component-action}.

Let $Y$ and $J_Y$ be as in that proposition, and write
$S=T/Z(T)$ and $d=\dim_{\F_p}Y$.
Then:
\begin{enumerate}
\item $V_0$ has an orthogonal decomposition
\[
 V_0=\mathop{\perp}_{\omega\in\Omega}Y_\omega,
 \qquad Y_{\omega_0}=Y,
\]
permuted transitively by $G_0$, and the stabilizer of $Y$ induces
exactly $J_Y$ on $Y$.
\item The group $Q$ induced by $G_0$ on
$\coprod_{\omega\in\Omega}(Y_\omega\setminus\{0\})/\{\pm1\}$
is a rational quotient of $G_0$ satisfying all the hypotheses of
Proposition~\ref{prop:local-simple-recognition}, with local socle $S$.
\end{enumerate}
\end{proposition}

\begin{proof}
\noindent\textbf{Step 1: construct the decomposition of $V_0$.}

Initially, regard $V_0$ as a single block.
At each sign-eigenspace refinement, we decompose the
chosen block into nonzero nondegenerate subspaces
permuted transitively by its induced stabilizer group.
Lemma~\ref{lem:refine-global-blocks} extends this
decomposition to all blocks by applying elements of $G_0$.
After the successive refinements leading to $W$, we obtain
an orthogonal decomposition of $V_0$ whose summands are
permuted transitively by $G_0$, with $W$ as one summand.
Its stabilizer in $G_0$ induces exactly $L$ on $W$.

By Proposition~\ref{prop:component-decomposition},
$W$ decomposes orthogonally into component subspaces
permuted transitively by $L$.
We again apply Lemma~\ref{lem:refine-global-blocks}.
The stabilizer in $G_0$ of the chosen component subspace
$U$ induces exactly the group $J$ from
Lemma~\ref{lem:component-subspace-action}.

Finally, we decompose $U$ into the nonzero common
eigenspaces of $C_J(\overline T)$.
The proof of Proposition~\ref{prop:local-component-action}
shows that these eigenspaces are nondegenerate, mutually
orthogonal, and permuted transitively by $J$.
Applying Lemma~\ref{lem:refine-global-blocks} once more,
we obtain
\[
 V_0=\mathop{\perp}_{\omega\in\Omega}Y_\omega,
 \qquad Y_{\omega_0}=Y,
\]
where $G_0$ permutes the summands transitively and the
stabilizer of $Y$ in $G_0$ induces exactly $J_Y$ on $Y$.
All summands have the same dimension $d$, which is at least
$4$ by Lemma~\ref{lem:local-dimension}.

\medskip
\noindent\textbf{Step 2: form the rational group acting on vectors modulo sign.}

Write $X_\omega=(Y_\omega\setminus\{0\})/\{1,-1\}$, and let $D$
be the kernel of the action of $G_0$ on
$\coprod_{\omega\in\Omega}X_\omega$.
By Lemma~\ref{lem:fusion-on-supports},
\[
 D=\{g\in G_0:gY_\omega=Y_\omega\text{ and }
       g|_{Y_\omega}\in\{I_{Y_\omega},-I_{Y_\omega}\}
       \text{ for every }\omega\in\Omega\}.
\]
Define $Q=G_0/D$.
It is rational because it is a quotient of the rational group
$G_0$, and it acts faithfully on the disjoint union by construction.
Its action on the index set $\Omega$ is transitive.
Every map between two sets $X_\omega$ induced by $Q$ is induced by the restriction
of an element of $G_0\leq\Sp(V_0)$, and thus comes from an isometry.
If $G_0=\langle H_0,-I_{V_0}\rangle$ was obtained by adjoining
the scalar involution to an initial rational group $H_0$, then
$-I_{V_0}\in D$. Hence $H_0\to Q$ is surjective and
$Q\cong H_0/(H_0\cap D)$; the quotient used here is also a
quotient of the group before the scalar was adjoined.

\medskip
\noindent\textbf{Step 3: verify simultaneous fusion.}

Let $\varnothing\ne\Sigma\subseteq\Omega$.
For each $\omega\in\Sigma$, choose
$x_\omega=[y_\omega]\in X_\omega$, where
$0\ne y_\omega\in Y_\omega$, and put
\[
 v=\sum_{\omega\in\Sigma}y_\omega.
\]
Since the sum is direct, $v\ne0$, and its component
in $Y_\omega$ is nonzero exactly when $\omega\in\Sigma$.
By {\rm(SF)}, there is a $p'$-element $g\in G_0$
such that $gv=2v$.
As $g$ permutes the summands, we obtain
\[
 g\Sigma=\Sigma,\qquad
 gy_\omega=2y_{g\omega}\quad(\omega\in\Sigma)
\]
by comparing the components of $gv$ and $2v$.
The image $\bar g=gD$ is a $p'$-element of $Q$, and
\[
 \bar g\Sigma=\Sigma,\qquad
 \bar g x_\omega=2x_{\bar g \omega}\quad(\omega\in\Sigma).
\]
Thus the required fusion condition holds.

\medskip
\noindent\textbf{Step 4: identify the local socle.}

By Step~1, the stabilizer of $Y$ in $G_0$ induces
exactly $J_Y$ on $Y$.
Hence the group induced by $Q_{\omega_0}$ on
$X_{\omega_0}$ is the image of $J_Y$ acting on
nonzero vectors modulo sign.
The kernel of this action is exactly $\langle-I_Y\rangle$.
Indeed, a linear map fixing every pair $\{y,-y\}$
is either $I_Y$ or $-I_Y$, by the argument in
Lemma~\ref{lem:fusion-on-supports}, and
$-I_Y\in J_Y$ by Proposition~\ref{prop:local-component-action}.
Thus
\[
 A_{\omega_0}\cong J_Y/\langle-I_Y\rangle.
\]
By Proposition~\ref{prop:local-component-action}, we get
\[
 S\leq A_{\omega_0}\leq\Aut(S).
\]
Therefore $A_{\omega_0}$ is almost simple with socle $S$.
Since $Q$ is transitive on $\Omega$, all the local
permutation groups are isomorphic.
Together with Steps~1--3, this verifies all the hypotheses
of Proposition~\ref{prop:local-simple-recognition}.
\end{proof}

We now apply the recognition result to the local actions
constructed above. This gives a list of the possible
simple quotients of the components of $L$.

\begin{corollary}\label{cor:terminal-component-list}
Suppose that $(L,W)$ is obtained from a rational action as in
Proposition~\ref{prop:global-local-realization}.  Then, for every
component $T$ of $L$, the simple group $T/Z(T)$ is alternating or
belongs to $\mathcal E$.
\end{corollary}

\begin{proof}
Let $T$ be a component, and choose a corresponding local
space $Y$. By Proposition~\ref{prop:global-local-realization},
there is a rational quotient $Q$ of $G_0$ whose action
satisfies the hypotheses of
Proposition~\ref{prop:local-simple-recognition},
with local socle isomorphic to $T/Z(T)$.
The conclusion now follows from that proposition.
\end{proof}

\begin{remark}
The corollary lists the possible simple quotients of a
local component. Earlier, we excluded $\mathcal E$ only
for simple factors of a normal subgroup acting faithfully
and transitively on the set of blocks.
That argument does not exclude the exceptional groups
as quotients of local components.
We examine these component actions in the next section
and rule out the remaining possibilities.
\end{remark}

\section{Further structure of the components}
\label{sec:component-structure}

Throughout this section, retain the assumptions and notation of
Proposition~\ref{prop:component-decomposition}. In particular,
$F(L)=\langle z\rangle$, where $z=-I_W$, and $(L,z)$ satisfies
$\boldsymbol{(R^-)}$.
We show that no group in the Feit--Seitz list can occur
as the simple quotient of a component of $L$. The arguments combine
the conjugacy restrictions imposed by $\boldsymbol{(R^-)}$
with restrictions on irreducible symplectic modules,
using ordinary characters and their Frobenius--Schur indicators.

We start with a result showing that the rationality of the negative characters imposes a conjugacy
condition on elements of odd order.

\begin{lemma}
\label{lem:negative-power-conjugacy}
Let $J$ be a finite group, let $z\in Z(J)$ be an
involution, and suppose that $(J,z)$ satisfies
$\boldsymbol{(R^-)}$.
Then the following statements hold.
\begin{enumerate}
\item
For $x\in J$, the elements $x$ and $zx$ are conjugate
in $J$ if and only if
$\chi(x)=0$ for every $\chi\in\Irr^-(J,z)$.

\item
If $x$ and $zx$ are not conjugate, then $x$ is
conjugate to $x^a$  in $J$ for every integer $a$
with $\gcd(a,o(x))=1$.

\item
Every element of odd order in $J$ or $J/\langle z\rangle$
is conjugate to all its coprime powers in the respective group.
\end{enumerate}
\end{lemma}

\begin{proof}
Since $z$ is central, Schur's lemma implies that
in any irreducible complex representation of $J$,
the element $z$ acts as $I$ or $-I$.
Thus every $\chi\in\Irr(J)$ satisfies
\[
 \chi(zx)=\varepsilon_\chi\chi(x),
 \qquad \varepsilon_\chi\in\{1,-1\},
\]
and $\varepsilon_\chi=-1$ precisely when
$\chi\in\Irr^-(J,z)$.

For $x,y\in J$, write
\[
 S(x,y)=
 \sum_{\chi\in\Irr^-(J,z)}
       \chi(x)\overline{\chi(y)}.
\]
Separating the irreducible characters according
to the sign of $\varepsilon_\chi$, we obtain
\[
 2S(x,y)
 =
 \sum_{\chi\in\Irr(J)}\chi(x)\overline{\chi(y)}
 -
 \sum_{\chi\in\Irr(J)}\chi(zx)\overline{\chi(y)}.
\]
Hence
\begin{equation}\label{eq:negative-orthogonality}
S(x,y)
 =
 \frac{|C_J(x)|}{2}
 \left(
   \mathbf 1_{\,y\sim_J x}
   -
   \mathbf 1_{\,y\sim_J zx}
 \right).
\end{equation}
Here $\mathbf 1_{\,y\sim_J x}$ is one if $y$ and $x$
are conjugate in $J$, and zero otherwise.

Taking $y=x$ in \eqref{eq:negative-orthogonality}, we obtain
\[
 \sum_{\chi\in\Irr^-(J,z)}|\chi(x)|^2
 =
 \begin{cases}
 0,&x\sim_J zx,\\[2mm]
 |C_J(x)|/2,&x\not\sim_J zx.
 \end{cases}
\]
This proves assertion~\textup{(1)}.

Now suppose that $x\not\sim_J zx$, and let
$a$ be an integer relatively prime to $o(x)$.
For every ordinary character $\chi$, the value
$\chi(x^a)$ is obtained from $\chi(x)$ by the
cyclotomic automorphism that sends each
$o(x)$th root of unity $\zeta$ to $\zeta^a$.
Condition $\boldsymbol{(R^-)}$ therefore gives
\[
 \chi(x^a)=\chi(x)
 \qquad(\chi\in\Irr^-(J,z)).
\]
Consequently,
\[
 S(x,x^a)
 =
 \sum_{\chi\in\Irr^-(J,z)}|\chi(x)|^2
 =
 \frac{|C_J(x)|}{2}.
\]
Using \eqref{eq:negative-orthogonality}, we conclude that
\[
 \mathbf 1_{\,x^a\sim_J x}
 -
 \mathbf 1_{\,x^a\sim_J zx}
 =1.
\]
Hence $x^a\sim_J x$, proving assertion~\textup{(2)}.

If $x$ has odd order, then $zx$ has order $2o(x)$.
Thus $x$ and $zx$ cannot be conjugate.
Assertion~\textup{(2)} proves the first part of~\textup{(3)}.

Finally, let $\overline x\in J/\langle z\rangle$ have odd
order $m$. We may choose a lift $x\in J$ with $o(x)=m$.
For every integer $a$ relatively prime to $m$,
the preceding argument shows that $x$ is conjugate to $x^a$
in $J$. Passing to the quotient, we conclude that
$\overline x$ is conjugate to $\overline x^{\,a}$
in $J/\langle z\rangle$.
This completes the proof.
\end{proof}

\begin{corollary}\label{cor:odd-power-obstruction}
Let $J$ be a finite group with a central involution $z$.
If either $J$ or $J/\langle z\rangle$ contains an
 element of odd order that is not conjugate to one of
its coprime powers, then $(J,z)$ does not satisfy
$\boldsymbol{(R^-)}$.
\end{corollary}

\begin{proof}
This is the contrapositive of
Lemma~\ref{lem:negative-power-conjugacy}(3).
\end{proof}

\subsection{Alternating components}

In this subsection we consider components whose quotient
by the center is an alternating group.
We begin with a standard consequence of James's construction
of the irreducible modules of the symmetric group
and Clifford theory.

\begin{lemma}\label{lem:alternating-orthogonal}
Let $K$ be an algebraically closed field of odd
characteristic, and let $n\geq5$.
Every self-dual irreducible $KA_n$-module admits
a nondegenerate $A_n$-invariant symmetric bilinear form.
Consequently, no irreducible $KA_n$-module admits
a nondegenerate invariant alternating bilinear form.
\end{lemma}

\begin{proof}
Let $X$ be a self-dual irreducible $KA_n$-module.
Choose an irreducible $KS_n$-module $V$ whose restriction
to $A_n$ contains $X$.
Denote $p=\operatorname{char}K$. By James's construction,
$V$ is isomorphic to
\[
 D^\lambda=S^\lambda/
       \bigl(S^\lambda\cap(S^\lambda)^\perp\bigr)
\]
for a $p$-regular partition $\lambda$; see \cite{James1978}.
Here $S^\lambda$ is a Specht submodule of the tabloid permutation
module, whose tabloid basis is orthonormal. The displayed quotient
therefore inherits a nondegenerate $S_n$-invariant symmetric
bilinear form $b$.

Since $[S_n:A_n]=2$, Clifford theory shows that
$V\mathord\downarrow_{A_n}$ is either $X$ or
$X\oplus X'$, where $X'\not\cong X$ is irreducible.
In the latter case, self-duality of $X$ gives
$\Hom_{KA_n}(X,(X')^*)=0$, so $b(X,X')=0$.
Thus, in either case, the restriction of $b$ to $X$
is nondegenerate, proving the first assertion.

An irreducible module admitting a nondegenerate invariant
alternating form is self-dual. By Schur's lemma, its
invariant bilinear forms are unique up to scalar.
In odd characteristic, a nonzero symmetric form cannot
also be alternating. This proves the final assertion.
\end{proof}

For alternating components with nontrivial center, we also need a conjugacy obstruction.

\begin{lemma}
\label{lem:alternating-odd-power}
For every $n\geq5$, there exist an odd-order element
$x\in A_n$ and an integer $a$ relatively prime to $o(x)$
such that $x$ and $x^a$ are not conjugate in $A_n$.
\end{lemma}

\begin{proof}
By \cite[Theorem~1.5]{BhuniaKaurSingh2019}, a conjugacy class
with distinct odd cycle lengths is nonrational precisely when
the product of those lengths is not a square. Such a cycle type
can be chosen explicitly. If $n$ is odd and not a square, take
$(n)$; if $n$ is an odd square, take $(n-4,3,1)$.
In the latter case $n\geq9$ and the product $3(n-4)$ is
$7$ modulo $8$. If $n$ is even and $n-1$ is not a square, take
$(n-1,1)$; otherwise take $(n-3,3)$.
In this last case $n\geq10$ and $3(n-3)$ is $5$ modulo $8$.
All the listed parts are distinct and odd, and their products
are nonsquares. An element $x$ of the chosen cycle type has
odd order and a nonrational class, so $x$ is not conjugate
to $x^a$ for some integer $a$ relatively prime to $o(x)$.
\end{proof}

We next exclude the two Schur double covers of a symmetric group
using exceptional spin-character values.

\begin{lemma}\label{lem:symmetric-double-covers}
Let $n\geq5$, and let $\widetilde S_n$ be either Schur
double cover of $S_n$, with central kernel $\langle z\rangle$
of order $2$.
Then some character in $\Irr^-(\widetilde S_n,z)$ is not
rational-valued. In particular, $(\widetilde S_n,z)$ does
not satisfy $\boldsymbol{(R^-)}$.
\end{lemma}

\begin{proof}
Recall that a partition
$\lambda=(\lambda_1,\ldots,\lambda_r)$ of $n$ is
\emph{strict} if
$\lambda_1>\lambda_2>\cdots>\lambda_r>0$.
Let $\lambda=(\lambda_1,\ldots,\lambda_r)$ be a strict
partition of $n$ with $n-r$ odd, and let $\widetilde S_n$
be the Schur double cover in which transpositions
lift to elements of order $4$.
By Schur's exceptional-value formula
\cite[Theorem~2]{HuJing2014},
there is a character
$\chi_\lambda\in\Irr^-(\widetilde S_n,z)$ such that,
at a lift $g_\lambda$ of a permutation of cycle type
$\lambda$,
\begin{equation}\label{eq:spin-exceptional-square}
 \chi_\lambda(g_\lambda)^2
 =
 (-1)^{(n-r+1)/2}
 \frac{\lambda_1\cdots\lambda_r}{2}.
\end{equation}
The square remains unchanged if we replace $g_\lambda$
by $zg_\lambda$.

For the other Schur double cover, the square of the
corresponding character value is the negative of
\eqref{eq:spin-exceptional-square}.
We explain this explicitly.
In a negative representation, the matrices $T_i$
representing the standard adjacent-transposition lifts
satisfy
\[
 T_i^2=-I,\qquad
 T_iT_{i+1}T_i=T_{i+1}T_iT_{i+1},\qquad
 T_iT_j=-T_jT_i\quad (|i-j|>1).
\]
Replacing every $T_i$ by $\mathrm{i}T_i$, where
$\mathrm{i}^2=-1$, changes the first relation to
$(\mathrm{i}T_i)^2=I$ and preserves the other two
relations. This gives a negative representation of the
other cover. Irreducibility is preserved, since a subspace
is invariant under all $T_i$ if and only if it is invariant
under all $\mathrm{i}T_i$.

A permutation of cycle type $\lambda$ has sign
$(-1)^{n-r}=-1$.
Thus a word in adjacent transpositions representing it
has odd length. The corresponding character value in the
other cover is therefore multiplied by $\pm\mathrm{i}$,
and its square changes sign.

We now choose strict partitions for which these squares
cannot all be squares of rational numbers.

Suppose first that $n\geq6$ is even. Take
$\lambda=(n)$ and $\mu=(n-3,2,1)$.
Both partitions are strict, and their lengths are $1$
and $3$, respectively. Formula
\eqref{eq:spin-exceptional-square} gives
\[
 \chi_\lambda(g_\lambda)^2
   =(-1)^{n/2}\frac n2,
 \qquad
 \chi_\mu(g_\mu)^2
   =(-1)^{(n-2)/2}(n-3).
\]
These are nonzero real numbers with opposite signs.
For the other cover, both squares have the opposite signs.
Consequently, in either cover one of these character
values has negative square and hence is not rational.

Suppose next that $n\geq11$ is odd. Take
$\lambda=(n-1,1)$ and $\mu=(n-6,3,2,1)$.
These partitions are strict and have lengths $2$ and $4$.
The corresponding squares are
\[
 \chi_\lambda(g_\lambda)^2
   =(-1)^{(n-1)/2}\frac{n-1}{2},
 \qquad
 \chi_\mu(g_\mu)^2
   =(-1)^{(n-3)/2}3(n-6).
\]
Again, they are nonzero and have opposite signs.
Thus each cover has a negative irreducible character
with a nonrational value.

It remains to consider $n=5,7,9$. We take
$\lambda=(4,1)$, $(6,1)$, $(7,2)$, respectively.
For the cover with transposition lifts of order $4$,
formula \eqref{eq:spin-exceptional-square} yields
\[
\begin{array}{c|c|c}
 n & \lambda & \chi_\lambda(g_\lambda)^2\\ \hline
 5 & (4,1) & 2\\
 7 & (6,1) & -3\\
 9 & (7,2) & 7
\end{array}
\]
For the other cover, these squares are $-2,3,-7$.
None of the numbers $\pm2,\pm3,\pm7$ is the square of
a rational number. Hence the displayed character value
is nonrational in each case.

Therefore neither Schur double cover satisfies
$\boldsymbol{(R^-)}$.
\end{proof}

\begin{lemma}[Schur \cite{Schur1911}]
\label{lem:symmetric-extension-identification}
Let $n\geq5$, and suppose that
\[
 1\longrightarrow \langle z\rangle
 \longrightarrow L
 \overset{\pi}{\longrightarrow} S_n
 \longrightarrow 1
\]
is a central extension with $o(z)=2$.
If $T=\pi^{-1}(A_n)$ is perfect, then $L$ is isomorphic to one of
the two Schur double covers of $S_n$.
\end{lemma}

\begin{proof}
Since $T$ is perfect,
\[
 \langle z\rangle\leq T=[T,T]\leq[L,L].
\]
Thus this is a stem extension: its central kernel is contained
in $[L,L]$. Among the central extensions of $S_n$ by $C_2$,
\cite[Theorem~2.6]{BazlovBerenstein2013} gives exactly two stem
extension classes, namely the Schur double covers.
\end{proof}

For $A_6$, the larger outer automorphism group requires a separate
argument.

\begin{lemma}\label{lem:A6-extensions-excluded}
Let $L$ be a finite group with a normal subgroup
$T\cong\SL_2(9)$, and write $Z(T)=\langle z\rangle$.
Suppose that
$C_L(T)=\langle z\rangle$. Then $(L,z)$ does not satisfy $\boldsymbol{(R^-)}$.
\end{lemma}

\begin{proof}
Suppose, for a contradiction, that $(L,z)$ satisfies
$\boldsymbol{(R^-)}$.
Since $Z(T)$ is characteristic in $T$ and has order $2$,
the element $z$ is central in $L$.
The action of $L$ by conjugation on $T$ induces an action
on $T/Z(T)\cong A_6$ with kernel $\langle z\rangle$.
We therefore identify
\[
 A_6\cong T/\langle z\rangle
 \leq L/\langle z\rangle
 \leq\Aut(A_6).
\]

We use the standard description
$\Aut(A_6)\cong\PGL_2(9)\rtimes\langle\phi\rangle$,
where $\phi$ is induced by the field automorphism
$a\mapsto a^3$ of $\F_9$.
Choose a generator $\omega$ of $\F_9^\times$, and write
\[
 D=\begin{pmatrix}\omega&0\\0&1\end{pmatrix}.
\]
On $T=\SL_2(9)$, define
\[
 \delta(t)=DtD^{-1},
 \qquad
 \phi(t)=(t_{ij}^{\,3}).
\]
Their outer classes generate
$\Out(A_6)\cong C_2\times C_2$.
We refer to these as the diagonal and field classes,
respectively.

We first show that $L/T$ cannot contain the mixed
outer class $\delta\phi$.
Set
\[
 \alpha=\delta\phi,\qquad
 U=\begin{pmatrix}\omega^2&0\\0&\omega^{-2}\end{pmatrix}.
\]
Since $\omega$ has order $8$, direct calculation gives
\[
 \alpha^2=\operatorname{Int}(U),\qquad
 U^2=z,\qquad
 \alpha(U)=zU.
\]
Indeed,
\[
 D\phi(D)=\operatorname{diag}(\omega^4,1)
         =\omega^2U,
\]
so conjugation by $D\phi(D)$ equals conjugation by $U$.
Also $\phi(U)=U^{-1}=zU$, and $D$ commutes with $U$.

Suppose that the mixed outer class occurs in $L/T$.
Multiplying a representative by a suitable element
of $T$, we obtain $x\in L$ whose conjugation action
on $T$ is exactly $\alpha$.
Then $x^2$ and $U$ induce the same inner automorphism
of $T$, whence
$x^2U^{-1}\in C_L(T)=\langle z\rangle$.
Thus $x^2=U$ or $x^2=zU$.
But the conjugation by $x$ interchanges $U$ and $zU$,
whereas every element commutes with its own square.
This is a contradiction.

Consequently, $L/T$ is trivial, is generated by the
diagonal class, or is generated by the field class.
We exclude these three possibilities one by one.

If $L=T$, then $L/\langle z\rangle\cong A_6$. By
Lemma~\ref{lem:alternating-odd-power},  $L/\langle z\rangle$ has an element of odd order that is not conjugate
to one of its coprime powers.
This contradicts
Corollary~\ref{cor:odd-power-obstruction}.

Suppose that $L/T$ is generated by the diagonal class.
Then $L/\langle z\rangle\cong\PGL_2(9)$.
We shall find an element of $L/\langle z\rangle$ of order $5$
that is not conjugate to its square.

Choose $\xi\in\F_{81}^{\times}$ of order $5$.
Regard multiplication by $\xi$ as an $\F_9$-linear map
on the two-dimensional space $\F_{81}$, and denote
its matrix by $h$.
Its determinant is
\[
 N_{\F_{81}/\F_9}(\xi)=\xi^{10}=1,
\]
so $h\in\SL_2(9)$.
Over $\F_{81}$ its eigenvalues are
$\xi$ and $\xi^9=\xi^{-1}$.
The projective image $\bar h$ therefore has order $5$.

For a semisimple element of $\operatorname{PGL}_2$ with
two distinct eigenvalues, their ratio is determined up to
inversion and is preserved under conjugacy. Indeed, the
ratio does not depend on the scalar multiple chosen for
the matrix representative, and it is inverted if we
interchange the two eigenvalues.
For $\bar h$ the ratio is $\xi^2$, whereas for
$\bar h^{\,2}$ it is $\xi^4$.
Since
\[
 \xi^4\notin\{\xi^2,\xi^{-2}\},
\]
the elements $\bar h$ and $\bar h^{\,2}$ are not
conjugate in $\PGL_2(9)$.
Again Corollary~\ref{cor:odd-power-obstruction}
gives a contradiction.

Finally, suppose that $L/T$ is generated by the field
class. In this case, $\PSL_2(9)\rtimes\langle\phi\rangle\cong S_6$, and so $L/\langle z\rangle\cong S_6$.
The inverse image of $A_6$ is the perfect group $T$.
By Lemma~\ref{lem:symmetric-extension-identification},
$L$ is isomorphic to one of the two Schur double covers of $S_6$.
Lemma~\ref{lem:symmetric-double-covers} then contradicts
$\boldsymbol{(R^-)}$.

This completes the proof.
\end{proof}

We can now exclude all alternating simple quotients.

\begin{proposition}\label{prop:no-alternating-components}
No component of $L$ has an alternating simple quotient.
\end{proposition}

\begin{proof}
Suppose that a component $T$ satisfies $T/Z(T)\cong A_n$, where $n\geq5$.
Write $K=\overline{\F}_p$. By
Lemmas~\ref{lem:layer-constituents} and~\ref{lem:single-component-action},
a simple constituent associated with $T$ is a nonzero simple
symplectic $KT$-module. If $Z(T)=1$, then $T\cong A_n$,
contrary to Lemma~\ref{lem:alternating-orthogonal}.
Thus $Z(T)\ne1$, and Lemma~\ref{lem:unique-centered-component} shows that
\[
 Z(T)=\langle z\rangle,\qquad E(L)=T,\qquad
 C_L(T)=\langle z\rangle,\qquad
 A_n\leq L/\langle z\rangle\leq\Aut(A_n).
\]

Suppose first that $n\ne6$. Then $\Aut(A_n)\cong S_n$, so
$L/\langle z\rangle$ is $A_n$ or $S_n$.
The first possibility contradicts
Lemma~\ref{lem:alternating-odd-power} and
Corollary~\ref{cor:odd-power-obstruction}.
In the second case, the inverse image of $A_n$ is $T$, since
it contains $T$ and both groups have order $2|A_n|$.
Since $T$ is perfect, it follows from Lemma~\ref{lem:symmetric-extension-identification}
that $L$ is a Schur double cover of $S_n$.
By Lemma~\ref{lem:symmetric-double-covers},  $(L,z)$
does not satisfy $\boldsymbol{(R^-)}$, which is a contradiction again.

It remains to consider $n=6$.
The Schur multiplier of $A_6$ is cyclic of order $6$.
Since its perfect central extension with kernel of order $2$ is unique up to isomorphism,
we have
$T\cong 2.A_6\cong\SL_2(9)$.
Lemma~\ref{lem:A6-extensions-excluded} now applies and yields
a contradiction to $\boldsymbol{(R^-)}$.
\end{proof}

\subsection{Ordinary characters and modular form types}

For an ordinary irreducible character $\chi$ of a finite
group $G$, the Frobenius--Schur indicator of $\chi$ is
\[
 \nu(\chi)=\frac{1}{|G|}\sum_{g\in G}\chi(g^2).
\]
Its value is $0$, $1$, or $-1$, according as the
corresponding complex representation is not self-dual,
has symmetric type, or has alternating type.

Let $k$ be an algebraically closed field of odd
characteristic $p$.
If $V$ is a self-dual simple $kG$-module, write
\[
 \varepsilon(V)=
 \begin{cases}
  1,&\text{if $V$ has symmetric type},\\
 -1,&\text{if $V$ has alternating type}.
 \end{cases}
\]
By Schur's lemma, the invariant nondegenerate bilinear
form on $V$ is unique up to a nonzero scalar, so these
two possibilities are mutually exclusive.

%To exclude the remaining groups, we compare modular form types with
%ordinary character data. We use the following theorem of Thompson
%\cite[p.~227]{ThompsonForms} and Willems
%\cite[Theorem~2.8]{Willems1991}. %Its precise formulation, including
%%odd multiplicity and equality of form types, is also stated in
%%\cite[Section~5, p.~29]{HissMalle2001}.

%\begin{theorem}[Thompson--Willems]
%\label{thm:ordinary-modular-form-types}
%Let $G$ be a finite group, let $k$ be an algebraically
%closed field of odd characteristic $p$, and let $V$
%be a self-dual simple $kG$-module with Brauer character
%$\varphi$.
%Then there exists a real-valued ordinary irreducible
%character $\chi$ of $G$ such that
%\[
% d_{\chi\varphi}\equiv1\pmod2,
% \qquad
% \nu(\chi)=\varepsilon(V),
%\]
%where $d_{\chi\varphi}$ is the decomposition number of $\varphi$ in the restriction of $\chi$ to $p$-regular elements of $G$.
%\end{theorem}
%
%The theorem applies to arbitrary finite groups in odd characteristic;
%no assumption on the defect groups is required.

\begin{lemma}\label{lem:ordinary-symplectic-witness}
Let $G$ be a finite group, and let $k$ be an
algebraically closed field of odd characteristic.
Suppose that $V$ is a simple symplectic $kG$-module,
with Brauer character $\varphi$.
Then there exists $\chi\in\Irr(G)$ such that
\[
 \nu(\chi)=-1,
 \qquad
 d_{\chi\varphi}\ \text{is odd}.
\]
If, in addition, $z\in Z(G)$ is an involution acting
as $-I_V$, then $\chi$ can be chosen in
$\Irr^-(G,z)$.
\end{lemma}

\begin{proof}
Since $V$ is symplectic, it is self-dual and
$\varepsilon(V)=-1$. 
By  \cite[Theorem~2.8]{Willems1991},
there exists
an ordinary irreducible character $\chi$ with
$\nu(\chi)=-1$ and $d_{\chi\varphi}\equiv1\pmod2$ (see also \cite[p.~227]{ThompsonForms}).

Now suppose that $z$ is a
central involution acting on $V$ as $-I_V$.
By Schur's lemma, $z$ acts as $\eta I$, where
$\eta\in\{1,-1\}$, in an ordinary irreducible representation
affording $\chi$. It therefore acts by the same scalar
on an invariant lattice and, after reduction modulo $p$,
on every composition factor.

Since $d_{\chi\varphi}>0$, the module $V$ is one of
these composition factors.
Thus the reduction of $\eta$ equals $-1$.
The characteristic is odd, so $1$ and $-1$ remain
distinct after reduction. Consequently $\eta=-1$,
and hence
$\chi(z)=-\chi(1)$.
Therefore $\chi\in\Irr^-(G,z)$.
\end{proof}

\begin{corollary}\label{cor:no-quaternionic-no-symplectic}
Let $G$ be a finite group.
\begin{enumerate}
\item If no $\chi\in\Irr(G)$ has $\nu(\chi)=-1$, then $G$ has no
absolutely irreducible symplectic module in odd characteristic.
\item If $z\in Z(G)$ is an involution and no
$\chi\in\Irr^-(G,z)$ has $\nu(\chi)=-1$, then $G$ has no such module
on which $z$ acts as $-I$.
\end{enumerate}
\end{corollary}

\begin{proof}
The absolute irreducibility of a module and the nondegenerate
alternating form are preserved under scalar extension to an algebraic closure.
Both assertions now follow from
Lemma~\ref{lem:ordinary-symplectic-witness}.
\end{proof}

\subsection{The exceptional simple quotients}

In this subsection, we exclude the remaining possibilities for the
simple quotients of components.
We first consider $\Sp_6(2)$, ${\rm O}_8^+(2)$,
and $\PSL_3(4)$.

\begin{lemma}
\label{lem:three-families-table-data}
Let $\mathcal S_0=
 \{\Sp_6(2), {\rm O}_8^+(2),
                       \operatorname{PSL}_3(4)\}$.
The following statements hold.
\begin{enumerate}
\item
If $T$ is quasisimple, $T/Z(T)\cong S\in\mathcal S_0$,
and $|Z(T)|\leq2$, then $T$ is isomorphic to $S$ or
to its double cover $2.S$.

\item
None of these six groups has an ordinary irreducible
character of Frobenius--Schur indicator $-1$.
%More precisely, their indicator distributions are
%\[
%\begin{array}{c|rrr}
% G&
% \#\{\chi:\nu(\chi)=1\}&
% \#\{\chi:\nu(\chi)=0\}&
% \#\{\chi:\nu(\chi)=-1\}\\ \hline
% \Sp_6(2)                         &30&0&0\\
% 2.\Sp_6(2)                      &41&2&0\\
% \operatorname{P}\Omega_8^+(2)    &53&0&0\\
% 2.\operatorname{P}\Omega_8^+(2)  &83&0&0\\
% \operatorname{PSL}_3(4)          &8&2&0\\
% 2.\operatorname{PSL}_3(4)        &14&4&0
%\end{array}
%\]
%Here $\chi$ ranges over $\Irr(G)$.
\end{enumerate}
\end{lemma}

\begin{proof}
The Schur multipliers of the three simple groups are,
respectively,
\[
 C_2,\qquad C_2\times C_2,\qquad
 C_3\times C_4\times C_4.
\]
A quasisimple central extension is a quotient of the
universal central extension.
For $\Sp_6(2)$ there is only one quotient with central
kernel of order $2$.
For ${\rm O}_8^+(2)$, triality permutes
the three possible kernels giving double covers.
For $\operatorname{PSL}_3(4)$, automorphisms likewise
permute the three possible kernels giving double
covers.
Thus, in each family, the double covers are isomorphic.
This proves (1).

Assertion~(2) follows from the ordinary character tables in
\cite{Atlas1985,CTblLib}. Appendix~\ref{app:character-data} records
the table identifiers, indicator distributions, and the exact
class-sum formula used to recover these data.
\end{proof}

\begin{proposition}\label{prop:three-exceptional-families-excluded}
No component of $L$ has a simple quotient isomorphic to $\Sp_6(2)$,
${\rm O}_8^+(2)$, or $\PSL_3(4)$.
\end{proposition}

\begin{proof}
Suppose that $T$ is such a component. Since distinct components
commute,
\[
 Z(T)\leq Z(E(L))\leq F(L)=\langle z\rangle.
\]
In particular, $|Z(T)|\leq2$. Hence
$T$ is isomorphic to one of the six groups listed in Lemma~\ref{lem:three-families-table-data}.
None has an ordinary
irreducible character of Frobenius--Schur indicator $-1$.
By Corollary~\ref{cor:no-quaternionic-no-symplectic}, none has
a simple symplectic module over $K=\overline{\F}_p$.

On the other hand, $T$ acts nontrivially on $W$,
since the action of $L$ is faithful.
By Lemma~\ref{lem:layer-constituents}, we may therefore
choose an irreducible $KE(L)$-submodule $V$ of
$K\otimes_{\F_p}W$ on which $T$ acts nontrivially.
The restriction of the symplectic form to $V$ is
nondegenerate. By Lemma~\ref{lem:single-component-action},
all other components act trivially on $V$.
Thus every $T$-invariant subspace of $V$ is also
$E(L)$-invariant, so $V$ is irreducible as a $KT$-module.
Therefore $V$ is a simple symplectic $KT$-module,
a contradiction.
\end{proof}

We now consider $\PSp_4(3)$ and its double cover.
The following lemma gives the dimension bound needed
to exclude this case.

\begin{lemma}\label{lem:psp43-symplectic-degrees}
Let $k$ be an algebraically closed field of
characteristic $p\in\{7,11\}$. Then:
\begin{enumerate}
\item
$\operatorname{PSp}_4(3)$ has no simple symplectic
$k$-module.
\item
Every simple symplectic $k\operatorname{Sp}_4(3)$-module
has dimension in $\{20,60,64,80\}$, and hence
has dimension at least $20$.
\end{enumerate}
\end{lemma}

\begin{proof}
The ordinary character tables of
$\operatorname{PSp}_4(3)$ and $\operatorname{Sp}_4(3)$
are available in the GAP Character Table Library
under the identifiers \texttt{U4(2)} and
\texttt{2.U4(2)}, respectively; see
\cite{Atlas1985,CTblLib}.
The Frobenius--Schur indicators in these tables show that
\[
 \{\chi(1):\chi\in\Irr(\operatorname{PSp}_4(3)),
                         \ \nu(\chi)=-1\}
 =\varnothing
\]
and
\[
 \{\chi(1):\chi\in\Irr(\operatorname{Sp}_4(3)),
                         \ \nu(\chi)=-1\}
 =\{20,60,64,80\}.
\]
Assertion~(1) therefore follows from
Corollary~\ref{cor:no-quaternionic-no-symplectic}.

For~(2), let $V$ be a simple symplectic
$k\operatorname{Sp}_4(3)$-module with Brauer character
$\varphi$. By Lemma~\ref{lem:ordinary-symplectic-witness},
there is an ordinary irreducible character $\chi$
such that $\nu(\chi)=-1$ and $d_{\chi\varphi}>0$.
Since
\[
 |\operatorname{Sp}_4(3)|=51840=2^7\cdot3^4\cdot5,
\]
the prime $p$ does not divide the group order.
Thus the decomposition matrix is the identity matrix
after a suitable ordering of its rows and columns.
It follows that $\varphi(1)=\chi(1)$, and hence
\[
 \dim_k V\in\{20,60,64,80\}.
\]
This proves~(2).
\end{proof}

\begin{proposition}\label{prop:psp43-component-excluded}
No component of $L$ has a simple quotient isomorphic to
$\operatorname{PSp}_4(3)$.
\end{proposition}

\begin{proof}
Suppose that a component $T$ of $L$ satisfies
$T/Z(T)\cong S:=\operatorname{PSp}_4(3)$.
Write $K=\overline{\F}_p$.

Since $T$ acts nontrivially on $W$, we may choose an
irreducible $KE(L)$-submodule $V$ of $K\otimes_{\F_p}W$
on which $T$ acts nontrivially.
By Lemma~\ref{lem:layer-constituents}, $V$ is symplectic.
By Lemma~\ref{lem:single-component-action}, all other
components act trivially on $V$. Thus every
$T$-invariant subspace of $V$ is also $E(L)$-invariant,
so $V$ is an irreducible symplectic $KT$-module.

If $Z(T)=1$, then $T\cong S$, contrary to
Lemma~\ref{lem:psp43-symplectic-degrees}.
Therefore $Z(T)\ne1$. By Lemma~\ref{lem:unique-centered-component}, we have
\[
 Z(T)=\langle z\rangle,\qquad
 E(L)=T,\qquad C_L(T)=\langle z\rangle,
\]
and
$S\leq L/\langle z\rangle\leq\Aut(S)$.
The Schur multiplier of $S$ has order $2$, so $T\cong\operatorname{Sp}_4(3)$.

Write $d=\dim_{\F_p}W$.
Choose a simple symplectic $KT$-constituent $V$ of
$K\otimes_{\F_p}W$.
Lemma~\ref{lem:psp43-symplectic-degrees} yields
\[
 d\geq\dim_K V\geq20.
\]
On the other hand, the projective counting bound
from Lemma~\ref{lem:projective-count} gives
\[
 p^{d/2}+1
 \leq |L/\langle z\rangle|
 \leq|\Aut(S)|
 =51840.
\]
Since $p\geq7$ and $d\geq20$, this implies
$7^{10}+1\leq51840$,
which is impossible.
Thus no such component exists.
\end{proof}

Finally, we exclude the case of $\PSU_4(3)$.

\begin{lemma}\label{lem:psu43-selfdual-degree}
Let $T$ be a quasisimple group such that
\[
 T/Z(T)\cong\operatorname{PSU}_4(3),
 \qquad |Z(T)|\leq2.
\]
Let $k$ be an algebraically closed field of
characteristic $p\notin\{2,3\}$.
Then every faithful finite-dimensional self-dual
$kT$-module has dimension at least $18$.
\end{lemma}

\begin{proof}
The group $\operatorname{PSU}_4(3)$ contains an
extraspecial subgroup $R_0$ of order $3^5$;
see the maximal subgroup of structure $3_+^{1+4}.2S_4$
listed in \cite{Atlas1985,CTblLib}.
Since $|Z(T)|\leq2$, a Sylow $3$-subgroup $R$ of the
inverse image of $R_0$ in $T$ maps isomorphically onto
$R_0$.

Let $V$ be a faithful finite-dimensional self-dual
$kT$-module. Since $p\ne3$, its restriction to $R$
is semisimple. The extraspecial group $R$ has exactly
two nonlinear irreducible $k$-modules. Both have dimension $9$, and they are dual to one another
because their central characters are the two nontrivial
characters of $Z(R)$.

Since $V$ is faithful, $Z(R)$ acts nontrivially on $V$.
Thus at least one of these nonlinear modules occurs
in $V\mathord\downarrow_R$. Since the two modules are dual to one another and
$V$ is self-dual, both occur in $V\mathord\downarrow_R$.
 Hence
\[
 \dim_k V\geq9+9=18.
\]
\end{proof}

\begin{proposition}\label{prop:psu43-component-excluded}
No component of $L$ has a simple quotient isomorphic to $\PSU_4(3)$.
\end{proposition}

\begin{proof}
Suppose that $T/Z(T)\cong S:=\PSU_4(3)$.
As above, $Z(T)\leq F(L)=\langle z\rangle$, so $|Z(T)|\leq2$.
Let $Y$ be the local space constructed in
Proposition~\ref{prop:local-component-action}, and write
$d=\dim_{\F_p}Y$.
The group $T$ acts faithfully on $Y$, and the restriction
of the symplectic form is nondegenerate and $T$-invariant.
Thus $K\otimes_{\F_p}Y$, where $K=\overline{\F}_p$, is a
faithful self-dual $KT$-module. By Lemma~\ref{lem:psu43-selfdual-degree}, we see that $d\geq18$.

Notice that $|S|=3\,265\,920$ and $|\Out(S)|=8$ according to \cite{Atlas1985}.
By
Corollary~\ref{cor:local-dimension-bound}, we obtain
\[
 p^{d/2}+1\leq|\Aut(S)|=26\,127\,360.
\]
But $p\geq7$ and $d\geq18$, so
\[
 p^{d/2}+1\geq7^9+1=40\,353\,608>26\,127\,360,
\]
a contradiction.
\end{proof}

\section{Proofs of the main theorems}
\label{sec:terminal-contradiction}

We now combine the preceding reductions and exclusions to prove
the main theorems.

\begin{theorem}\label{thm:no-rational-origin-admissible-pair}
Let $p\in\{7,11\}$. Then there is no admissible pair $(L_0,W_0)$ over $\F_p$
such that $L_0$ is rational.
\end{theorem}

\begin{proof}
Suppose that such a pair exists.
Then $W_0\ne0$, and $L_0$ acts faithfully and
absolutely irreducibly on $W_0$, preserving a
nondegenerate alternating form.

We may first arrange that $-I_{W_0}\in L_0$.
Indeed, replace $L_0$ by
\[
 \widehat L_0=\langle L_0,-I_{W_0}\rangle.
\]
If $-I_{W_0}\notin L_0$, then
$\widehat L_0\cong L_0\times C_2$.
Consequently $\widehat L_0$ is rational in either
case.

The enlarged action remains absolutely irreducible
and symplectic, and it still satisfies {\rm(SF)}.
It also satisfies {\rm(R)}.
To see this directly, every element of $\widehat L_0$
is of the form $\varepsilon g$, with
$\varepsilon\in\{I,-I\}$ and $g\in L_0$.
If $\varepsilon g$ is a $p'$-element, then so is $g$,
and
\[
 \BrTr(\varepsilon g)=\varepsilon\BrTr(g)\in\mathbb Z,
\]
where the scalar $\varepsilon$ is identified with
$1$ or $-1$.
Thus the enlarged pair is admissible.
Replacing $L_0$ by $\widehat L_0$, we may assume
$z_0=-I_{W_0}\in L_0$.
Rationality of $L_0$ implies that $(L_0,z_0)$
satisfies $\boldsymbol{(R^-)}$.

Applying Proposition~\ref{prop:2-terminal-refinement}, we obtain a $2$-terminal admissible pair $(L,W)$
such that
\[
 z=-I_W\in L
 \quad\text{and}\quad
 (L,z)\text{ satisfies }\boldsymbol{(R^-)}.
\]
By Proposition~\ref{prop:extraspecial-excluded}, we have
$F(L)=O_2(L)=\langle z\rangle$ and $E(L)\ne1$.

Since this terminal pair was obtained by refinement
from the rational group $L_0$,
Proposition~\ref{prop:global-local-realization} and
Corollary~\ref{cor:terminal-component-list} apply.
Hence, for every component $T$ of $L$, its simple
quotient $T/Z(T)$ is either an alternating group
$A_n$, with $n\geq5$, or one of the groups $\operatorname{PSp}_4(3),
 \Sp_6(2),  {\rm O}_8^+(2),
 \operatorname{PSL}_3(4)$ or $\operatorname{PSU}_4(3)$.
However, Proposition~\ref{prop:no-alternating-components}
excludes the alternating quotients,
Proposition~\ref{prop:three-exceptional-families-excluded}
excludes
$\Sp_6(2),~{\rm O}_8^+(2),~\operatorname{PSL}_3(4)$, and
Propositions~\ref{prop:psp43-component-excluded}
and~\ref{prop:psu43-component-excluded}
exclude the remaining two possibilities.
Thus $L$ has no components, so $E(L)=1$.
This contradicts $E(L)\ne1$.
\end{proof}

\begin{proof}[Proof of Theorem~\ref{thm:linear}]
Let a finite rational group $G$ act on a nonzero finite-dimensional
$\F_p$-space $V$ with the eigenvector property.
Suppose that $p=7$ or $11$. Choose a nonzero irreducible
$G$-submodule $U$ of $V$, and let $H$ be the faithful image
of $G$ on $U$. The group $H$ is rational, being a quotient
of $G$, and the action on $U$ has the eigenvector property.
Taking the $p'$-part of each scalar-fusing element, as in
Lemma~\ref{lem:scalar-fusion}, gives {\rm(SF)}.

Rationality of $H$ gives integer-valued irreducible Brauer
characters, by the proof of Lemma~\ref{lem:rational-brauer}.
The descent and self-duality arguments in
Lemma~\ref{lem:symplectic} therefore show that $U$ is absolutely
irreducible and has a nondegenerate invariant bilinear form.
That proof excludes a symmetric form using $3\ne0$ in $\F_p$,
so the form is alternating for $p=7$ or 11.
Thus $(H,U)$ is an admissible pair with $H$ rational,
contrary to Theorem~\ref{thm:no-rational-origin-admissible-pair}.
Hence $p\in\{2,3,5\}$.
\end{proof}

\begin{proof}[Proof of Theorem~\ref{thm:main}]
If the assertion fails, the least-counterexample reduction in
the introduction gives a rational group $H$ acting on a nonzero
$\F_p$-space $V$, where $p\in\{7,11\}$. By Lemma~\ref{lem:scalar-fusion},
the action of $H$ on $V$ has the eigenvector property,
contrary to Theorem~\ref{thm:linear}.
The examples $C_2$, $S_3$, and $\F_5^2\rtimes Q_8$ given in
the introduction show that each of the three primes occurs.
\end{proof}

\section*{Acknowledgements}
The first author gratefully acknowledges financial support by the National Natural Science Foundation of China
(11761034) and the Natural Science Foundation of Jiangxi Province (20232BAB202012),
and the second author was supported by the NSFC (12671027).

The authors thank AI models for valuable assistance
in exploring proof strategies and carrying out calculations.
Nevertheless, the authors take full responsibility for all arguments, computations, and the final exposition of the paper.

\appendix
\section{Ordinary character-table data}\label{app:character-data}

We record the ordinary character data used in
Lemmas~\ref{lem:three-families-table-data} and
\ref{lem:psp43-symplectic-degrees}. The identifiers refer to
the GAP Character Table Library, version~1.3.11 \cite{CTblLib}.
Write $n_\epsilon(G)=|\{\chi\in\Irr(G):\nu(\chi)=\epsilon\}|$.
Then
\[
\begin{array}{l|l|rrr}
G&\text{identifier}&n_{1}(G)&n_{0}(G)&n_{-1}(G)\\ \hline
\Sp_6(2)&\texttt{S6(2)}&30&0&0\\
2.\Sp_6(2)&\texttt{2.S6(2)}&41&2&0\\
{\rm O}_8^+(2)&\texttt{O8+(2)}&53&0&0\\
2.{\rm O}_8^+(2)&\texttt{2.O8+(2)}&83&0&0\\
\PSL_3(4)&\texttt{L3(4)}&8&2&0\\
2.\PSL_3(4)&\texttt{2.L3(4)}&14&4&0\\
\PSp_4(3)&\texttt{U4(2)}&10&10&0\\
\Sp_4(3)&\texttt{2.U4(2)}&10&20&4
\end{array}
\]
For the last group, the four characters with indicator $-1$ have
degrees $20$, $60$, $64$, and $80$.

For class representatives $c_j$, let $s(j)$ be the class index of
$c_j^2$. The indicators can be computed entirely in the cyclotomic
field containing the table entries by
\[
 \nu(\chi)=\sum_j\frac{\chi(c_{s(j)})}{|C_G(c_j)|}.
\]
Thus only the ordinary table, class centralizer orders, and square
power map are required. This formula also avoids reducing ordinary
indicators modulo $p$. The following GAP commands reproduce the
displayed distributions and the negative-indicator degree set.

\begingroup\small
\begin{verbatim}
LoadPackage("ctbllib");;
ids := ["S6(2)", "2.S6(2)", "O8+(2)", "2.O8+(2)",
        "L3(4)", "2.L3(4)", "U4(2)", "2.U4(2)"];;
for id in ids do
  t := CharacterTable(id);;
  ir := Irr(t);;
  sq := PowerMap(t, 2);;
  cen := SizesCentralizers(t);;
  fs := List(ir, x -> Sum([1..Length(cen)],
                j -> x[sq[j]] / cen[j]));;
  Print(id, " ", List([1,0,-1], a -> Number(fs, b -> b=a)),
        " ", Set(List(Filtered([1..Length(ir)],
                    j -> fs[j]=-1), j -> ir[j][1])), "\n");
od;
\end{verbatim}
\endgroup

\end{document}